\documentclass[10pt]{article}
\usepackage{amsthm,amsfonts,amsmath,amscd,amssymb,latexsym,epsfig}
\usepackage[all]{xy}
\title{The moduli space of regular stable maps\footnote{The authors 
would like to thank the referee for his/her diligent work. 
We are grateful for the careful attention to detail in the report.}}

\author{Joel W. Robbin\\
        University of Wisconsin
        \and
          Yongbin Ruan\\
        University of Wisconsin
        \and
        Dietmar~A.~Salamon\\
        ETH-Z\"urich}

\date{23 August 2007}

\newtheorem{PARA}{}[section]
\newtheorem{theorem}[PARA]{Theorem}
\newtheorem{corollary}[PARA]{Corollary}
\newtheorem{lemma}[PARA]{Lemma}

\newtheorem{definition}[PARA]{Definition}
\newtheorem{remark}[PARA]{Remark}
\newtheorem{example}[PARA]{Example}

\newcommand{\para}{\begin{PARA}\rm}
\newcommand{\arap}{\end{PARA}\rm}
\newcommand{\dfn}{\begin{definition}\rm}
\newcommand{\nfd}{\end{definition}\rm}
\newcommand{\rmk}{\begin{remark}\rm}
\newcommand{\kmr}{\end{remark}\rm}
\newcommand{\xmpl}{\begin{example}\rm}
\newcommand{\lpmx}{\end{example}\rm}
\newcommand{\jdef}[1]{{\bf #1}}

\newcommand{\Times}[2]{{}_{#1}\!\!\times_{#2}}

\newcommand{\cB}{\mathcal{B}}

\newcommand{\cD}{\mathcal{D}}
\newcommand{\cF}{\mathcal{F}}
\newcommand{\cE}{\mathcal{E}}
\newcommand{\cG}{\mathcal{G}}
\newcommand{\cH}{\mathcal{H}}
\newcommand{\cI}{\mathcal{I}}
\newcommand{\cJ}{\mathcal{J}}

\newcommand{\cL}{\mathcal{L}}
\newcommand{\cM}{\mathcal{M}}
\newcommand{\cMbar}{\overline{\mathcal{M}}}
\newcommand{\cN}{\mathcal{N}}
\newcommand{\cP}{\mathcal{P}}
\newcommand{\cQ}{\mathcal{Q}}

\newcommand{\cS}{\mathcal{S}}
\newcommand{\cT}{\mathcal{T}}
\newcommand{\cU}{\mathcal{U}}
\newcommand{\cV}{\mathcal{V}}
\newcommand{\cW}{\mathcal{W}}
\newcommand{\cX}{\mathcal{X}}
\newcommand{\cY}{\mathcal{Y}}
\newcommand{\cZ}{\mathcal{Z}}

\newcommand{\Rectangle}[8]{
\begin{array}{ccc}
{#1} & \mapright{#2} & {#3}\\
\mapdown{#4} && \mapdown{#5} \\
{#6} & \mapright{#7} & {#8}
\end{array}
}
\newcommand{\mapdown}[1]{\Big\downarrow
\rlap{$\vcenter{\hbox{$\scriptstyle#1$}}$}}
\newcommand{\mapright}[1]{\smash{
\mathop{\longrightarrow}\limits^{#1}}}
\newcommand{\dbar}{{\overline\partial}}
\newcommand{\one}{{{\mathchoice \mathrm{ 1\mskip-4mu l} \mathrm{ 1\mskip-4mu l}
\mathrm{ 1\mskip-4.5mu l} \mathrm{ 1\mskip-5mu l}}}}
\newcommand{\A}{{\mathbb{A}}}

\newcommand{\C}{{\mathbb{C}}}
\newcommand{\D}{{\mathbb{D}}}

\renewcommand{\H}{{\mathbb{H}}}

\newcommand{\R}{{\mathbb{R}}}
\newcommand{\T}{{\mathbb{T}}}

\newcommand{\Z}{{\mathbb{Z}}}
\newcommand{\sa}{{\mathsf{a}}}
\renewcommand{\sb}{{\mathsf{b}}}
\newcommand{\sd}{{\mathsf{d}}}
\newcommand{\sg}{{\mathsf{g}}}
\newcommand{\si}{{\mathsf{i}}}

\newcommand{\sk}{{\mathsf{k}}}
\newcommand{\sm}{{\mathsf{m}}}
\newcommand{\sn}{{\mathsf{n}}}
\newcommand{\coker}{\mathrm{ coker }}  
\newcommand{\im}{\mathrm{ im }}        
\newcommand{\id}{\mathrm{ id}}         

\newcommand{\INT}{\mathrm{ int}}       
\newcommand{\INDEX}{\mathrm{ index}}   
\newcommand{\G}{\mathrm{G}}
\newcommand{\PSL}{\mathrm{PSL}}
\newcommand{\Lie}{\mathrm{ Lie}}          
\newcommand{\Diff}{\mathrm{ Diff}}        
\newcommand{\Vect}{\mathrm{ Vect}}        
\newcommand{\Hol}{\mathrm{ Hol}}          
\newcommand{\End}{\mathrm{ End}}          
\newcommand{\eps}{{\varepsilon}}
\newcommand{\Cinf}{C^{\infty}}
\newcommand{\reg}{\mathrm{ reg}}
\newcommand{\jhat}{{\hat{\jmath}}}
\newcommand{\inner}[2]{\left\langle #1, #2\right\rangle}

\def\NABLA#1{{\mathop{\nabla\kern-.5ex\lower1ex\hbox{$#1$}}}}
\def\Nabla#1{\nabla\kern-.5ex{}_{#1}}
\def\Tabla#1{\Tilde\nabla\kern-.5ex{}_{#1}}

\def\Abs#1{\left|#1\right|}

\def\Norm#1{\left\|#1\right\|}
\renewcommand{\Tilde}{\widetilde}

\newcommand{\half}{\mbox{$\frac12$}}
\newcommand{\p}{{\partial}}

\begin{document}

\maketitle


\section{Introduction}

This paper is a sequel to~\cite{RS}. It studies
the moduli space of stable maps whereas~\cite{RS}
studied the moduli space of stable marked nodal Riemann surfaces.
The latter can be considered as a special case of the former
by taking the target manifold $M$ to be a point. In both cases
the moduli space is the orbit space of a groupoid where
the objects are compact surfaces with additional structure.
(We think of a map from a surface to another
manifold as a structure on the surface.)
In both cases the difficulty is that to achieve compactness of
this moduli space it is necessary to include objects
whose underlying surfaces are not homeomorphic.

Here we study only that part of the moduli space
of stable maps which can be represented by
regular stable maps. Only by restricting attention
to regular stable maps can we hope to construct
an orbifold structure. We also limit attention
to target manifolds $M$ which are integrable
complex and not just almost complex.

As  in~\cite{RS} we make heavy use of ``Hardy decompositions''.
The idea is to decompose a Riemann surface $\Sigma$
into two surfaces $\Sigma'$ and $\Sigma''$ 
intersecting in their common boundary $\Gamma$.
A holomorphic map from $\Sigma$ into a complex manifold $M$ 
is uniquely determined by its restriction to $\Gamma$ and so the
space of all such holomorphic maps can be embedded 
into the space $\cV$ of smooth maps from $\Gamma$  to $M$.
In this way we identify the holomorphic  maps
with $\cV'\cap\cV''$ where $\cV'$ and $\cV''$ are the maps
 from $\Gamma$ to $M$ which extend holomorphically
 to $\Sigma'$ and $\Sigma''$ respectively.  
 (In the case where $\Sigma$ is the
 Riemann sphere, $M=\C\cup\{\infty\}$, 
 and $\Gamma$ is the equator,
 $\cV'$ would consist of those maps whose negative
 Fourier coefficients vanish and
$\cV''$ would consist of those maps whose positive
Fourier coefficients vanish. Hence the name
{\em Hardy decomposition}.)
The importance of this construction becomes
clear when we  consider a parameterized family
$\{\Sigma_b\}_{b\in B}$ of Riemann surfaces.
By judiciously choosing the decomposition 
we can arrange that the one dimensional manifolds
$\Gamma_b$ are all diffeomorphic, even though
the manifolds $\Sigma'_b$ are not all homeomorphic.
Then we identify the various $\Gamma_b$ with
a disjoint union $\Gamma$ of circles.
Under suitable hypotheses we are able
to
represent the holomorphic maps 
from $\Sigma_b$ to $M$
(for varying $b$) as a submanifold of
the manifold of smooth maps from 
$\Gamma\cong\p\Sigma'_b=\p\Sigma''_b$ to $M$.

Our theorems led to a theory of Fredholm triples
in Section~\ref{sec:interfred}.
These are triples $(X,X',X'')$ where $X$ is a Hilbert
manifold and $X'$, $X''$ are Hilbert submanifolds 
such that $T_xX'\cap T_xX''$ and $T_xX/(T_xX'+T_xX'')$ 
are finite dimensional for every $x\in X'\cap X''$. 
We prove a finite dimensional
reduction theorem for morphisms of such triples.
We hope this theory has separate interest.

In Section~\ref{sec:topology} we show that the orbifold
topology is the same as the well known
topology of Gromov convergence.

Naming the additional structures which occur in this paper
as opposed to~\cite{RS} caused us to exhaust the
Latin and Greek alphabets. Accordingly we
have changed notation somewhat. For example,
the aforementioned decomposition 
$\Sigma=\Sigma'\cup\Sigma''$
was
$\Sigma=\Delta\cup\Omega$
in~\cite{RS}. We also
use the following notations
$$
\begin{array}{lcl}
\sg &:=&\mbox{arithmetic genus of }\Sigma/\nu, \\
\sn &:=&\mbox{number of marked points}, \\
\sk &:=&\mbox{number of nodal points}, \\
\sa &:=&\mbox{complex dimension of } A,\\
\sb &:=&\mbox{complex dimension of } B,\\
\sm &:=&\mbox{complex dimension of } M.\\
\end{array}
$$
We have used the \verb$\mathsf$ font for these integers so that we 
can write $a\in A$, $b\in B$ for the elements.  We will also use the
symbol  $\sd$ to denote a homology class in $H_2(M;\Z)$.


\section{Stable maps}\label{sec:stablemap}

\para
Throughout let $(M,J)$ be a complex manifold without boundary.
A \jdef{configuration} in $M$ is a tuple $(\Sigma,s_*,\nu,j,v)$ where
$(\Sigma,s_*,\nu,j)$ is a marked nodal Riemann surface (see~\cite[\S3]{RS})
whose quotient $\Sigma/\nu$ is connected 
and $v:\Sigma\to M$
is a smooth map satisfying the nodal conditions
$$
\{x,y\}\in\nu\implies v(x)=v(y).
$$
Thus $v$ descends to the quotient $\Sigma/\nu$ and we write 
$v:\Sigma/\nu\to M$ for a smooth map $v:\Sigma\to M$ satisfying 
the nodal conditions. 
We say that the configuration has \jdef{type} $(\sg,\sn)$ 
if the marked nodal surface $(\Sigma,s_*,\nu)$
has type $(\sg,\sn)$ in the sense of \cite[Definition~3.7]{RS}
and that it has  \jdef{type} $(\sg,\sn,\sd)$ 
if in addition the map $v$  
sends the fundamental class of $\Sigma$ to the homology
class $\sd\in H_2(M;\Z)$.
The configurations form the objects of a groupoid;
an isomorphism
$$
\phi:(\Sigma',s'_*,\nu',j',v')\to(\Sigma,s_*,\nu,j,v)
$$
is an isomorphism $\phi:\Sigma'\to\Sigma$ of the underlying
marked nodal Riemann surfaces such that
$$
v' = v\circ\phi.
$$
Given two nonnegative integers $\sg$ and $\sn$ and a homology
class $\sd\in H_2(M;\Z)$ we denote by
$\cB_{\sg,\sn}(M,J)$   the groupoid of configurations of type $(\sg,\sn)$ 
and by $\cB_{\sg,\sn,\sd}(M,J)$ the  subgroupoid of configurations of type 
$(\sg,\sn,\sd)$.
\arap

\para
The configuration $(\Sigma,s_*,\nu,j,v)$ is called \jdef{holomorphic} if the
map $v$ is holomorphic, i.e.  if
$$
\bar\p_{j,J}(v):=\half\left(dv + J(v) dv\circ j\right)=0.
$$
A \jdef{stable map} is a holomorphic configuration   
whose automorphism group is finite. 
This means that each genus-0 component of $\Sigma$ 
on which $v$ is constant carries at least three special points 
and each genus-1 component of $\Sigma$ 
on which $v$ is constant carries at least one special point. 
A component on which $v$ is constant is commonly
called a {\em ghost component} so  a stable map
is a holomorphic configuration such that each ghost component
is stable in the sense of \cite[Definition~3.7]{RS}.
The stable maps of type $(\sg,\sn)$ are a subgroupoid of  $\cB_{\sg,\sn}(M,J)$;
the orbit space  $\cMbar_{\sg,\sn}$ of this subgroupoid is
(set theoretically) the \jdef{moduli space} of stable maps of type $(\sg,\sn)$.
Similarly define the subset $\cMbar_{\sg,\sn,\sd}$.
Our goal is to construct
a canonical orbifold structure on the regular part
of this space.
\arap

\dfn \label{def:regular}
A holomorphic configuration $(\Sigma,s_*,\nu,j,v)$ is called 
\jdef{regular} if
\begin{equation}\label{eq:regular}
\Omega^{0,1}_j(\Sigma,v^*TM)
= \im\,D_v + dv\cdot\Omega^{0,1}_j(\Sigma,T\Sigma)
\end{equation}
where 
$$
D_v:\Omega^0(\Sigma/\nu,v^*TM)\to\Omega^{0,1}_j(\Sigma,v^*TM)
$$
is the linearized Cauchy Riemann operator (see~\cite[page~41]{MS} and~\ref{Dv} below).  
\nfd

\para\label{regular}    Fix $\nu$ and $s_*$.
Let  $\cJ(\Sigma)\subset\End(T\Sigma)$ 
denote  the manifold of complex structures on $\Sigma$ and let
$$
\cB:=\cJ(\Sigma)\times  \Cinf(\Sigma/\nu,M).
$$
Form the  vector bundle $\cE\to\cB$ with  fiber  
$$
\cE_{j,v}:=\Omega^{0,1}_j(\Sigma,v^*TM)
$$
and  let  $\cS:\cB\to\cE$ denote the section defined by
the nonlinear Cauchy--Riemann operator
$$
\cS(j,v):=\bar\p_{j,J}(v). 
$$
A configuration $(j,v)$   is holomorphic   and only if  $\cS(j,v)=0$.
The intrinsic derivative of $\cS$ at a zero $(j,v)\in\cS^{-1}(0)$  
is the operator 
$
\cD_{j,v}:T_{j,v}\cB\to\cE_{j,v}
$
given by
$$
\cD_{j,v}(\jhat,\hat{v})=D_v\hat{v}+\half J(v)\, dv\cdot \jhat.
$$ 
{\em A holomorphic configuration $(j,v)$ 
is regular if and only if the operator $\cD_{j,v}$ is surjective.}
This follows from the following three assertions:
(1)~the tangent space to $\cB$ at $(j,v)$ is
$$
T_{j,v}\cB=
\Omega^{0,1}_j(\Sigma,T\Sigma)\times \Omega^0(\Sigma,v^*TM)
$$
(2)~When $v$ is holomorphic, we have $J(v)\, dv\cdot \jhat=dv\cdot j\jhat$.
(3)~The map 
$$
\Omega^{0,1}_j(\Sigma,T\Sigma)\to \Omega^{0,1}_j(\Sigma,T\Sigma):
\jhat\mapsto j\jhat
$$
is bijective. 
Hence, for a regular holomorphic configuration,
the zero set of  $\cS$ is a Fr\'echet manifold near $(j,v)$ 
with tangent space $\ker\,\cD_{j,v}$. 
This zero set is the ``stratum'' consisting of the holomorphic
configurations of type $(\sg,\sn)$ obtained by fixing $\nu$ and varying $(j,v)$.
Fixing $j$ gives the vector bundle over $\Cinf(\Sigma/\nu,M)$ with fibers
$\Omega^{0,1}_j(\Sigma,v^*TM)$.
When the configuration $(j,v)$ is holomorphic,
the operator $D_v$ is the intrinsic derivative of 
the section $v\mapsto\cS(j,v)$. 
\arap

\para\label{Diff}
The section $(j,v)\mapsto\cS(j,v)=\bar\p_{j,J}(v)$ is equivariant 
under the action of the group $\Diff(\Sigma,\nu)$ 
of orientation preserving diffeomorphisms that preserve the 
nodal structure.  The Lie algebra of $\Diff(\Sigma,\nu)$ is the 
space 
$$
\Vect(\Sigma,\nu)
:=\{\xi\in\Omega^0(\Sigma,T\Sigma)\,|\,\xi(z)=0\,\forall z\in\cup\nu\}
$$
of vector fields on $\Sigma$ that vanish on the nodal set. 
The infinitesimal equivariance condition is
\begin{equation}\label{eq:Ddbar}
D_v(dv\cdot\xi) = dv\cdot\bar\p_j\xi
\end{equation}
for every $\xi\in\Vect(\Sigma,\nu)$. 
The diffeomorphism group $\Diff(\Sigma,\nu)$ acts on 
the space 
$$
\cZ_\sn(\Sigma,\nu;M,J):= (\Sigma^\sn\setminus\Delta)\times\cS^{-1}(0)
$$ 
(where $\Delta$ is the fat diagonal) by 
\begin{equation}\label{eq:Diff}
g^*(s_1,\dots,s_\sn,j,v)
:=(g^{-1}(s_1),\dots,g^{-1}(s_\sn),g^*j,v\circ g)
\end{equation}
for $g\in\Diff(\Sigma,\nu)$. Let $\cP_\sn(\Sigma,\nu;M,J)\subset \cZ_\sn(\Sigma,\nu;M,J)$
denote the subset of stable maps, i.e.\ the subset where $\Diff(\Sigma,\nu)$
acts with finite isotropy.
Then the quotient space
$$
\cM_\sn(\Sigma,\nu;M,J) 
:= \cP_\sn(\Sigma,\nu;M,J)/\Diff(\Sigma,\nu) 
$$
is a stratum of the moduli space $\bar\cM_{\sg,\sn}(M,J)$ of all
stable maps of genus $\sg$ with $\sn$ marked points.
The stratum can also be expressed as the quotient
$
\cM_\sn(\Sigma,\nu;M,J) 
= \cS^{-1}(0)_\mathrm{stable}/\Diff(\Sigma,\nu,s_*)
$
where $\Diff(\Sigma,\nu,s_*)\subset\Diff(\Sigma,\nu)$ 
denotes the subgroup of all diffeomorphisms $\phi\in\Diff(\Sigma,\nu)$
that satisfy $\phi(s_\si)=s_\si$ for $\si=1,\dots,\sn$. 
\arap

\para\label{Dv}
Let $(\Sigma,\nu,j)$ be a nodal Riemann surface 
and $v:\Sigma\to M$ be a smooth map. 
Fix a connection on $TM$ and
define 
\begin{equation}\label{eq:Dv}
D_v\hat{v}
:=\half\left(\nabla\hat{v}+J(v)\nabla\hat{v}\circ j\right)
- \half J(v)\nabla_{\hat{v}}J(v)\p_{j,J}(v).
\end{equation}
(See~\cite[page~41]{MS}.)  
The definition for $D_v$ is meaningful even when 
$J$ is not integrable.
If $\bar\p_{j,J}(v)=0$,  then 
the right hand side of~(\ref{eq:Dv}) 
is independent of the choice of
the connection $\nabla$
and is the operator of Definition~\ref{def:regular}. 
If $J$ is integrable, 
$v^*TM\to\Sigma$ is a holomorphic vector bundle
and $D_v$ is its Cauchy Riemann operator.
If $\nabla$ is the Levi Civita 
connection of a K\"ahler metric, then  $\nabla J=0$ and the 
last term vanishes. In general (assuming neither integrability nor
that $(j,v)$ is a zero) the  formula for $D_v$ still defines
a Cauchy--Riemann operator on $v^*TM$ which depends
however on the connection and might not be complex linear,
but it is always Fredholm. 
\arap


\section{Unfoldings of stable maps}\label{sec:unfolding}

\para\label{family}
Fix two nonnegative integers $\sg$ and $\sn$ 
and a homology class $\sd\in H_2(M;\Z)$. 
A \jdef{(holomorphic) family of maps 
(of type $(\sg,\sn,\sd)$)}
is a triple 
$$
(\pi:Q\to B,S_*,H)
$$
where $(\pi,S_*)$ is a marked nodal Riemann family 
(of type $(\sg,\sn)$) and 
$$
H:Q\to M
$$ 
is a holomorphic map such that the restriction of $H$ 
to each fiber $Q_b$ represents the homology class $\sd$. 
A desingularization $u:\Sigma\to Q_b$ of a fiber induces 
a holomorphic configuration $(\Sigma,s_*,\nu,j,v)$  with 
$$
v:=H\circ u.
$$
The family of maps is called \jdef{stable} if each configuration
that arises from a desingularization of a fiber is a stable map. 
Given two families of maps ${(\pi_A:P\to A,R_*,H_A)}$
and ${(\pi_B:Q\to B,S_*,H_B)}$ a map $f:P_a\to Q_b$ is called 
a \jdef{fiber isomorphism} if it is a fiber isomorphism 
of marked nodal Riemann families and 
$$
H_A|P_a = H_B\circ f.
$$
A \jdef{morphism} between two families of maps
$(\pi_A,R_*,H_A)$ and $(\pi_B,S_*,H_B)$ is a 
commutative diagram 
$$
\xymatrix{
 &&&& M\\
P\ar[rrr]^\Phi\ar[d]_{\pi_A}\ar[rrrru]^{H_A} &&& Q\ar[d]^{\pi_B}\ar[ru]_{H_B} \\
A\ar[rrr]^\phi &&& B \\
}
$$
such that, for each $a\in A$, the restriction of $\Phi$ 
to the fiber $P_a$ is  a fiber isomorphism.  
The morphism is called continuous, continuously differentiable,
smooth, or holomorphic if both maps $\phi$ and $\Phi$ are. 
\arap

\dfn\label{def:univ}
An \jdef{unfolding of maps} is a quadruple $(\pi_B,S_*,H_B,b)$ 
where $(\pi_B,S_*,H_B)$ is a family of maps and $b\in B$. 
An unfolding $(\pi_B,S_*,H_B,b)$ is called \jdef{universal}
if, for every other unfolding $(\pi_A,R_*,H_A,a)$ 
and every fiber isomorphism ${f:P_a\to Q_b}$, 
there is a unique morphism 
$$
(\phi,\Phi):(\pi_A,R_*,H_A,a)\to(\pi_B,S_*,H_B,b)
$$ 
of families of maps such that 
$$
\Phi|P_a=f.
$$
This is to be understood in the sense of germs; the morphism 
may only be defined after shrinking $A$, and two morphisms
are considered equal if they agree on some neighborhood of
$P_a$.  
\nfd

\dfn\label{def:infuniv}
Let $(\pi:Q\to B,S_*,H,b)$ be an unfolding of maps and 
${u:\Sigma\to Q_b}$ be a desingularization with induced 
structures $s_*$, $\nu$, $j$, and $v$ on $\Sigma$
Define the spaces 
$$
\cX_u:=\left\{\hat{u}\in\Omega^0(\Sigma/\nu,u^*TQ)\,|\,
d\pi(u)\hat u\equiv\mathrm{constant},\;
     \hat u(s_\si)\in T_{u(s_\si)}S_\si
\right\},
$$
$$
\cY_u:=\{\eta\in\Omega^{0,1}_j(\Sigma,u^*TQ)\,|\,d\pi(u)\eta=0\},
$$
$$
\cX_v := \Omega^0(\Sigma/\nu,v^*TM),\qquad
\cY_v := \Omega^{0,1}_j(\Sigma,v^*TM).
$$
Consider the diagram 
\begin{equation}\label{eq:XY}
\xymatrix{
{\cX_u}\ar[r]^{dH(u)}\ar[d]_{D_u}&{\cX_v}\ar[d]^{D_v} \\
{\cY_u}\ar[r]^{dH(u)}&{\cY_v} 
}
\end{equation}
where the vertical maps are the restrictions 
to the indicated subspaces
of the linearized Cauchy--Riemann operators (see~\ref{Dv})
$$
D_u:\Omega^0(\Sigma,u^*TQ)\to\Omega^{0,1}(\Sigma,u^*TQ),
$$  
$$
D_v:\Omega^0(\Sigma,v^*TM)\to\Omega^{0,1}(\Sigma,v^*TM)
$$ 
associated  to the holomorphic maps $u$ and $v$.
Thus 
$D_v$ 
is the intrinsic derivative in~\ref{def:regular}.
The diagram~(\ref{eq:XY}) commutes because $H$ is 
holomorphic and hence 
$\bar\p_{j,J_M}(H\circ u)=dH(u)\cdot\bar\p_{j,J_Q}(u)$.
The commutative diagram~(\ref{eq:XY}) determines maps
\begin{equation}\label{eq:dH}
dH(u):\ker D_u\to \ker D_v,\qquad 
dH(u):\coker D_u\to \coker D_v
\end{equation}
The unfolding is called \jdef{infinitesimally universal} if the maps
in~(\ref{eq:dH}) are both bijective.
\nfd

\rmk\label{rmk:regular}
{\em Let  $(\Sigma,s_*,\nu,j,v)$ be induced by a desingularization 
$u:\Sigma\to Q_b$ of an unfolding $(\pi:Q\to B,S_*,H,b)$. 
Then $(\Sigma,s_*,\nu,j,v)$ is regular if and only if 
the map $dH(u):\coker D_u\to \coker D_v$ is surjective.}
To see this note that $dH(u):\coker D_u\to \coker D_v$ is surjective
if and only if 
\begin{equation}\label{eq:regular1}
\cY_v = \im\, D_v + \im\,(dH(u):\cY_u\to\cY_v).
\end{equation}
Since $u$ is an immersion, the map
$$
T_j\cJ(\Sigma)=\Omega^{0,1}_j(\Sigma,T\Sigma)\to \cY_u:
\eta\mapsto du\cdot \eta
$$
is an isomorphism. But $v=H\circ u$ so 
$dv\cdot\eta=dH(u)\circ du\cdot\eta$ so
$$
dv\cdot\Omega^{0,1}_j(\Sigma,T\Sigma)=\im\,(dH(u):\cY_u\to\cY_v).
$$
Hence equation~(\ref{eq:regular}) is equivalent to 
equation~(\ref{eq:regular1}) which asserts that the 
holomorphic configuration $(\Sigma,s_*,\nu,j,v)$ is regular.
\kmr

When $M$ is a point the above definitions and the following theorems 
agree with the corresponding ones in~\cite{RS}.  

\begin{theorem}\label{thm:exists} 
A holomorphic configuration $(\Sigma,s_*,\nu,j,v)$ admits 
an infinitesimally universal unfolding if and only if 
it is a regular stable map.
\end{theorem}

\begin{proof}
The hard part of the proof is to show that `if' holds under the 
additional assumption that the underlying marked nodal Riemann 
surface $(\Sigma,s_*,\nu,j)$ is stable. We will prove this
in Section~\ref{sec:proof}.  Here we give the easy parts 
of the proof.  

We prove `if' (assuming the aforementioned result of 
Section~\ref{sec:proof}).  By adding marked points in the 
appropriate components we may construct a stable map
whose underlying marked nodal Riemann 
surface is stable.  Hence, by backwards induction, it 
is enough to prove the following

\medskip
\noindent{\bf Claim.}  {\it If a stable map admits an
infinitesimally universal unfolding and the configuration
which results on deleting a marked point is also a stable map, 
then it too admits an infinitesimally universal unfolding.}
\medskip

To prove the claim let $(\pi:Q\to B,S_1,\dots,S_\sn,H,b_0)$ be an infinitesimally 
universal unfolding of $(\Sigma,s_1,\dots,s_\sn,\nu,j,v)$ with 
associated desingularization $u:\Sigma\to Q_{b_0}$
and assume that $(\Sigma,s_1,\dots,s_{\sn-1},\nu,j,v)$ is still stable.
We will construct an infinitesimally universal unfolding
$
(\pi:Q'\to B',S_1',\ldots,S_{\sn-1}',H',b_0)
$
such that $B'$ is a submanifold of $B$, 
$Q':=\pi^{-1}(B')$ is a submanifold of $Q$, $H':=H|Q'$,
and $S_\si'=S_\si\cap Q'$ for $\si=1,\ldots,\sn-1$. 
Define the space 
$$
\Tilde{\cX}_u:=\left\{\hat{u}\in\Omega^0(\Sigma/\nu,u^*TQ)\,|\,
d\pi(u)\hat u\equiv\mathrm{constant},\,
\hat u(s_\si)\in T_{u(s_\si)}S_\si
\mbox{ for }\si<\sn\right\}.
$$
Note that $\Tilde{\cX}_u$ is obtained from $\cX_u$ by removing the constraint
on the value $\hat u(s_\sn)$ at the last marked point. Thus $\cX_u$ is a 
subspace of $\Tilde{\cX}_u$ of complex codimension one; 
a complement of $\cX_u$ in $\Tilde{\cX}_u$ is spanned 
by any vertical vector field along $u$, satisfying the nodal condition,  
that vanishes at the marked points $s_\si$ for $\si<1$ 
and does not vanish at $s_\sn$. Denote by 
$$
\Tilde{D}_u:\Tilde{\cX}_u\to\cY_u
$$ 
the operator given by the same formula as $D_u$ on the larger domain.  
Note that the diagram~(\ref{eq:XY}) continues to commute when 
we replace $\cX_u$ and $D_u$ by $\Tilde{\cX}_u$ and $\Tilde{D}_u$,
respectively. We prove the following.
\begin{description}
\item[(a)]
$\mathrm{im}\,D_u=\mathrm{im}\,\Tilde{D}_u$ 
and $\ker D_u\subset\ker\Tilde{D}_u$ is a 
subspace of codimension one. 
\item[(b)]
There is an element $\hat u\in\ker\Tilde{D}_u$ with 
$dH(u)\hat u\equiv 0$ and $\hat b:=d\pi(u)\hat u\ne 0$. 
\end{description}
With this understood we choose a complex submanifold
$B'\subset B$ of codimension one such that 
$\pi$ is tranverse to $B'$ and $\hat b\notin T_{b_0}B'$. 
Then the kernel of the resulting operator $D_u'$ is a complex 
subspace of the kernel of $\Tilde{D}_u$ of codimension one.
Since $\hat b\notin T_{b_0}B'$, the kernel of $D_u'$
is mapped under $dH(u)$ isomorphically onto the kernel
of $D_v$.  Since $D_u'$ has the same image as $\Tilde{D}_u$ 
and $D_u$ we deduce that $dH(u)$ also induces an 
isomorphism from the cokernel of $D'_u$ to that of $D_v$. 
Hence $(\pi:Q'\to B',S_1',\dots,S_{\sn-1}',H',b_0)$ is an
infinitesimally universal unfolding 
of $(\Sigma,s_1,\dots,s_{\sn-1},\nu,j,v)$
as claimed.

\smallbreak

It remains to prove~(a) and~(b).  To prove~(a) note 
that $\Tilde{D}_u$ has the same image as $D_u$.
(If $\eta\in\cY_u$ belongs to the image of $\Tilde{D}_u$
then $dH(u)\eta\in\mathrm{im}\,D_v$ and, since 
the second map in~(\ref{eq:dH}) is injective, this 
implies that $\eta$ belongs to the image of $D_u$.) 
Hence~(a) follows from the fact that $\cX_u$ has 
codimension one in $\Tilde{\cX}_u$.
To prove~(b) we use the fact that the first map in~(\ref{eq:dH}) 
is surjective and $dH(u)$ maps the kernel of $\Tilde{D}_u$
to the kernel of $D_v$.  Hence there is an element 
$$
\hat u\in\ker\Tilde{D}_u\cap\ker dH(u)\setminus\ker D_u.  
$$
Any such element satisfies 
$$
d\pi(u)\hat u\ne 0. 
$$
Otherwise there is a vector field $\xi\in\Vect(\Sigma)$ with 
$\hat u=du\cdot\xi$;  since $\hat u\in\Tilde{\cX}_u$ this implies 
that $\xi$ belongs to the Lie algebra of the stabilizer subgroup
of $(\Sigma,s_1,\dots,s_{\sn-1},\nu,j,v)$, contradicting stablility.
Thus we have proved~(a) and~(b) and hence the claim.

We prove `only if'.  
Let  $(\Sigma,s_*,\nu,j,v)$ be induced by a 
desingularization  $u:\Sigma\to Q_b$ of the  
infinitesimally universal unfolding $(\pi:Q\to B,S_*,H,b)$. 
Then the holomorphic configuration $(\Sigma,s_*,\nu,j,v)$ 
is regular, by Remark~\ref{rmk:regular}.
Next we argue as in~\cite{RS}.
Assume that $(\Sigma,s_*,\nu,j,v)$ is regular but not stable. 
Then either $\Sigma$ has genus one,
$v$ is constant, and there are no special points 
or else $\Sigma$ contains a component of
genus zero on which $v$ is constant and which carries
at most two special points. In either case there
is an abelian complex Lie group $A$ (namely $A=\Sigma$ in the former case
and $A=\C^*$ in the latter) and an effective holomorphic action
$$
A\times\Sigma\to\Sigma:(a,z)\mapsto a_\Sigma(z)
$$
that preserves the given structures.
Let $P:=A\times\Sigma$, $\pi_A$ be the projection
on the first factor, $R_*:=A\times s_*$, 
$f_A(a,z):=v(z)$, and $a_0\in A$ be the identity. 
If $u_0:\Sigma\to Q$ is any desingularization
of a fiber $Q_{b_0}$ of an unfolding $(\pi_B:Q\to B,S_*,f_B,b_0)$
which induces the given structures on~$\Sigma$, then
$$
\Phi_1(a,z):=u_0(z),\qquad \Phi_2(a,z):=u_0(a_\Sigma(z))
$$
are distinct morphisms from $(\pi_A,R_*,f_A,a_0)$
to $(\pi_B,S_*,f_B,b_0)$ which extend the fiber isomorphism
${P_{a_0}\to Q_{b_0}:(a_0,z)\mapsto u_0(z)}$.  
Hence $(\pi_B,S_*,f_B,b_0)$ is not a universal unfolding.
\end{proof}

\begin{theorem}\label{thm:universal} 
An unfolding of a regular stable map
is universal if and only if it is infinitesimally universal.
\end{theorem}

\begin{proof}
We prove `if' in Section~\ref{sec:proof}.
For `only if' we argue as in~\cite{RS}.
A composition of morphisms (of nodal families of maps) 
is again a morphism. The only morphism which is the identity 
on the central fiber of a universal unfolding  is the identity.
It follows that any two universal unfoldings of the same 
holomorphic configuration are isomorphic.
By Theorem~\ref{thm:exists} there is  an 
infinitesimally universal unfolding and by `if'  it  is universal 
and hence isomorphic to every other universal unfolding.
Any unfolding isomorphic to an infinitesimally
universal unfolding is itself infinitesimally universal.
\end{proof}

\xmpl 
Here is an example of an unfolding which is universal
but not infinitesimally universal. 
Let $B=\C$, $b_0=0$, $\Sigma$ be a Riemann surface 
of genus $\sg\ge1$, $Q=M=B\times\Sigma$, 
$\pi_B:Q\to B$ be the projection on the first factor,
and $H_B:Q\to M$ be the identity map.
This is trivially universal as follows.
If $(\pi_A,H_A,a_0)$ is another unfolding 
and $f_0:P_{a_0}\to Q_{b_0}$ is a fiber isomorphism
as in~\ref{family}, then $f_0=H_A|P_{a_0}$,
the unique solution of $H_B\circ\Phi=H_A$ is  $\Phi=H_A$,
and $\phi$ is uniquely determined by the condition
$\pi_B\circ\Phi=\phi\circ\pi_A$.
To show that that the example is not infinitesimally universal  
it is enough (by Theorem~\ref{thm:exists})
to show that the fiber is not regular, i.e. that
$$
\im D_v +dv\cdot\Omega^{0,1}_j(\Sigma,T\Sigma)
\subsetneq \Omega^{0,1}_j(\Sigma,TM)
$$
where $v:\Sigma\to M$ is the map $v(z):=(b_0,z)$.
Now $TM$ is the direct sum of $dv\cdot T\Sigma$ with a trivial bundle, so
it is enough to show that $D_v$ followed by projection of the trivial
bundle is not surjective. But this is the linear operator
$\dbar:\Omega^0(\Sigma)\to \Omega^{0,1}_j(\Sigma)$.
Its cokernel is the space of holomorphic $1$-forms 
and it has dimension $\sg$.
\lpmx

\begin{theorem}\label{thm:stable}
If an unfolding $(\pi,S_*,H,b_0)$ is infinitesimally universal, then
the unfolding $(\pi,S_*,H,b)$ is infinitesimally universal 
for $b$ sufficiently near $b_0$.
\end{theorem}

\para\label{univ-family}
Fix two nonnegative integers $\sg$ and $\sn$ and a homology 
class $\sd\in H_2(M;\Z)$. 
A \jdef{universal family of maps} of type $(\sg,\sn,\sd)$
is a marked nodal family of maps $(\pi_B:Q\to B,S_*,H_B)$ 
satisfying the following conditions.
\begin{description}
\item[(1)] 
$(\pi_B,S_*,H_B,b)$ is a universal unfolding of maps of type $(\sg,\sn,\sd)$
for every $b\in B$.
\item[(2)] 
Every regular stable map of type $(\sg,\sn,\sd)$
arises from a desingularization of at least one fiber of $\pi_B$.
\item[(3)] 
$B$ is second countable.
\end{description}
The existence of a universal marked nodal family of maps 
for every triple $(\sg,\sn,\sd)$
follows immediately from 
Theorems~\ref{thm:exists}, \ref{thm:universal}, and~\ref{thm:stable} 
as in~\cite[Proposition~6.3]{RS}.
\arap

\para\label{B-Gamma}
Every universal family $(\pi_B:Q\to B,S_*,H_B)$ 
of maps of type $(\sg,\sn,\sd)$ determines a groupoid 
$(B,\Gamma,s,t,e,i,m)$ as in~\cite[Definition~6.4]{RS};
here $\Gamma$ denotes the set of all triples $(a,f,b)$ 
such that $a,b\in B$ and $f:Q_a\to Q_b$ is a fiber isomorphism
satisfying $H_B\circ f = H_B|Q_a$,  and the structure maps 
$s,t:\Gamma\to B$, $e:B\to\Gamma$, $i:\Gamma\to\Gamma$, 
and $m:\Gamma\Times{s}{t} \Gamma\to\Gamma$ 
are defined by
$$
s(a,f,b):=a,\qquad
t(a,f,b):=b,\qquad
e(a):=(a,\id,a),
$$
$$
i(a,f,b):=(b,f^{-1},a),\qquad
m((b,g,c),(a,f,b)):=(a,g\circ f,c).
$$
The associated groupoid is equipped with a functor
$
B\to\bar\cB^\reg_{\sg,\sn,\sd}(M,J):b\mapsto\Sigma_b 
$
to the groupoid of Definition~\ref{def:regular}, i.e. $\iota_b:\Sigma_b\to Q_b$
denotes the canonical desingularization in~\cite[Remark~4.4]{RS}.
By definition the induced map
$$
B/\Gamma\to\bar\cM^\reg_{\sg,\sn,\sd}(M,J)
$$
on orbit spaces is bijective.  As in~\cite[Theorem~6.5]{RS} 
the groupoid $(B,\Gamma)$ equips the moduli space 
$\bar\cM^\reg_{\sg,\sn,\sd}(M,J)$ with an orbifold structure which 
is independent of the choice of the universal family.  
\arap

\begin{theorem}\label{thm:proper}
Let $(\pi_B:Q\to B,S_*,H_B)$ 
 be a universal family of maps of 
type $(\sg,\sn,\sd)$ as in~\ref{univ-family}.  
Then the associated 
groupoid $(B,\Gamma)$ constructed in~\ref{B-Gamma} is proper
in the sense of~\cite[2.2]{RS}.
\end{theorem}
\begin{proof}
See Section~\ref{sec:proof}.
\end{proof}

\begin{corollary} 
Fix a homology class $\sd\in H_2(M;\Z)$. 
Then the moduli space $\bar\cM^\reg_{\sg,\sn,\sd}(M,J)$ 
of isomorphism classes of regular stable maps of genus $\sg$ 
with $\sn$ marked points representing the class $\sd$ 
is a complex orbifold of dimension
$$
\dim_\C\bar\cM^\reg_{\sg,\sn,\sd}(M,J) 
=(g-1)(3-\dim_\C M) +\inner{c_1(TM)}{\sd} +\sn. 
$$
\end{corollary}

\rmk 
If $(M,\omega,J)$ is a K\"ahler manifold with a transitive action by
a compact Lie group $G$, then every genus zero configuration in 
$M$ is regular (see~\cite{RT} or~\cite[Proposition~7.4.3]{MS}). 
Hence the moduli space $\bar\cM_{0,\sn,\sd}(M,J)$ 
is a (compact) complex orbifold for every $\sd\in H_2(M;\Z)$. 
For $M=\C P^\sm$ 
this result is due to Fulton and Pandharipande~\cite{FP}.
Their result applies to all projective manifolds 
whenever all the stable maps are regular.
In such cases they show that the moduli space is an algebraic orbifold.
In contrast, our result shows that the set of regular maps
into {\em any} complex manifold is an orbifold. 
\kmr


\section{Stable maps without nodes}\label{sec:nonodes}

In this section we restrict attention to regular stable maps without nodes.  
Let $(\Sigma,s_*,j_0,v_0)$ be a  regular stable map
of type 
$(\sg,\sn,\sd)$ 
without nodes. 
We will construct an infinitesimally universal unfolding
$(\pi_B,S_*,H_B,b_0)$ of $(\Sigma,s_*,j_0,v_0)$, 
show that it is universal, and prove that every other
infinitesimally universal unfolding of $(\Sigma,s_*,j_0,v_0)$
is isomorphic to the one we've constructed. 

\para\label{cBcQ}
Fix two nonnegative integers $\sn$ and $\sg$,
a homology class $\sd\in H_2(M;\Z)$,  
and a compact  oriented surface 
$\Sigma$ without boundary 
of genus $\sg$. 
Denote
$$
\cP:=\left\{(s_1,\dots,s_\sn,j,v)\,\Bigg|\,\begin{array}{l}
s_*\in\Sigma^\sn\setminus\Delta,\,
j\in\cJ(\Sigma),\,v\in \Cinf(\Sigma,M) \\
\bar\p_{j,J}(v)=0,\,[v]=\sd\\ 
\cD_{j,v}\mbox{ is onto},\,
(s_*,j,v)\mbox{ is stable}
\end{array}
\right\}
$$
where $\Delta\subset\Sigma^\sn$ denotes the fat diagonal,
$[v]:=v_*[\Sigma]$ denotes the homology class represented by $v$,
and
$$
\cD_{j,v}:
\Omega^{0,1}_j(\Sigma,T\Sigma)\times
\Omega^0(\Sigma,v^*TM)\to\Omega^{0,1}(\Sigma,v^*TM)
$$
denotes the linearized Cauchy--Riemann operator 
of~\ref{regular}.  Thus $\cP$ is the regular part of the
space $\cP_{\sn,\sd}(\Sigma;M,J)$ in~\ref{Diff}. 
The group 
$$
\cG:=\Diff_0(\Sigma)
$$ 
of orientation preserving diffeomorphisms of $\Sigma$
that are isotopic to the identity acts on $\cP$ 
as in equation~(\ref{eq:Diff}):
$$
g^*(s_1,\dots,s_\sn,j,v)
:= (g^{-1}(s_1),\dots,g^{-1}(s_\sn),g^*j,g^*v)
$$
for $g\in\cG$. 
\arap

\rmk 
Roughly speaking, the tuple $(\cQ\to\cB,\cS_*,\cH)$
defined by
$$
\cB := \cP/\cG,\qquad \cQ := \cP\times_\cG\Sigma,
$$
$$
\cH\left([s_1,\dots,s_\sn,j,v,z]\right) := v(z),\quad
\cS_\si := \left\{[s_1,\dots,s_\sn,j,v,z]\in\cQ\,|\,z=s_\si\right\}
$$
is a universal family. Our task is to make sense of these quotients.
In the case 
$$
\sn>2-2\sg
$$ 
the action is free. In general, the action is only semi-free, 
i.e.~the isotropy group of a point in~$\cP$ is always finite 
but it might be nontrivial.
(Example: $\sn=0$, $\Sigma=M=S^2$, $v(z)=z^2$.)
In this case the quotient spaces $\cB$ and $\cQ$ cannot 
be manifolds and hence do not qualify as universal unfoldings.
However, we shall prove that even in this case every point in $\cP$
admits a holomorphic local slice for the $\cG$-action and 
that these slices can be used to construct universal unfoldings. 
\kmr

\para\label{TP}
The space $\cP$ is an infinite dimensional Frech\'et manifold.
Its tangent space at a point $p=(s_*,j,v)\in\cP$ 
is the space $T_p\cP$ of all tuples $\hat p=(\hat s_*,\jhat,\hat v)$
with $\hat s_\si\in T_{s_\si}\Sigma$, $\jhat\in T_j\cJ(\Sigma)$, 
$\hat v\in\Omega^0(\Sigma,v^*TM)$ that satisfy
\begin{equation}\label{eq:TP}
D_v\hat v + \frac12 J(v)dv\circ\jhat=0.
\end{equation}
The Lie algebra of $\cG$ is $\Lie(\cG)=\Vect(\Sigma)$ and 
its (contravariant) infinitesimal action at $p\in\cP$ is the operator 
$\cL_p:\Vect(\Sigma)\to T_p\cP$
defined by
\begin{equation}\label{eq:cL}
\cL_p\xi:= \left.\frac{d}{dt}g_t^*p\right|_{t=0}
\end{equation}
where $p=(s_*,j,v)\in\cP$ and $\R\to\cG:t\mapsto g_t$ satisfies
\begin{equation}\label{eq:gxi}
g_0=\id,\qquad \left.\frac{d}{dt}g_t\right|_{t=0}=\xi. 
\end{equation}
(The right hand side of~(\ref{eq:cL}) is independent of the choice of $g_t$
satisfying~(\ref{eq:gxi}).)
Since $2j\bar\p_j\xi=\cL_\xi j\in T_j\cJ(\Sigma)$ is the 
Lie derivative of $j$ in the direction $\xi$, equation~(\ref{eq:cL}) may be written
\begin{equation}\label{eq:infinitesimal}
\cL_p\xi=(-\xi(s_1),\dots,-\xi(s_\sn),2j\bar\p_j\xi,dv\cdot\xi),\qquad
p=(s_1,\dots,s_\sn,j,v).
\end{equation}
The image of $\cL_p$ is the tangent space $T_p\cG^*p$ 
to the $\cG$-orbit of $p$.  The space $T_p\cP$ carries 
a natural complex structure 
$\cI(p):T_p\cP\to T_p\cP$ given by 
\begin{equation}\label{eq:complex}
\cI(p)(\hat s_1,\dots,\hat s_\sn,\jhat,\hat v):=
(j(s_1)\hat s_1,\dots,j(s_\sn)\hat s_\sn,j\jhat,J(v)\hat v)
\end{equation}
for $p=(s_1,\dots,s_\sn,j,v)\in\cP$.  The tangent space $T_p\cP$
is invariant under $\cI(p)$ because the differential $dv$ 
and the operator $D_v$ are complex linear.  
The $\cG$-action preserves this complex structure 
and the formula
$$
\cL_pj\xi = \cI(p)\cL_p\xi,\qquad p=(s_*,j,v)\in\cP,
$$
shows that $T_p\cG^*p$ is a complex subspace of $T_p\cP$.  
In other words, the orbits of $\cG$ are 
complex submanifolds of $\cP$ and the complex structure 
descends to the quotient $\cP/\cG$.  
The space $\cP$ (without marked points) is the zero set 
of the section $(j,v)\mapsto\bar\p_{j,J}(v)$ of an infinite 
dimensional vector bundle.  The intrinsic differential of this section 
at a zero $(j,v)$ is the operator $\cD_{j,v}$ in~\ref{regular} and this 
operator is surjective by assumption.  Condition~(\ref{eq:TP}) asserts
that the pair $(\jhat,\hat v)$ belongs to the kernel of $\cD_{j,v}$.
Choosing a suitable Sobolev completion $\cP^s$ of $\cP$ 
(see the proof of Theorem~\ref{thm:slice} below)
we can deduce that $\cP^s$ is a smooth Hilbert manifold whose tangent space
is given by~(\ref{eq:TP}).  The action of $\cG$ on this Hilbert manifold
is not smooth; on any Sobolev completion its differential takes
values in another Sobolev completion with one derivative less.
However, in the Frech\'et category, where $B$ is a finite dimensional
smooth manifold, the notion of a smooth map $\iota:B\to\cP$ and its differential 
$d\iota(b):T_bB\to T_{\iota(b)}\cP$ have well defined meanings 
via evaluation maps.  
\arap

\begin{lemma}\label{le:hol-unfolding}
Let $A$ be a complex manifold (with complex structure $\sqrt{-1}$), 
$$
A\to\cP:a\mapsto p(a)=(r_1(a),\dots,r_\sn(a),j(a),v(a))
$$ 
be a smooth map and $\eta:TA\to\Vect(\Sigma)$ be 
a $1$-form on $A$ with values in the space 
of vector fields  on $\Sigma$ such that
\begin{equation}\label{eq:jeta1}
\eta(a,\sqrt{-1}\hat{a})=-j(a)\eta(a,\hat{a})
\end{equation}
for all $(a,\hat{a})\in TA$.
Define an almost complex structure $J_P$ on $P:=A\times\Sigma$,
sections $R_1,\dots,R_\sn\subset P$, and a map $H_A:P\to M$ by 
\begin{equation}\label{eq:JP}
J_P(a,z)(\hat a,\hat z) 
:= \left(\sqrt{-1}\hat a,j(a)(z)\hat z+\eta(a,\hat a)(z)\right),
\end{equation}
\begin{equation}\label{eq:RH}
R_\si := \left\{(a,r_\si(a))\,|\,a\in A\right\},\qquad
H_A(a,z) := v(a)(z).
\end{equation}
Then the following are equivalent.
\begin{description}
\item[(i)]
The tuple $(\pi_A,R_*,H_A)$ is a (holomorphic)  family of maps, 
i.e.~$J_P$ is integrable, each $R_\si$ is a complex submanifold of $P$,
and $H_A:P\to M$ is holomorphic.
\item[(ii)]
$p$ and $\eta$ satisfy the differential equation
\begin{equation}\label{eq:vortex}
dp(a)\hat a + \cI(p(a))dp(a)\sqrt{-1}\hat a 
- \cL_{p(a)}\eta(a,\sqrt{-1}\hat a) = 0
\end{equation}
for every $a\in A$ and every $\hat a\in T_aA$.
\end{description}
\end{lemma}

\begin{proof}
We prove that~(i) implies~(ii). If the almost 
complex structure $J_P$ is integrable then, 
by~\cite[Corrigendum, Lemma~A]{RS}, we have
\begin{equation}\label{eq:jeta2}
dj(a)\hat a + j(a)dj(a)\sqrt{-1}\hat a
- \cL_{\eta(a,\sqrt{-1}\hat a)}j(a)=0.
\end{equation}
Moreover, for $\si=1,\dots,\sn$ the set $R_\si$ 
is a complex submanifold of $A\times\Sigma$,
if and only if
$$
dr_\si(a)\hat a + j(a)dr_\si(a)\sqrt{-1}\hat a
+ \eta(a,\sqrt{-1}\hat a)(r_\si(a)) = 0
$$
and $H_A:A\times\Sigma\to M$
is holomorphic if and only if
$$
(dv(a)\hat a)(z) 
+ J(v(a)(z))(dv(a)\sqrt{-1}\hat a)(z) 
- d(v(a))(z)\eta(a,\sqrt{-1}\hat a)(z)
= 0.
$$
In the last formula $(dv(a)\hat{a})(z)$ denotes the derivative 
of $v(a)(z)$ with respect to $a$ and $d(v(a))(z)\hat{z}$ 
denotes the derivative of $v(a)(z)$ with respect to $z$.
This proves that~(i) implies~(ii). 

Conversely, assume~(ii) and, without loss of generality, 
that $A$ is an open set in $\C^\sa$.
Fix two vectors $\hat a,\hat b\in \C^\sa$ and, for $a\in A$,
define $\zeta(a)\in\Vect(\Sigma)$ by  
\begin{eqnarray*}
\zeta(a)
&:=&
\p_1\eta(a,\hat a)\sqrt{-1}\hat b - j(a)\p_1\eta(a,\hat a)\hat b  \\
&&
- \p_1\eta(a,\hat b)\sqrt{-1}\hat a + j(a)\p_1\eta(a,\hat b)\hat a
+ [\eta(a,\hat a),\eta(a,\hat b)].
\end{eqnarray*}
Then
$$
\cL_{\zeta(a)}j(a)=0,\qquad 
\zeta(a)(r_i(a))=0,\qquad 
\cL_{\zeta(a)}v(a) = 0
$$
for $a\in A$ and $i=1,\dots,\sn$. Here the first 
equation follows from~\cite[Corrigendum, Lemma~B]{RS} 
and the other two equations follow from similar,
though somewhat lengthy, calculations.
Now it follows from the stability condition in the definition
of $\cP$ that $\zeta(a)=0$ for every $a\in A$ and hence, 
by~\cite[Corrigendum, Lemma~A]{RS} the almost complex 
structure $J_P$ is integrable. 
This proves the lemma.
\end{proof}

\para\label{slice}
Let $p_0:=(s_{0,*},j_0,v_0)\in\cP$, $B$ be a complex manifold 
with base point $b_0\in B$, and $\iota:B\to\cP$ be a smooth map 
such that $\iota(b_0)=p_0$.  
The map $\iota$ is called \jdef{holomorphic} if its 
differential $d\iota(b):T_bB\to T_{\iota(b)}\cP$ 
is complex linear for every $b\in B$.  
The map $\iota$ is called a \jdef{slice}  at $b_0$ if
for every smooth map $p:(A,a_0)\to(\cP,p_0)$ there is a 
neighborhood $A_0$ of $a_0$ in $A$
and unique smooth maps $\Phi:(A_0,a_0)\to(\cG,\id)$ and
$\phi:(A_0,a_0)\to(B,b_0)$ such that  
$$
p(a)=\Phi(a)^*\iota(\phi(a))
$$
for $a\in A_0$. The map $\iota$ is called an \jdef{infinitesimal slice} 
at $b_0$ if
\begin{equation}\label{eq:slice}
\mathrm{im}\,d\iota(b_0)\oplus T_{p_0}\cG^*p_0
=T_{p_0}\cP,\qquad \ker d\iota(b_0) = 0.
\end{equation}
Write $\iota(b)=:(\sigma_1(b),\dots,\sigma_\sn(b),j(b),v(b))$.
Then~(\ref{eq:slice}) can be expressed as follows.
\begin{description}
\item[($\dagger$)]
If $\hat b\in T_{b_0}B$ and $\hat u\in\Vect(\Sigma)$ satisfy
\begin{equation}\label{eq:slice1}
\left.
\begin{aligned}
d\sigma_\si(b_0)\hat b-\hat u(s_{0,\si}) &=0\\
dj(b_0)\hat b+2j_0\bar\p_{j_0}\hat u&=0\\
dv(b_0)\hat b+dv_0\cdot\hat u &=0
\end{aligned}
\right\}
\qquad\implies\qquad
\hat b=0,\;\;\hat u=0.
\end{equation}
\item[($\ddagger$)]
If $\hat s_\si\in T_{s_{0,\si}}\Sigma$, $\jhat\in T_{j_0}\cJ(\Sigma)$, 
and $\hat v\in\Omega^0(\Sigma,v_0^*TM)$ satisfy~(\ref{eq:TP}) 
then there exists a pair $(\hat b,\hat u)\in T_{b_0}B\times\Vect(\Sigma)$
such that 
\begin{equation}\label{eq:slice2}
\begin{aligned}
d\sigma_\si(b_0)\hat b-\hat u(s_{0,\si})&=\hat s_\si,\\
dj(b_0)\hat b + 2j_0\bar\p_{j_0}\hat u&=\jhat,\\
dv(b_0)\hat b + dv_0\cdot\hat u&=\hat v.
\end{aligned}
\end{equation}
\end{description}
\arap

\begin{theorem}[Slice Theorem]\label{thm:slice}
{\bf (i)} A smooth infinitesimal slice is a slice.

\smallskip\noindent{\bf (ii)}
If $\iota:B\to\cP$ is an infinitesimal slice at $b_0\in B$ 
then it is an infinitesimal slice at $b$ 
for $b$ sufficiently near $b_0$.

\smallskip\noindent{\bf (iii)}
Every point in $\cP$ admits a holomorphic infinitesimal slice 
$\iota:B\to\cP$ of complex dimension
$
\dim_\C B= (\sm-3)(1-\sg) + \inner{c_1}{\sd} + \sn.
$
\end{theorem}

\begin{proof} 
Choose an integer $s\ge3$ and let $\cG^s$ denote the Sobolev completion
of $\cG$ in the $H^s$ topology and $\cP^s$ denote the Sobolev completion
of $\cP$ in the $H^{s-1}$ topology on $j$ and the $H^s$ topology on $v$.
Then 
$$
\cP^s\subset\Sigma^\sn\times\cJ^{s-1}(\Sigma)\times H^s(\Sigma,M)
$$ 
is a smooth Hilbert submanifold.
Now let $\iota:(B,b_0)\to(\cP,p_0)$ be a smooth infinitesimal slice.  

\medskip\noindent{\bf Claim 1: } {\em The map
$$
B\times\cG^s\to\cP^s:(b,g)\mapsto
\cF^s(b,g) := g^*\iota(b)
$$
is a $C^{s-2}$ map between Hilbert manifolds.
The tangent space of $\cG^s$ at $\phi=\id$ is the space
$H^s(\Sigma,T\Sigma)$ of vector fields of class $H^s$ and 
the differential of $\cF^s$ at the pair $(b,\id)$ is
$$
d\cF^s(b,\id)(\hat b,\xi)
= d\iota(b)\hat b + \cL_{\iota(b)}\xi
$$
for $\hat b\in T_bB$ and $\xi\in H^s(\Sigma,T\Sigma)$.
(See~(\ref{eq:infinitesimal}) for the definition of $\cL_{\iota(b)}$.)}

\medskip\noindent
Denote the value of $\iota(b)$ at a point $x\in\Sigma$ by 
$$
\iota(b)(x)= (\sigma_{1,b},\ldots,\sigma_{\sn,b},j_b(x),v_b(x).
$$
The maps $\sigma_i:B\to\Sigma$, $j:B\times\Sigma\to\End(T\Sigma)$, 
and $v:B\times\Sigma\to M$ are all smooth by hypothesis. 
The map $\cG^s\to\cG^s:g\mapsto g^{-1}$ is smooth. 
Hence the map $B\times\cG^s\to\Sigma:(b,g)\mapsto g^{-1}(\sigma_{i,b})$ is 
as smooth as the evaluation map $\cG^s\times\Sigma\to\Sigma$, 
i.e. it is $C^{s-2}$ by Sobelov.  Moreover, the map $g\mapsto dg$
is smooth as a map from $H^s$ to $H^{s-1}$.
Since $(g^*j_b)(x)=dg(x)^{-1}j_b(g(x))dg(x)$ this shows that the map
$$
B\times\cG^s\to\cJ^{s-1}(\Sigma):(b,g)\mapsto g^*j_b
$$
is smooth. The map $B\times\cG^s\to H^s(\Sigma,M):(b,g)\mapsto v_b\circ g$
is smooth because the map $v:B\times\Sigma\to M$ is smooth. 
This proves  claim~1.

\medskip\noindent{\bf Claim 2:} {\it The operator $d\cF^s(b,\id)$ is bijective 
if and only if $\iota$ is an infinitesimal slice at $b$.}

\medskip\noindent
To see this, assume first that $\iota$ is an infinitesimal slice at $b$.
Then, by elliptic regularity, every element in the 
kernel of $d\cF^s(b,\id)$ is smooth and hence
the operator is injective by~($\dagger$). 
For surjectivity we observe that the image of $d\cF^s(b,\id)$
is closed by the elliptic estimate, that the smooth elements
are dense in $T_{\iota(b)}\cP^s$, and that the smooth elements
of $T_{\iota(b)}\cP^s$ are contained in the image of $d\cF^s(b,\id)$
by~($\ddagger$).  Conversely, if $d\cF^s(b,\id)$ is bijective,
it follows from elliptic regularity that $\iota$ satisfies the 
infinitesimal slice conditions~($\dagger$) and~($\ddagger$)
at $b$.  This proves  claim~2. 

Shrinking $B$ if necessary, we may assume that 
$d\cF^s(b,\id)$ is bijective for every $b\in B$.  By Claim~2
this implies that $\iota$ is an infinitesimal slice at every point 
$b\in B$ and  $d\cF^{s'}(b,\id)$ is bijective for every $b$ and every~$s'$.  
Hence, by equivariance, $d\cF^{s'}(b,g)$ is bijective for 
every integer $s'\ge 2$, every $b\in B$, and every $g\in\cG^{s'}$. 
In particular, we have proved~(ii). 

Now fix an integer $s_0\ge 3$.  Then it follows from the inverse 
function theorem that $\cF^{s_0}$ maps an open $H^{s_0}$ 
neighborhood of $(b_0,\id)$ in $B\times\cG^{s_0}$ by a 
$C^{s_0-2}$-diffeomorphism onto an open neighborhood 
of $p_0$ in $\cP^{s_0}$.  Given a smooth map $p:(A,a_0)\to(\cP,p_0)$ 
choose $A_0\subset A$ to be the preimage of this neighborhood 
of $p_0$ and define the $C^{s_0-2}$ map
$$
A_0\to B\times\cG^{s_0}:a\mapsto(\phi(a),\Phi(a))
$$
by
$$
(\phi(a),\Phi(a)):=(\cF^{s_0})^{-1}(p(a)).
$$
Then
$$
p(a) = \Phi(a)^*\iota(\phi(a))
$$
for every $a\in A_0$.  Since the complex structures on $\Sigma$
associated to $\iota\circ\phi(a)$ and $p(a)$ are smooth it follows 
from elliptic regularity that $\Phi(a)\in\cG$ is smooth for every 
$a\in A_0$. Thus $\Phi(a)\in\cG^s$ and $\cF^s(\phi(a),\Phi(a))=p(a)$ 
for every $a\in A_0$ and every $s$.
Since the differential $d\cF^s(\phi(a),\Phi(a))$ is bijective 
for every $a\in A_0$ and every integer $s\ge 2$, it follows that
the map $a\mapsto(\phi(a),\Phi(a))$ is a $C^{s-2}$ map 
from $A_0$ to $B\times\cG^s$ for every integer $s\ge 3$.
Hence this map is smooth. This proves~(i).

We prove~(iii). 
Fix an element $(s_{0,*},j_0,v_0)\in\cP$.     
Let $\G\subset\cG$ denote the identity component of the
isotropy subgroup of the tuple $(s_{0,*},j_0)$. Thus 
\begin{equation}\label{eq:G}
\G := \left\{\begin{array}{ll}
\{\one\},&\mbox{if }\sn>2-2\sg, \\
\T^2,&\mbox{if }\sg=1,\,\sn=0, \\
\C^*,&\mbox{if }\sg=0,\,\sn=2, \\
\C^*\ltimes\C,&\mbox{if }\sg=0,\,\sn=1, \\
\PSL(2,\C),&\mbox{if }\sg=0,\,\sn=0. 
\end{array}\right.
\end{equation}
First we choose a $\G$-invariant holomorphic map 
$$
\iota_0:A\to(\Sigma^\sn\setminus\Delta)\times\cJ(\Sigma),\qquad
\iota_0(a)=(\sigma_1(a),\dots,\sigma_\sn(a),j(a)),
$$
defined on an open neighborhood $A\subset\C^{3\sg-3+\sn+\dim_\C\G}$
of a point $a_0$, that is transverse to the $\cG$-action and satisfies 
$$
\iota_0(a_0)=(s_{0,*},j_0).
$$
We do this as follows.  In the case $\sn>2-2\sg$ we choose a 
slice in Teichm\"uller space $\cT_{\sg,\sn}$ as in the proof 
of~\cite[Theorem~8.9]{RS}. There are two cases with $\sn\le 2-2\sg$.  
If $\sg=1$ (so $\Sigma\cong\T^2$) and $\sn=0$ 
we take $A=\H$ to be the upper half plane and define 
$\iota_0:A\to\cJ(\Sigma)$ as the standard map to the complex 
structures on the torus (see~\cite[Section~7]{RS}).  
If $\sg=0$ (so $\Sigma\cong S^2$) and $\sn\le 2$ 
we take $A$ to be a point. Note that
\begin{equation}\label{eq:dimAG}
\dim_\C A - \dim_\C\G = 3\sg-3+\sn
\end{equation}
in all cases and that $\G$ is the isotropy group of
each element of the slice, i.e. for $g\in\cG$ and $a\in A$ 
we have $g^*\iota_0(a)=\iota_0(a)$ if and only if $g\in\G$. 

The map $\iota_0$ gives rise to an infinite dimensional
vector bundle 
$$
\cE\to A\times\Cinf(\Sigma,M)
$$
with fibers
$$
\cE_{a,v} := \Omega^{0,1}_{j(a)}(\Sigma,v^*TM).
$$
The Cauchy--Riemann operator defines a section 
\begin{equation}\label{eq:dbarjJ}
A\times\Cinf(\Sigma,M)\to\cE:(a,v)\mapsto\bar\p_{j(a),J}(v)
\end{equation}
whose intrinsic derivative at a point $(a,v)$ is the operator
$$
\cD_{a,v}:T_aA\times\Omega^0(\Sigma,v^*TM)
\to\Omega^{0,1}_{j(a)}(\Sigma,v^*TM)
$$
given by 
\begin{equation}\label{eq:Dav}
\cD_{a,v}(\hat a,\hat v) 
:= \cD_{j(a),v}(dj(a)\hat a,\hat v) 
= D_v\hat v + \frac12J(v)dv\cdot dj(a)\hat a.
\end{equation}
Since the operator $\cD_{j_0,v_0}$ is surjective and $\iota_0$
is an infinitesimal slice, it follows that the section~(\ref{eq:dbarjJ}) 
is transverse to the zero section at $(a_0,v_0)$.  Hence it follows 
from the implicit function theorem in suitable Sobolev completions 
(see e.g.~\cite[Chapter~3]{MS}) that a neighborhood of $(a_0,v_0)$
in the zero set of~(\ref{eq:dbarjJ}) is a smooth submanifold of 
$A\times\Cinf(\Sigma,M)$. It is denoted by
$$
Z := \left\{(a,v)\in A\times\Cinf(\Sigma,M)\,\Big|\,
\bar\p_{j(a),J}(v)=0,\,\sup_{z\in\Sigma}d_M(v(z),v_0(z))<\eps\right\}.
$$
The group $\G$ acts on $Z$.
Since 
$$
\INDEX_\R(D_v) = \sm(2-2\sg) + 2\inner{c_1}{\sd}
$$
by the Riemann--Roch theorem, it follows from~(\ref{eq:dimAG})
that
$$
\dim_\R Z - \dim_\R\G
= (\sm-3)(2-2\sg) + 2\inner{c_1}{\sd} + 2\sn.
$$
Since $\iota$ is holomorphic and $J$ is integrable,
the operator~(\ref{eq:Dav}) is complex linear for all 
$(a,v)\in Z$. This shows that $Z$ is a finite dimensional
submanifold of ${A\times\Cinf(\Sigma,M)}$ whose tangent
space at each point $(a,v)\in Z$ is a complex subspace
of $T_aA\times\Omega^0(\Sigma,v^*TM)$.  
The almost complex structure on any such submanifold 
is integrable, because $\Cinf(\Sigma,M)$ is a complex 
manifold and the graph of a smooth function between 
complex vector spaces is a complex submanifold 
if and only if the function is holomorphic.
With this understood we obtain the desired
infinitesimal slice from a holomorphic slice $B\subset Z$ 
for the $\G$ action.  This proves the theorem.
\end{proof}

\rmk
In the proof of part~(iii) of Theorem~\ref{thm:slice}
one can reduce the case $\sn\le2-2\sg$ with $\G\ne\{\one\}$
to the case $\sn>2-2\sg$ with $\G=\{\one\}$ by a similar
argument as we used in the proof of Theorem~\ref{thm:exists}.
\kmr

\para\label{slice-unfolding}
Let $(s_{0,*},j_0,v_0)\in\cP$, 
$B$ be a manifold with base point $b_0\in B$, and
$$
B\to\cP:b\mapsto\iota(b)
=\left(\sigma_1(b),\dots,\sigma_\sn(b),j(b),v(b)\right)
$$
be a holomorphic map such that
$$
j(b_0)=j_0,\qquad v(b_0)=v_0,\qquad 
\sigma_\si(b_0)=s_{0,\si},\qquad \si=1,\dots,\sn.
$$
Define the unfolding $(\pi_\iota:Q_\iota\to B,S_{\iota,*},H_\iota,b_0)$ by
$$
Q_\iota := B\times\Sigma,\qquad
J_\iota(b,z) := \left(\begin{array}{cc}
      \sqrt{-1} & 0 \\ 0 & j(b)(z)\end{array}\right)
$$
where $\sqrt{-1}$ denotes the complex structure on $B$ and
$$
H_\iota(b,z):=v(b)(z),\qquad
S_{\iota,\si}:=\left\{(b,\sigma_\si(b))\,|\,b\in B\right\},\qquad
\si=1,\dots,\sn.
$$
\arap

\begin{lemma}\label{le:slice}
Let $(\pi_\iota,S_{\iota,*},H_\iota,b_0)$ be the unfolding 
associated to a holomorphic map $\iota:B\to\cP$ 
as in~\ref{slice-unfolding}. Then the following are equivalent.
\begin{description}
\item[(i)]
The unfolding $(\pi_\iota,S_{\iota,*},H_\iota,b_0)$
is infinitesimally universal.
\item[(ii)]
The map $\iota$ is an infinitesimal slice  at $b_0$.
\end{description}
\end{lemma}

\begin{proof}
Let $u_0:(\Sigma,j_0)\to Q_\iota$ be the holomorphic embedding
$$
u_0(z):=(b_0,z)
$$
so that $H_\iota\circ u_0=v_0$.
Then the operator $D_{u_0}$ has domain
$$
\cX_u:=\left\{(\hat u,\hat b)
\in\Omega^0(\Sigma,T\Sigma)\times T_{b_0}B\,|\,
\hat u(s_{0,\si})=d\sigma_\si(b_0)\hat b\right\},
$$
target space $\cY_u:=\Omega^{0,1}_{j_0}(\Sigma,T\Sigma)$,
and is given by
$$
D_{u_0}(\hat u,\hat b) 
= \dbar_{j_0}\hat u -\frac12 j_0dj(b_0)\hat b.
$$
The linearized operator in~\ref{Dv} is
$$
D_{v_0}:\cX_v\to\cY_v,\qquad
\cX_v:=\Omega^0(\Sigma,v_0^*TM),\qquad
\cY_v:=\Omega^{0,1}(\Sigma,v_0^*TM).
$$
The homomorphisms
\begin{equation}\label{eq:kercoker}
\ker D_{u_0}\to\ker D_{v_0},\qquad
\coker D_{u_0}\to\coker D_{v_0}
\end{equation}
are induced by the maps
$$
\cX_u\to\cX_v:(\hat u,\hat b)
\mapsto dv_0\cdot\hat u + dv(b_0)\hat b,\qquad
\cY_u\to\cY_v:\eta\mapsto dv_0\cdot \eta. 
$$
We must prove that the maps in~(\ref{eq:kercoker}) are
isomorphisms if and only if~(ii) holds. 
Note that the second 
map in~(\ref{eq:kercoker}) is necessarily surjective
because $(\Sigma,s_{0,*},j_0,v_0)$ is a regular stable map. 

\smallbreak

We prove that~(ii) implies~(i). 
We prove that  the first map in~(\ref{eq:kercoker})
is bijective. Let $(\hat u,\hat b)\in\ker D_{u_0}$ and 
assume that its image in $\ker D_{v_0}$ vanishes.  Then 
$$
\dbar_{j_0}\hat u - \frac12 j_0dj(b_0)\hat b=0,\qquad
dv_0\cdot \hat u + dv(b_0)\hat b = 0.
$$
Since $(\hat u,\hat b)\in\cX_u$, we have $d\sigma_\si(b_0)=\hat u(s_{0,\si})$
for $\si=1,\dots,\sn$ and hence, by~(ii) and~($\dagger$) in~\ref{slice},  
$\hat b=0$ and $\hat u=0$. Thus we have proved that the 
homomorphism $\ker D_{u_0}\to \ker D_{v_0}$ is injective.
Next we prove that  this map is surjective. 
Let $\hat v\in\Omega^0(\Sigma,v_0^*TM)$ be a vector 
field along $v_0$ such that ${D_{v_0}\hat v=0}$. 
Then the tuple $(\hat s_1,\dots\hat s_\sn,\jhat,\hat v)$
with $\hat s_\si=0$ and $\jhat=0$ satisfies~(\ref{eq:TP}).
Hence, by~(ii) and~($\ddagger$) in~\ref{slice}, 
there is a pair $(\hat u,\hat b)$ such that 
$$
d\sigma_\si(b_0)\hat a - \hat u(s_{0,\si}) =0,\qquad
dj(b_0)\hat b + 2j_0\dbar_{j_0}\hat u = 0,\qquad
dv(b_0)\hat b + dv_0\cdot\hat u = \hat v.
$$
This implies
$$
\dbar_{j_0}\hat u - \frac{1}{2}j_0dj(b_0)\hat b =0,\qquad
\hat v = dv(b_0)\hat b + dv_0\cdot\hat u
$$
and so $\hat v$ belongs to the image of the map 
$\ker D_{u_0}\to \ker D_{v_0}$.  This shows that 
the first map in~(\ref{eq:kercoker}) is an isomorphism. 

Next we prove that the second map in~(\ref{eq:kercoker}) 
is bijective. Let $\eta\in\cY_u$ such that
$dv_0\cdot\eta\in\mathrm{im}\,D_{v_0}$ and choose 
$\hat v\in\Omega^0(\Sigma,v_0^*TM)$ such that 
$$
dv_0\cdot\eta + D_{v_0}\hat v = 0.
$$
Then $\hat v$ and $\jhat:=-2j_0\eta$ satisfy~(\ref{eq:TP}). 
Hence, by~(ii) and~($\ddagger$) in~\ref{slice}, 
there is a pair $(\hat u,\hat b)$ such that 
$$
d\sigma_\si(b_0)\hat b - \hat u(s_{0,\si}) =0,\qquad
dj(b_0)\hat b + 2j_0\dbar_{j_0}\hat u = \jhat,\qquad
dv(b_0)\hat b + dv_0\cdot\hat u = \hat v. 
$$
This implies
$$
(\hat u,\hat b)\in\cX_u,\qquad
D_{u_0}(\hat u,\hat b)=-\frac12j_0\jhat=-\eta,
$$
and hence $\eta\in\mathrm{im}D_{u_0}$.
This shows that the second map in~(\ref{eq:kercoker}) 
is injective and, since we have already proved surjectivity,
it is an isomorphism. Thus we have proved that~(ii) implies~(i).

We prove that~(i) implies~(ii). 
Assume that the maps in~(\ref{eq:kercoker})
are bijective. If $\hat u$ and $\hat b$
satisfy~(\ref{eq:slice1}) then $(\hat u,\hat b)\in\cX_u$, 
$D_{u_0}(\hat u,\hat b)=0$, and the image of $(\hat u,\hat b)$
under the homomorphism $\cX_u\to\cX_v$ vanishes.
Since the first map in~(\ref{eq:kercoker}) is injective,
this implies $\hat u=0$ and $\hat b=0$. 
Now suppose that $\jhat$ and $\hat v$ 
satisfy~(\ref{eq:TP}) with $v=v_0$, i.e. 
$$
0=D_{v_0}\hat v + \frac12J(v_0)dv_0\circ\jhat 
= D_{v_0}\hat v + dv_0\circ\eta,\qquad \eta := \frac12j_0\jhat.
$$
Hence
$
dv_0\circ \eta=-D_{v_0}\hat v\in\mathrm{im}D_{v_0}. 
$
Since the second map in~(\ref{eq:kercoker}) is injective
this implies $\eta\in\mathrm{im}D_{u_0}$.  Choose a pair
$(\hat u,\hat b)\in\cX_u$ such that $D_{u_0}(\hat u,\hat b)=-\eta$.
Then $\hat u$ and $\hat b$ satisfy 
$$
d\sigma_\si(b_0)\hat b-\hat u(s_{0,\si})=0,\qquad 
\jhat = -2j_0\eta = 2j_0D_{u_0}(\hat u,\hat b)
= 2j_0\bar\p_{j_0}\hat u + dj(b_0)\hat b.
$$
Hence 
$$
D_{v_0}\hat v = -dv_0\cdot \eta 
=  dv_0\cdot D_{u_0}(\hat u,\hat b)
= D_{v_0}\left(dv_0\cdot\hat u + dv(b_0)\hat b\right).
$$
The last equation follows from the fact that 
and the diagram~(\ref{eq:XY}) in Definition~\ref{def:infuniv} 
commutes, reading $H_\iota(p,z)=v(b)(z)$ for $H$. 
Since the first map in~(\ref{eq:kercoker}) is surjective, there 
exists a pair $(\hat u_0,\hat b_0)\in\ker D_{u_0}$ such that 
$$
\hat v = dv_0\cdot(\hat u+\hat u_0) + dv(b_0)(\hat b+\hat b_0).
$$
Hence the pair $(\hat u+\hat u_0,\hat b+\hat b_0)$ 
satisfies~(\ref{eq:slice2}) with $\hat s_\si=0$.  
In the case $\hat s_\si\ne 0$ choose first a vector field
$\hat u_0\in\Vect(\Sigma)$ such that $-\hat u_0(s_{0,\si})=\hat s_\si$
for $\si=1,\dots,\sn$ and denote
$$
\jhat_1 := \jhat - 2j_0\bar\p_{j_0}\hat u_0,\qquad
\hat v_1 := \hat v - dv_0\cdot \hat u_0.
$$
This pair still satisfies~(\ref{eq:TP}).  Hence, by what we have 
already proved, there exists a pair $(\hat u_1,\hat b_1)$ that
satisfies~(\ref{eq:slice2}) with $(\hat s_\si,\jhat,\hat v)$ 
replaced by $(0,\jhat_1,\hat v_1)$.  Hence the pair 
$\hat u:=\hat u_0+\hat u_1$, $\hat b:=\hat b_1$ satisfies~(\ref{eq:slice2}).
Thus we have proved that~(i) implies~(ii). This completes
the proof of the lemma. 
\end{proof}

\begin{lemma}\label{le:holomorphic}
Fix a regular stable map $(\Sigma,s_{0,*},j_0,v_0)$ and let
$$
B\to\cP:b\mapsto\iota(b)=(\sigma_1(b),\dots,\sigma_\sn(b),j(b),v(b))
$$ 
be a holomorphic infinitesimal slice such that 
$$
\iota(b_0)=(s_{0,*},j_0,v_0).
$$
Let $(\pi_\iota,S_{\iota,*},H_\iota,b_0)$
be the unfolding constructed in~\ref{slice-unfolding}. 
Then every continuously differentiable morphism $(\phi,\Phi)$ from 
$(\pi_A:P\to A,R_*,H_A,a_0)$ to $(\pi_\iota,S_{\iota,*},H_\iota,b_0)$ 
is holomorphic.
\end{lemma}

\begin{proof}
Choose a smooth trivialization 
$$
A\times\Sigma\to P:(a,z)\mapsto\tau(a,z)=\tau_a(z)
$$
so that $\tau_a:\Sigma\to P_a$ is a desingularization
(with no singularities) for every $a\in A$. The stable map 
on $\Sigma$, induced by $\tau_a$, is the tuple 
$$
p(a) := \iota\circ\phi(a) 
= \bigl(\sigma_1(\phi(a)),\dots,\sigma_\sn(\phi(a)),j(\phi(a)),v(\phi(a))\bigr)
\in \cP. 
$$
The complex structure on $A\times\Sigma$ induced by $\tau$ 
has the form
$$
(\hat a,\hat z)\mapsto\left(\sqrt{-1}\hat a, 
j(\phi(a))(z)\hat z+\eta(a,\hat a)(z)\right)
$$
for a suitable $1$-form $T_aA\to\Vect(\Sigma):\hat a\mapsto\eta(a,\hat a)$.
Since this complex structure is integrable, the map
$H_A\circ\tau:A\times\Sigma\to M$ is holomorphic, and
$\tau^{-1}(R_\si)$ is a complex submanifold of $A\times\Sigma$
for every $\si$, it follows from Lemma~\ref{le:hol-unfolding} that 
$$
dp(a)\hat a + \cI(p(a))dp(a)\sqrt{-1}\hat a - \cL_{p(a)}\eta(a,\sqrt{-1}\hat a) = 0
$$
for every $a\in A$ and every $\hat a\in T_aA$.  
Since $p=\iota\circ\phi$ and $\iota$ is holomorphic, this implies
$$
d\iota(\phi(a))
\left(d\phi(a)\hat a + \sqrt{-1}d\phi(a)\sqrt{-1}\hat a\right)
= \cL_{\iota(\phi(a))}\eta(a,\sqrt{-1}\hat a)
$$
for all $a$ and $\hat a$.  Since $\iota$ is a slice this implies 
that $\eta\equiv0$ and $\phi$ is holomorphic.  Hence $\Phi$
is holomorphic as well and this proves the lemma. 
\end{proof}

\begin{theorem}\label{thm:nonodes}
Theorems~\ref{thm:exists}, \ref{thm:universal}, and~\ref{thm:stable}
hold for regular stable maps without nodes.  Moreover, 
if $(\pi_B:Q\to B,S_*,H_B,b_0)$ is any universal unfolding without 
nodes and $(\phi,\Phi)$ is a continuously differentiable 
morphism from ${(\pi_A:P\to A,R_*,H_A,a_0)}$ to $(\pi_B,S_*,H_B,b_0)$ 
then $\phi$ and $\Phi$ are holomorphic. 
\end{theorem}

\begin{proof} 
{\bf Step 1.}  
{\em Theorem~\ref{thm:exists} holds for stable maps without nodes. }
We proved ``only if''  immediately after the statement of 
Theorem~\ref{thm:exists}; we prove ``if'' here.
Fix a regular stable map $(\Sigma,s_{0,*},j_0,v_0)$, let
$\iota:B\to\cP$ be a holomorphic infinitesimal slice such that 
$
\iota(b_0)=(s_{0,*},j_0,v_0),
$
and let $(\pi_\iota,S_{\iota,*},H_\iota,b_0)$
be the unfolding constructed in~\ref{slice-unfolding}. 
Then it follows from Lemma~\ref{le:slice} that
$(\pi_\iota,S_{\iota,*},H_\iota,b_0)$ is infinitesimally universal. 

{\bf Step 2.}   
{\em The unfolding $(\pi_\iota,S_{\iota,*},H_\iota,b_0)$ is universal.}
Let $(\pi_A,R_*,H_A,a_0)$ be an unfolding
of $(\Sigma,s_{0,*},j_0,v_0)$ and $f_0:P_{a_0}\to Q_{b_0}$
be a fiber isomorphism.  Assume w.l.o.g.~that
$$
P=A\times\Sigma,\qquad f_0(a_0,z)=(b_0,z).
$$
Denote by $p(a)=(r_*(a),j(a),v(a))\in\cP$
the regular stable map on the fiber over $a$ determined by
$(\pi_A,R_*,H_A,a_0)$.  Then 
$$
p(a_0) = (s_{0,*},j_0,v_0) = \iota(b_0). 
$$
Now any two smooth maps $\phi:A\to B$ and 
$\Phi:P\to Q_\iota$ that intertwine the projections 
and satisfy $\Phi|P_{a_0}=f_0$ have the form
$$
\Phi(a,z) = (\phi(a),\Phi_a(z)),
$$
where $A\to\Diff(\Sigma):a\mapsto\Phi_a$ is a smooth map
such that 
$
\Phi_{a_0}=\id.  
$
The pair $(\phi,\Phi)$ is a 
smooth morphism from $(\pi_A,R_*,H_A,a_0)$
to $(\pi_\iota,S_{\iota,*},H_\iota,b_0)$
if and only if 
$$
p(a) = \Phi_a^*\iota(\phi(a))
$$
for every $a\in A$.  Hence the existence and uniqueness 
of smooth morphisms follows from the Theorem~\ref{thm:slice}~(i).
That every smooth morphism is  holomorphic follows from 
Lemma~\ref{le:holomorphic}.  

{\bf Step 3.}   
{\em Every infinitesimally universal unfolding 
of $(\Sigma,s_{0,*},j_0,v_0)$ is isomorphic to 
$(\pi_\iota,S_{\iota,*},H_\iota,b_0)$.}
Let $(\pi_A,R_*,H_A,a_0)$ be an unfolding
and 
$$
f_0:P_{a_0}\to Q_{b_0}
$$
be a fiber isomorphism. By Step~2, there exists 
a holomorphic morphism $(\phi,\Phi)$ from 
$(\pi_A,R_*,H_A,a_0)$ to $(\pi_\iota,S_{\iota,*},H_\iota,b_0)$.  
The map 
$$
p:=\iota\circ\phi:A\to\cP
$$ 
is holomorphic. Since $(\pi_A,R_*,H_A,a_0)$ is infinitesimally 
universal, $p$ is a infinitesimal slice at $a_0$, by Lemma~\ref{le:slice}.
Hence the differential $d\phi(a_0)$ is bijective.  This implies that
$(\phi,\Phi)$ is an isomorphism. 

{\bf Step 4.}   
Since every infinitesimally universal unfolding 
of $(\Sigma,s_{0,*},j_0,v_0)$ is isomorphic to 
$(\pi_\iota,S_{\iota,*},H_\iota,b_0)$ and  $(\pi_\iota,S_{\iota,*},H_\iota,b_0)$
is universal we have proved Theorem~\ref{thm:universal} for 
stable maps without nodes. By
Lemma~\ref{le:slice} and Theorem~\ref{thm:slice}, 
the unfolding $(\pi_\iota,S_{\iota,*},H_\iota,b)$
is infinitesimally universal for $b$ near $b_0$ and hence
Theorem~\ref{thm:stable} holds for stable maps without nodes. 
The `moreover' assertion follows from Lemma~\ref{le:holomorphic}
and Step~3.   This  proves Theorem~\ref{thm:nonodes}.
\end{proof}


\section{Hardy decompositions}\label{sec:Hardy}

This section follows closely Sections~9 and~11 of~\cite{RS}.   
It is convenient to use slightly different notation;
for example $P=N\cup M$ in~\cite{RS} becomes $P=P'\cup P''$
and the open sets $U,V\subset Q$ in~\cite{RS} are replaced by $U',U''$. 
With these changes we review the notation from~\cite{RS}.

\para\label{AB}
Throughout this section
$$
(\pi_A:P\to A,R_*,H_A,a_0),\qquad (\pi_B:Q\to B,S_*,H_B,b_0)
$$  
are unfoldings of maps,
$$
f_0: P_{a_0}\to Q_{b_0}
$$
is a fiber isomorphism, and $p_1,p_2,\ldots,p_\sk$ are
the nodal points of the central fiber~$P_{a_0}$,
so $q_\si:=f_0(p_\si)$ (for $\si=1,\ldots,\sk$)
are the nodal points of the central fiber~$Q_{b_0}$.
As in~\cite{RS} we denote by $C_A\subset P$ and 
$C_B\subset Q$ the critical points of $\pi_A$ and $\pi_B$, 
respectively. 
\arap

\para\label{UQ}
Let $U'\subset Q$ be an open neighborhood of $C_B$ 
equipped with nodal coordinates.  This means
$$
    U'=U'_1\cup\cdots\cup U'_\sk
$$
where the sets $U'_\si$ have pairwise disjoint closures,
each $U'_\si$ is a connected neighborhood of one of the components
of $C_B$, and for $\si=1,\ldots,\sk$ there is a holomorphic coordinate
system
$$
(\zeta_\si,\tau_\si):B\to \C\times\C^{\sb-1}  
$$
and holomorphic functions $\xi_\si,\eta_\si:U'_\si\to\C$ such that
$$
(\xi_\si,\eta_\si,\tau_\si\circ\pi_B):U'_\si\to\C\times\C\times\C^{\sb-1}
$$
is a holomorphic coordinate system and $\xi_\si\eta_\si=\zeta_\si\circ\pi_B$.
Assume that $\bar U'\cap S_*=\emptyset$.
Let $U''\subset Q$ be an open set such that
$$
Q=U'\cup U'', \qquad \bar U''\cap C_B=\emptyset,
$$
and $U_\si'\cap U''$ intersects each fiber $Q_b$ in two open annuli
with $|\xi_\si|>|\eta_\si|$ on one component and $|\xi_\si|<|\eta_\si|$ on the other.
Introduce the abbreviations
$$
U:=U'\cap U'', \quad U_\si:=U_\si'\cap U'', \quad U_{\si,1}:=\{|\xi_\si|>|\eta_\si|\},
\quad U_{\si,2}:=\{|\xi_\si|<|\eta_\si|\},
$$
$$
U'_b:= U'\cap Q_b, \qquad  U''_b:=U''\cap Q_b, \qquad U_b:=U\cap Q_b.
$$
\arap

\para\label{Hardy}
As in~\cite{RS} we use a Hardy decomposition
$$
P=P'\cup P'', \qquad \p P'=\p P'' = P'\cap P'',
$$
for $(\pi_A,R_*,a_0)$. Thus
$P'$ and $P''$ are submanifolds of $P$
intersecting in their common  boundary and
$$
P'=P'_1\cup\cdots \cup P'_\sk,
$$
where $P'_\si$ is a closed neighborhood of $p_\si$
disjoint from the elements of $R_*$,
 the $P'_\si$ are pairwise disjoint,
and each $P'_\si$  is the domain of a nodal coordinate system.
The latter consists of three  holomorphic maps
$$
(x_\si,y_\si):P'_\si\to\D^2,\qquad z_\si: A \to\C,\qquad t_\si:A\to\C^{\sa-1},
$$
such that each map
$$
A\to   \D\times\C^{\sa-1}:a\mapsto(z_\si(a),t_\si(a))
$$
is a holomorphic coordinate system, each  map
$$
P'_\si\to \D^2\times \C^{\sa-1}:
p\mapsto \bigl(x_\si(p),y_\si(p),t_\si(\pi_A(p))\bigr)
$$
is a holomorphic coordinate system, and
$$
x_\si(p_\si)=y_\si(p_\si)=0,\qquad    z_\si\circ\pi_A =x_\si y_\si.
$$
Restricting to a fiber gives a decomposition
$$
P_a=P'_a\cup P''_a, \qquad
P'_a:=P'\cap P_a, \qquad
P''_a:=P''\cap P_a,
$$
where $P''_a$ is a Riemann surface with boundary
and each component of $P'_a$ is either
a closed annulus or a pair of transverse closed disks.
Abbreviate
$$
\Gamma_a:=P'_a\cap P''_a = \p P'_a = \p P''_a.
$$
The nodal coordinate system determines a trivialization
\begin{equation}\label{eq:trivialize}
\iota:A\times\Gamma\to\p P',\qquad 
\Gamma:=\bigcup_{\si=1}^\sk\{(\si,1),(\si,2)\}\times S^1,
\end{equation}
given by 
$$
\begin{array}{ll}
\iota^{-1}(p):=(\pi_A(p),(\si,1),x_\si(p)),&\qquad p\in\p_1P'_\si:=\{|x_\si|=1\}, \\
\iota^{-1}(q):=(\pi_A(p),(\si,2),y_\si(q)),&\qquad q\in\p_2P'_\si:=\{|y_\si|=1\}.
\end{array}
$$
For $a\in A$ and $\si=1,\dots,\sk$
define $\iota_a:\Gamma\to\Gamma_a$ by
$\iota_a(\lambda):=\iota(a,\lambda)$ and denote
$$
\p_{\si,1}P'_a:=\p_1P'_\si\cap P_a,\qquad
\p_{\si,2}P'_a:=\p_2P'_\si\cap P_a,\qquad
P'_{a,\si}:=P'_a\cap P'_\si.
$$
\arap

\begin{figure}[htp] 
\centering  
\includegraphics[scale=0.6]{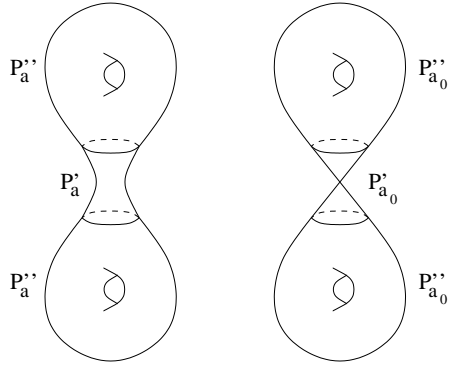} 
\caption{{A Hardy decomposition of $P$.}}\label{fig:hardy} 
\end{figure} 

\para\label{f=id}
Lemma~11.3 in~\cite{RS} asserts that, 
after shrinking $A$ and $B$ if necessary,
there is a Hardy decomposition $P=P'\cup P''$ as
in~\ref{Hardy} and there are open
subsets $U'=U'_1\cup\cdots\cup U'_\sk$, $U''$, $U$ of $Q$
and functions $\xi_\si,\eta_\si,\zeta_\si,\tau_\si$ as described in~\ref{UQ}
such that
$$
f_0(P'_{a_0})\subset U'_{b_0},\qquad f_0(P''_{a_0})\subset U''_{b_0},
$$
$$
\xi_\si\circ f_0\circ x_\si^{-1}(x,0,0)=x,\qquad
\eta_\si\circ f_0\circ y_\si^{-1}(0,y,0)=y
$$
for $x,y\in\D$. Fix a Hardy decomposition
$P=P'\cup P''$ for $(\pi_A,R_*,a_0)$,
open subsets $U'=U'_1\cup\cdots\cup U'_\sk$, $U''$, $U$ of $Q$,
and functions $\xi_\si,\eta_\si,\zeta_\si,\tau_\si$ as described in~\ref{UQ},
such that these conditions are satisfied.
\arap

\para\label{cU}
Fix an integer $s+1/2>1$.  For $a\in A$ and $b\in B$ define an  open subset
$$
\cU(a,b)\subset H^s(\Gamma_a,U_b)
$$
by the condition that for $\alpha\in H^s(\Gamma_a,U_b)$
we have $\alpha\in\cU(a,b)$ if
$$
\alpha\bigl(\p_{\si,1}P'_a\bigr)\subset U_{\si,1},\qquad
\alpha\bigl(\p_{\si,2}P'_a\bigr)\subset U_{\si,2},
$$
(see~\ref{UQ} for the notation $U_{\si,1}$ and $U_{\si,2}$)
and the curves $\xi_\si\circ \alpha\circ x_\si^{-1}$ and
$\eta_\si\circ \alpha\circ y_\si^{-1}$ from $S^1$ to $\C\setminus0$
both have winding number one about the origin.
\begin{equation*}
\begin{split}
\cU'(a,b)&:=
\left\{\alpha\in\cU(a,b) \,\Bigg|\,
\begin{aligned}
& \exists f'\in \Hol^{s+1/2}(P'_a,U'_b) : \alpha=f'|\Gamma_a\\
& \mbox{and } f'(C_A\cap P_a)=C_B\cap Q_b,
\end{aligned}\right\},  \\ \\
\cU''(a,b)&:=
\left\{\alpha\in\cU(a,b)\,\Bigg|\,
\begin{aligned}
& \exists f''\in \Hol^{s+1/2}(P''_a,U''_b):\alpha=f''|\Gamma_a\\
& \mbox{and } f''(R_*\cap P_a)=S_*\cap Q_b
\end{aligned}\right\}.
\end{split}
\end{equation*}
Here $\Hol^{s+1/2}(X,Y)$ denotes the set of maps
of class $H^{s+1/2}$ from $X$ to $Y$ which are holomorphic
on the interior of $X$. Holomorphicity at a nodal point
is defined as in~\cite[\S11.1]{RS}. 
Note that the function
$f':P'_a\to U'_b$ in the definition of $\cU'(a,b)$ maps the 
boundary $\Gamma_a=\p P'_a$ into $U_b=U'_b\cap U''_b$; 
similarly for $f''$ in the definition of $\cU''(a,b)$. 
Define
$$
\cU_a:=\bigsqcup_{b\in B}\cU(a,b),\qquad
\cU'_a:=\bigsqcup_{b\in B}\cU'(a,b),\;\qquad
\cU''_a:=\bigsqcup_{b\in B}\cU''(a,b),
$$
$$
\cU:=\bigsqcup_{a\in A}\cU_a,\qquad
\cU':=\bigsqcup_{a\in A}\cU'_a,\qquad
\cU'':=\bigsqcup_{a\in A}\cU''_a.
$$
Our notation means that the three formulas
$(a,\alpha,b)\in\cU$, $(\alpha,b)\in\cU_a$,
and $\alpha\in\cU(a,b)$ have the same meaning.
\arap

\para\label{cU0}
We use the nodal coordinate system of~\ref{Hardy} to 
construct an auxiliary Hilbert manifold structure on $\cU$.  
The domains of the maps in this space vary with $a$ 
so we replace them with a constant domain by using 
an appropriate trivialization.   Define an open set
$$
\cU_0\subset \left\{(a,\alpha,b)\in A\times H^s(\Gamma,U)\times B\,|\,
\pi_B\circ\alpha=b\right\}
$$
by the condition that the map
$$
\cU_0\to\cU:(a,\alpha,b)\mapsto(a,\alpha\circ\iota_a^{-1},b)
$$
is a bijection.  In particular $\alpha((\si,1)\times S^1)\subset U_{\si,1}$
and $\alpha((\si,2)\times S^1)\subset U_{\si,2}$ for $(a,\alpha,b)\in\cU_0$.
(By a standard construction $H^s(\Gamma,U)$ is a complex
Hilbert manifold and the subset $\{(a,\alpha,b)\,|\,\pi_B\circ\alpha=b\}$
is a complex Hilbert submanifold of $A\times H^s(\Gamma,U)\times B$.
This is because the map $H^s(\Gamma,U)\to H^s(\Gamma,B)$ 
induced by $\pi_B$ is a holomorphic submersion.
Note that $\cU_0$ is a connected component
of $\{(a,\alpha,b)\,|\,\pi_B\circ\alpha=b\}$ and hence inherits
its Hilbert manifold structure.)
We emphasize that the resulting Hilbert manifold structure
on $\cU$ depends on the choice of the Hardy trivialization.
Two different Hardy trivializations give rise to a homeomorphism
which is of class $C^\ell$ on the dense subset $\cU\cap H^{s+\ell}$.
\arap

\para\label{shrinking}
The fiber isomorphism $f_0:P_{a_0}\to Q_{b_0}$ determines
a point
$$
(a_0,\alpha_0:=f_0|\Gamma_{a_0},b_0)\in\cU;
$$
this point lies in $\cU'\cap\cU''$ as
$$
\alpha_0=f'_0|\Gamma_{a_0}=f''_0|\Gamma_{a_0},
\qquad\mbox{where}   \qquad
f'_0:=f_0|P'_{a_0},\quad f''_0:=f_0|P''_{a_0}.
$$
In the sequel we will denote neighborhoods of $a_0$ in $A$ and
$(a_0,\alpha_0,b_0)$ in $\cU'$, $\cU''$, or $\cU$ by the same letters
$A$, respectively $\cU'$,  $\cU''$, or $\cU$, and signal this with
the text ``shrinking $A$, $\cU'$,  $\cU''$, or $\cU$, if necessary''.
\arap

\begin{lemma}\label{le:fgamma}
For every $(a,\alpha,b)\in\cU'\cap\cU''$
there is a unique fiber isomorphism
$f:P_a\to Q_b$ with $f|\Gamma_a=\alpha$.
\end{lemma}

\begin{proof}
This follows immediately from~\cite[Lemma~9.4]{RS}.
\end{proof}

\begin{theorem}\label{thm:U}
Fix an integer $s+1/2>4$.
After shrinking $A$, $\cU'$, $\cU''$, $\cU$,
if necessary, the following holds.
\begin{description}
\item[(i)]
For each $a\in A$, $\cU'_a$ and $\cU''_a$ 
are complex submanifolds of~$\cU_a$.
\item[(ii)]
Let $(a,\alpha,b)\in\cU'\cap\cU''$ and $f:P_a\to Q_b$ be the
associated fiber isomorphism with $\alpha=f|\Gamma_a$. 
Let $w:\Sigma\to P_a$ be a desingularization with induced 
structures $j,\nu,s_*,u:=f\circ w$ on~$\Sigma$ and $D_u$ 
be the operator in Definition~\ref{def:infuniv}. 
Then 
$$
\ker D_u\cong
T_{(\alpha,b)}\cU'_a\cap T_{(\alpha,b)}\cU''_a,\qquad
\coker D_u\cong
\frac{T_{(\alpha,b)}\cU_a}
{T_{(\alpha,b)}\cU'_a+T_{(\alpha,b)}\cU''_a}.
$$
\item[(iii)]
$\cU'$ and $\cU''$ are  complex submanifolds of $\cU$.
\item[(iv)]
The projections $\cU\to A$, $\cU'\to A$, $\cU''\to A$ 
are holomorphic submersions.
\end{description}
\end{theorem}

\begin{proof}
Theorems~9.5 and~11.9 in~\cite{RS}.
The condition $s+1/2>4$ is used in compactness arguments
for the proofs of~(i) and~(iii). These compactness arguments
can be eliminated by modifying the definition of $\cU''$ along
the lines of the definition of $\cV''$ in~\ref{cV} below.
\end{proof}

\para\label{hardyTrivialization}
As in~\cite[Definition~11.6]{RS}, we use a Hardy trivialization for  
$(\pi_A:P\to A,R_*,a_0)$, i.e. a triple $(P'\cup P'',\iota,\rho)$ 
where $P=P'\cup P''$ is a Hardy decomposition with 
corresponding trivialization $\iota:A\times\Gamma\to \p P'$
as in~\ref{Hardy} and
$$
\rho:P''\to P''_{a_0}=:\Omega
$$
is a trivialization such that $\rho_a:=\rho|P''_a:P''_a\to\Omega$ 
is a diffeomorphism satisfying
$$
\rho_{a_0} = \id,\qquad
\rho_a\circ\iota_a=\iota_{a_0}
$$
for $a\in A$.  We require further that $\rho$ is holomorphic 
in a neighborhood of the boundary. 
\arap

\para\label{cV}
Let $(\pi_A:P\to A,R_*,a_0)$ be an unfolding of marked nodal 
Riemann surfaces and $h_0:P_{a_0}\to M$ be a holomorphic map. 
Choose a Hardy decomposition $P=P'\cup P''$ as in~\ref{Hardy}
and a Hardy trivialization $\rho_a:P_a\to\Omega$ as 
in~\ref{hardyTrivialization}.
We would like to imitate Theorem~\ref{thm:U}
and define subsets  $\cV'_a,\cV''_a\subset  H^s(\Gamma_a,M)$
of those maps $\beta\in\cV_a$ which extend holomorphically
to $P'_a,P''_a$ respectively, but it is convenient
to restrict the extensions. Let 
$$
V'=V'_1\cup\cdots\cup V'_\sk\subset M
$$ 
be an open neighborhood of the image $h_0(P_{a_0}\cap C_A)$ 
of the nodal set so that each pair $(V'_\si,h_0(p_\si))$ 
is holomorphically diffeomorphic to the open unit ball 
in $\C^\sm$ centered at origin, the closures of the 
sets $V'_\si$ are pairwise disjoint, and 
$$
h_0(P_{a_0}\cap P'_\si) \subset V'_\si.
$$
For $a\in  A$ abbreviate 
$$
\cV_a := H^s(\Gamma_a,M).
$$
Let $\cV'_a\subset\cV_a$ be the subspace of those 
$\beta$ that extend holomorphically to $P_a'$, {i.e.}
$$
\cV'_a:=\left\{\beta\in\cV_a\,\big|\,\exists\,
h'\in\Hol^{s+1/2}(P'_a,M)\,s.t.\,h'(P'_{a,\si})\subset V'_\si
\mbox{ and }
\beta=h'|\Gamma_a\right\}.
$$
Let  $\cW_0$ be a neighborhood of $h_0|\Omega$ 
in $H^{s+1/2}(\Omega,M)$, where $\Omega=P_{a_0}''$ 
as in~\ref{hardyTrivialization}.  Via the trivialization
$\rho_a:P_a''\to \Omega$ this determines an open subset
$$
\cW_a:=\left\{h''\in H^{s+1/2}(P''_a,M)\,\big|\,
h''\circ\rho_a^{-1}\in\cW_0\right\}
$$ 
of $H^{s+1/2}(P''_a,M)$ for $a\in A$. Let 
$$
\cV''_a:=\left\{\beta\in\cV_a\,\big|\,\exists\,
h''\in\cW_a\cap\Hol^{s+1/2}(P''_a,M)\,s.t.\,
\beta=h''|\Gamma_a\right\}
$$
Define 
$$
\cV:=\bigsqcup_{a\in A}\cV_a,\qquad
\cV':=\bigsqcup_{a\in A}\cV'_a,\qquad
\cV'':=\bigsqcup_{a\in A}\cV''_a.
$$
Then every pair $(a,\beta)\in\cV'\cap\cV''$ determines
a holomorphic map $h:P_a\to M$ such that 
$h|\Gamma_a=\beta$. As in~\ref{cU0} we use the 
nodal coordinate system of~\ref{Hardy} 
to construct an auxiliary Hilbert 
manifold structure on $\cV$ via  the bijection
\begin{equation}\label{eq:Vtriv}
\cV\to A\times H^s(\Gamma,M):
(a,\beta)\mapsto (a,\beta\circ\iota_a).
\end{equation}
\arap

\begin{theorem}\label{thm:V}
Continue the notation of~\ref{Hardy}, \ref{hardyTrivialization},
and~\ref{cV}. Fix an integer $s+1/2>1$. After shrinking 
$A$ and $\cW_0$, if necessary, the following holds.
\begin{description}
\item[(i)]
For each $a\in A$, $\cV'_a$ and $\cV''_a$ 
are complex submanifolds of~$\cV_a$.
\item[(ii)]
Let $(a,\beta)\in\cV'\cap\cV''$ and 
$h:P_a\to M$ be the associated holomorphic map
with $\beta=h|\Gamma_a$. 
Let $w:\Sigma\to P_a$ be a desingularization with 
induced structures $s_*$, $\nu$, $j$, $v:=h\circ w$ 
on~$\Sigma$ and $D_v$ be the operator in~\ref{def:infuniv}. 
Then
$$
\ker D_v\cong
T_\beta\cV'_a\cap T_\beta\cV''_a,\qquad
\coker D_v\cong
\frac{T_\beta\cV_a}
{T_\beta\cV'_a+T_\beta\cV''_a}.
$$
\item[(iii)]
$\cV'$ and $\cV''$ are  complex submanifolds of $\cV$.
\item[(iv)]
The projections $\cV\to A$, $\cV'\to A$, $\cV''\to A$ 
are holomorphic submersions.
\end{description}
\end{theorem}

\begin{proof}[Proof of Theorem~\ref{thm:V}~(i) and~(ii).]
In parts~(i) and~(ii) the point $a$ is fixed.  
We introduce the following notation to make 
the proof look more like the proof of~\cite[Theorem~9.5]{RS}.
Use the notation of part~(ii).  Abbreviate
$$
\Sigma' := w^{-1}(P_a'),\qquad
\Sigma'' := w^{-1}(P_a'').
$$
Thus $\Sigma'$ and $\Sigma''$ are submanifolds of $\Sigma$
such that
$$
\Sigma=\Sigma'\cup\Sigma'',\qquad
\p\Sigma'=\p\Sigma''=\Sigma'\cap\Sigma''.
$$
Now $w^{-1}\circ\iota_a$ is a diffeomorphism from
$\Gamma$ in~(\ref{eq:trivialize}) to $\Sigma'\cap\Sigma''$.
To simplify the notation we assume that
$
\Gamma=\Sigma'\cap\Sigma''.
$
The submanifold $\Sigma'$ is a disjoint union
$$
\Sigma'=\Sigma'_1\cup\dots\cup \Sigma'_\sk
$$
where each set $\Sigma'_\si$ is either an embedded
closed annulus  or else the union of two disjoint embedded closed 
disks centered at two equivalent nodal points.  It follows that every 
pair of equivalent nodal points appears in some~$\Sigma'_\si$. 
In case $\Sigma'_\si$ is a disjoint union of two disks, say
$\Sigma'_\si=\Sigma'_{\si,1}\cup\Sigma'_{\si,2}$, choose holomorphic
diffeomorphisms $x_\si:\Sigma'_{\si,1}\to\D$ and $y_\si:\Sigma'_{\si,2}\to\D$
which send the nodal point to $0$.   In case  $\Sigma'_\si$  is an annulus 
choose a holomorphic diffeomorphism  $x_\si:\Sigma'_\si\to\A(\delta_\si,1)$  
and define  $y_\si:\Sigma'_\si\to\A(\delta_\si,1)$  by  $y_\si=\delta_\si/x_\si$.

Let $\cV'_0\subset H^s(\Gamma,M)$ be the subspace of those 
$H^s$-functions $\gamma:\Gamma\to M$ that extend
holomorphically to $H^{s+1/2}$-functions $v':\Sigma'\to M$
which map each pair of equivalent nodal points to the 
same point in $M$ and take  $\Sigma'_\si$ to $V'_\si$. 
Let $\cV''_0\subset H^s(\Gamma,M)$ be the subspace of 
those $H^s$-functions $\gamma:\Gamma\to M$ that extend
holomorphically to $H^{s+1/2}$-functions $v'':\Sigma''\to M$
such that 
$
h'':=v''\circ w^{-1}|P_a''\in\cW_a.
$
In this notation part~(i) asserts that $\cV_0'$ and 
$\cV_0''$ are complex submanifolds of~$H^s(\Gamma,M)$.

We prove that $\cV_0'$ is a complex submanifold of $H^s(\Gamma,M)$.
Choose coordinate charts $\psi_\si:V'_\si\to\C^\sm$ such that
$\psi_\si(H_A(p_\si))=0$ and $\psi_\si(V'_\si)$ is the 
open unit ball in $\C^\sm$ for every $\si$. Define the map 
\begin{equation}\label{eq:V0}
\cV_0'\to(H^s(S^1,\C^\sm))^{2\sk}:
\gamma\mapsto(\xi_1,\eta_1,
\dots,\xi_\sk,\eta_\sk)
\end{equation}
by 
\begin{equation}\label{eq:xietai}
\xi_\si:=\psi_\si\circ\gamma\circ x_\si^{-1},\qquad
\eta_\si:=\psi_\si\circ\gamma\circ y_\si^{-1}.
\end{equation}
The image of~(\ref{eq:V0}) is the set of all tuples
$(\xi_1,\eta_1,\dots,\xi_\sk,\eta_\sk)$ in 
$(H^s(S^1,\C^\sm))^{2\sk}$ that satisfy the following conditions.
\begin{description}
\item[(a)]
The functions $\xi_\si,\eta_\si:S^1\to\C^\sm$ 
take values in the open unit ball.
\item[(b)]
If $\Sigma'_\si$ is the disjoint union of two discs then all negative Fourier 
coefficients of $\xi_\si$ and $\eta_\si$ vanish and the zeroth coefficients agree.
\item[(c)]
If $\Sigma'_\si$ is an annulus then $\gamma_{\si,1}$ extends
holomorphically to an $H^{s+1/2}$ function on the annulus 
$\A(\delta_\si,1)$ and $\eta_\si(y) = \xi_\si(\delta_\si/y)$ 
for every $y\in S^1$.
\end{description}
Conditions~(b) and~(c) define a closed subspace of  
$(H^s(S^1,\C^\sm))^{2\sk}$ and condition~(a)
defines an open set in this subspace. Hence the image
of~(\ref{eq:V0}) is an open set in a Hilbert subspace and
this shows that $\cV_0'$ is a Hilbert submanifold of $H^s(\Gamma,M)$. 

We prove that $\Hol^{s+1/2}(\Sigma'',M)$ is a complex
submanifold of $H^{s+1/2}(\Sigma'',M)$. 
To see this note that the Cauchy--Riemann operator 
$v''\mapsto\bar\p_{j,J}(v'')$ 
defines a holomorphic section of the vector bundle 
$
\cE\to\cB:=H^{s+1/2}(\Sigma'',M)
$
with fibers
$$
\cE_{v''}:=H^{s-1/2}(\Sigma'',
\Lambda^{0,1}T^*\Sigma''\otimes(v'')^*TM)
$$
The intrinsic derivative of this section at a zero $v''$
is the Cauchy--Riemann operator 
$
D_{v''}:T_{v''}\cB\to\cE_{v''}
$
of the holomorphic vector bundle $(v'')^*TM\to\Sigma''$.
Since each component of $\Sigma''$ has nonempty boundary
the operator $D_{v''}$ is surjective; 
a right inverse can be 
constructed from an appropriate Lagrangian boundary condition
(see~\cite[Appendix~C.1.10]{MS}).   
This proves that 
$\Hol^{s+1/2}(\Sigma'',M)$ is a complex submanifold 
of $H^{s+1/2}(\Sigma'',M)$.  

We prove that $\cV_0''$ is a complex submanifold of $H^s(\Gamma,M)$.
The restriction map
$$
\Hol^{s+1/2}(\Sigma'',M)\to\cV_0:v''\mapsto v''|\Gamma
$$
is an injective holomorphic immersion.
That it is holomorphic is obvious, that it is injective follows 
from unique continuation, and that it is an immersion follows 
from the elliptic boundary estimate in~\cite[Theorem~B.4]{RS}.  
It follows that the image of a sufficiently small neighborhood
of $H_A\circ w|\Sigma''$ under the restriction map is 
a complex submanifold of $H^s(\Gamma,M)$; this image is $\cV_0''$. 
This proves~(i).

We prove~(ii).  It follows directly from the definitions that 
there is a map 
$$
\ker\,D_v\to T_\beta\cV_a'\cap T_\beta\cV_a'':
\hat v\mapsto \hat v\circ w^{-1}|\Gamma_a.
$$
As in the proof of Theorem~9.5~(ii) in~\cite{RS} this map is injective 
by unique continuation and is surjective by elliptic regularity.
Now define a map 
$$
\coker\,D_v\to\frac{T_\beta\cV_a}{T_\beta\cV_a'+T_\beta\cV_a''}:
[\eta]\mapsto[\hat\beta]
$$
as follows.  Given $\eta\in\Omega^{0,1}(\Sigma,v^*TM)$ choose two vector 
fields $\xi'$ along $v':=v|\Sigma'$ and $\xi''$ along $v'':=v|\Sigma''$
that satisfy
$$
D_{v'}\xi'=\eta|\Sigma',\qquad D_{v''}\xi''=\eta|\Sigma'',\qquad
\xi'|\Gamma-\xi''|\Gamma = \hat\beta\circ w|\Gamma.
$$
One verifies as in the proof of~\cite[Theorem~9.5~(iii)]{RS}
that this map is well defined and bijective.  That this map
is well defined follows directly from the definitions and that it
is injective uses elliptic regularity.  The proof of surjectivity
is based on the following two assertions. 
\begin{description}
\item[(a)]
Each element in the quotient 
$T_\beta\cV_a/(T_\beta\cV_a'+T_\beta\cV_a'')$
can be represented by a smooth vector field along $\beta$. 
\item[(b)]
For every smooth vector field $\hat\beta$ along $\beta$ there exist 
vector fields $\xi'$ along $v'$ and $\xi''$ along $v''$ such that
$\xi'|\Gamma-\xi''|\Gamma = \hat\beta\circ w|\Gamma$
and the $(0,1)$-form $\eta$ along $v$ defined by 
$\eta|\Sigma':=D_{v'}\xi'$ and $\eta|\Sigma'':=D_{v''}\xi''$
is smooth.
\end{description}
One first proves~(b) by an argument in local coordinates, 
using the construction due to  Emile Borel of a smooth function with
a prescribed Taylor series at a point.  Once~(b) is established 
assertion~(a) follows from the observation that the subspace 
of those elements of the quotient 
$T_\beta\cV_a/(T_\beta\cV_a'+T_\beta\cV_a'')$ that admit 
smooth representatives is both finite dimensional and dense. 
The details are exactly as in the proof of~\cite[Theorem~9.5~(iii)]{RS} 
and will be omitted. Thus we have proved~(ii).  
The proofs of~(iii) and~(iv) are given below after some preparation.
\end{proof}

\para\label{standardNode}
Let $\D\subset\C$ be the closed unit disc.
The \jdef{standard node} is defined as the map
$$
N\to\INT(\D): (x,y)\mapsto xy, \qquad N:=\{(x,y)\in\D\times\D\,|\,\Abs{xy}<1\}.
$$
For $z\in\INT(\D)$ denote
$$
N_z:=\{(x,y)\in\D\times\D\,|\,xy=z\}.
$$
The boundary $\p N_z$ has two components
$$
\p_1N_z:=\{(x,y)\in N_z\,|\,\Abs{x}=1\},\qquad
\p_2N_z:=\{(x,y)\in N_z\,|\,\Abs{y}=1\}
$$
which can be identified with the unit circle $S^1=\p\D\subset\C$ 
via the embeddings $\iota_1,\iota_2:S^1\to N_z$ given by 
$$
\iota_{1,z}(e^{i\theta}) := (e^{i\theta},e^{-i\theta}z),\qquad
\iota_{2,z}(e^{i\theta}) := (e^{-i\theta}z,e^{i\theta}).
$$
We study the set of all triples $(z,\xi,\eta)$
where $z\in\INT(\D)$ and $\xi:S^1\to\C^\sm$,
$\eta:S^1\to\C^\sm$ are the boundary values 
a holomorphic map $v:N_z\to\C^\sm$, namely
$$
\xi := v\circ\iota_{1,z},\qquad
\eta := v\circ\iota_{2,z}.
$$
At $z=0$, the functions  $\xi$ and $\eta$ extend to the
closed unit disk and agree at the origin.
More precisely, fix an integer $s+1/2>1$. 
For $z\in\INT(\D)\setminus 0$ let $\Hol^{s+1/2}(N_z,\C^\sm)$ be the space of
all maps $v:N_z\to\C^\sm$ of class $H^{s+1/2}$ which are holomorphic
in $\INT(N_z)$.  The space $N_0$ consists
of two disks $\D\times 0$ and $0\times\D$ intersecting in $(0,0)$.
In this case let  $\Hol^{s+1/2}(N_0,\C^\sm)$ denote the space of
all continuous maps $v:N_0\to\C^\sm$ such that $v_1:=v|\D\times 0$ 
and $v_2:=v|0\times\D$ are holomorphic in the interior 
and restrict to $H^s$ functions on the boundary. 
In both cases the trace theorem gives rise to a  map
$$
\Hol^{s+1/2}(N_z,\C^\sm)\to H^s(S^1,\C^\sm)\times H^s(S^1,\C^\sm);
v\mapsto(v\circ\iota_{1,z},v\circ\iota_{2,z}).
$$
The norm on $ H^s(S^1,\C^\sm)$ is given by
$$
\left\|\zeta\right\|_s:=\sqrt{\sum_{n\in\Z} (1+|n|)^{2s}|\zeta_n|^2},\qquad
\zeta(e^{i\theta})=\sum_{n\in\Z} \zeta_n e^{in\theta}.
$$
\arap
 
\begin{lemma}\label{le:localmodel}
{\bf (i)} The set 
$$
\cN:= \left\{(z,\xi,\eta)\,\big|\,
\exists\,v\in \Hol^{s+1/2}(N_z,\C^\sm)\,s.t.\,
\xi=v\circ\iota_{1,z},\,\eta=v\circ\iota_{2,z}\right\}.
$$
is a complex submanifold of  
$H^s(S^1,\C^\sm)\times H^s(S^1,\C^\sm)\times\INT(\D)$.

\smallskip\noindent{\bf (ii)}
The projection $\cN\to\INT(\D):(\xi,\eta,z)\mapsto z$
is a surjective submersion. 

\smallskip\noindent{\bf (iii)}
Let $A\subset\INT(\D)\times\C^{\sa-1}$ be an open set and 
$A\to\cN:(z,t)\mapsto(z,\xi_{z,t},\eta_{z,t})$ be a holomorphic
map.  Then the map
$$
H:\left\{(x,y,t)\in\C^{\sa+1}\,|\,x,y\in\INT(\D),\,(xy,t)\in A\right\}\to\C^\sm
$$
well defined by 
$$
H(x,y,t) := \left\{\begin{array}{ll}
\xi_{xy,t}(x),&\mbox{if } y\ne 0,\\   
\eta_{xy,t}(y),&\mbox{if } x\ne 0,\\  
\xi_{0,t}(0)=\eta_{0,t}(0),&\mbox{if }x=y=0,
\end{array}\right.
$$
is holomorphic. 
\end{lemma}

\begin{proof}   
Let $(z,\xi,\eta)\in\INT(\D)\times H^s(S^1,\C^\sm)\times H^s(S^1,\C^\sm)$
and write
$$
\xi(x)=:\sum_{n\in\Z}\xi_nx^n,\qquad
\eta(y)=:\sum_{n\in\Z}\eta_ny^n,
$$
i.e. $\xi_n,\eta_n\in\C^\sm$ are the Fourier coefficients
of $\xi,\eta$. 
When $(z,\xi,\eta)\in\cN$ each of these series converges on the annulus
with inner radius $\Abs{z}$ and outer radius one. 
(Thus was used in defining $H$.)
When $z\ne 0$ we have 
$$
(z,\xi,\eta)\in\cN\qquad\iff\qquad 
\eta_{-n} = z^n\xi_n\mbox{ for all }n\in\Z,
$$
but 
$$
(0,\xi,\eta)\in\cN\qquad\iff\qquad
\xi_0=\eta_0,\;\;\xi_n=\eta_n=0\mbox{ for }n<0. 
$$
Denote by $H^s_\pm(S^1,\C^\sm)\subset H^s(S^1,\C^\sm)$
the Hardy space of all $\zeta\in H^s(S^1,\C^\sm)$ whose 
Fourier coefficients $\zeta_n$ vanish for $\mp n\ge 0$.  
For $z\in\INT(\D)$ define the bounded linear operator
$\cT_z:H^s_+(S^1,\C^\sm)\to H^s_-(S^1,\C^\sm)$ by 
$$
\cT_z\left(\sum_{n>0}c_n e^{in\theta}\right) 
:= \sum_{n>0} z^nc_ne^{-in\theta}.
$$
Then the resulting map 
$$
\INT(\D)\times H^s_+(S^1,\C^\sm)\to H^s_-(S^1,\C^\sm):
(z,\zeta_+)\mapsto \cT_z(\zeta_+)
$$
is holomorphic.  Moreover, the set $\cN$
can be written in the form 
$$
\cN=\left\{
(z,\xi_++\lambda+\cT_z(\eta_+),\eta_++\lambda+\cT_z(\xi_+))\,\bigg|\,
\begin{array}{l}
\xi_+,\eta_+\in H^s_+(S^1,\C^\sm),\\
\lambda\in\C^\sm,\,z\in\INT(\D)
\end{array}
\right\}.
$$
Hence $\cN$ is a complex Hilbert submanifold of the space  
$$
\C\times H^s(S^1,\C^\sm)^2\cong
\C\times H^s_+(S^1,\C^\sm)^2\times(\C^\sm)^2\times H^s_-(S^1,\C^\sm)^2.
$$
The formula shows that the projection $\cN\to\INT(\D)$ is a
surjective submersion.  This proves~(i) and~(ii).  

To prove~(iii) we observe that the projection 
$H^s(S^1,\C^\sm)\to H^s_+(S^1,\C^\sm)$ and the evaluation map 
$\INT(\D)\times H^s_+(S^1,\C^m)\to\C^\sm:(z,\zeta)\mapsto\zeta(z)$
are holomorphic.  Hence~(iii) follows from the identification 
$$
H(x,y,t) = \xi_{xy,t,+}(x) + \eta_{xy,t,+}(y) + \lambda(xy,t)
$$
where $\lambda(z,t)$ denotes the common constant term 
of the power series $\xi_{z,t}$ and~$\eta_{z,t}$.
This proves the lemma.
\end{proof}

\begin{proof}[Proof of Theorem~\ref{thm:V}~(iii) and~(iv).]
We prove that $\cV'$ is a complex Hil\-bert submanifold of $\cV$.
As in the proof of~(i) we choose holomorphic coordinate
charts $\psi_\si:V'_\si\to\C^\sm$ such that $\psi_\si(p_\si)=0$
and $\psi_\si(V'_\si)$ is the open unit disc in $\C^\sm$ for every $\si$. 
Define the map 
$$
\cV'\to A\times(H^s(S^1,\C^\sm))^{2\sk}:
(a,\beta)\mapsto(a,\xi_1,\eta_1,\dots,\xi_\sk,\eta_\sk)
$$
by
$$
\xi_\si:=\psi_\si\circ\beta\circ x_\si^{-1},\qquad
\eta_\si:=\psi_\si\circ\beta\circ y_\si^{-1}.
$$
as in~(\ref{eq:xietai}). The image of this map is the subset
$$
\left\{(a,\xi_1,\eta_1,\dots,\xi_\sk,\eta_\sk)
\in A\times H^s(S^1,\C^\sm))^{2\sk}\,\big|\,
(z_\si(a),\xi_\si,\eta_\si)\in\cN\;\forall\;\si\right\}.
$$
By Lemma~\ref{le:localmodel}, this set 
is a complex Hilbert submanifold of $A\times(H^s(S^1,\C^\sm))^{2\sk}$.
Hence $\cV'$ is a complex Hilbert submanifold of $\cV'$ and the
projection $\cV\to A$ is a submersion.

The proof that $\cV''$ is a complex Hilbert submanifold of $\cV$ follows
the argument in the proof of~\cite[Theorem~11.9~(ii)]{RS}. 
Define
\begin{equation*}
\begin{split}
\cB &:= \bigl\{(a,h'')\,\big|\,a\in A,\,h''\in H^{s+1/2}(P''_a,M)\bigr\}, \\
\cZ &:= \bigl\{(a,h'')\in\cB\,\big|\,h''\in\Hol^{s+1/2}(P''_a,M)\bigr\}.
\end{split}
\end{equation*}
We construct an auxiliary Hilbert manifold structure on $\cB$
and show that $\cZ$ is a smooth submanifold of $\cB$. 
Fix a Hardy trivialization $(P=P'\cup P'',\iota,\rho)$ 
as in~\ref{hardyTrivialization} and denote
$$
\cB_0:=\bigl\{(a,w)\,\big|\,a\in A,\,w\in H^{s+1/2}(\Omega,M)\bigr\}
$$
This space is a Hilbert manifold and the Hardy trivialization
induces a bijection
$$
\cB_0\to\cB:(a,w)\mapsto(a,h'':=w\circ\rho_a).
$$
This defines the Hilbert manifold structure on $\cB$.
The bijection $\cB_0\to\cB$ identifies the subset
$\cZ\subset\cB$ with the subset $\cZ_0\subset\cB_0$ given by
$$
\cZ_0:=\bigl\{(a,w)\in\cB_0\,\big|\,w\in\Hol^{s+1/2}((\Omega,j(a)),M)\bigr\},
$$
where $j(a):=(\rho_a)_*(J_P|P''_a)$, $\rho_a:P''_a\to \Omega$ 
is the Hardy trivialization, and $J_P$ is the complex structure on $P$.
(The map $a\mapsto j(a)$ need not be holomorphic.)

We prove that $\cZ_0$ is a smooth Hilbert submanifold of $\cB_0$.
The tangent space of $\cB_0$ at a pair $(a,w)$ is
$$
T_{a,w}\cB_0 = T_aA\times H^{s+1/2}(\Omega,w^*TM).
$$
Let $\cE\to\cB_0$ be the complex Hilbert space bundle
whose fiber
$$
\cE_{a,w} := H^{s-1/2}(\Omega,
\Lambda_{j(a)}^{0,1}T^*\Omega\otimes w^*TM)
$$
over $(a,w)\in\cB_0$ is the Sobolev space of $(0,1)$-forms
on $(\Omega,j(a))$ of class $H^{s-1/2}$ with values in the
pullback tangent bundle $w^*TM$. As before the
Cauchy--Riemann operator defines a smooth
section $\dbar:\cB_0\to\cE$ given by
\begin{equation}\label{eq:section}
\dbar(a,w) := \bar\p_{j(a),J}(w)
= \frac12\left(dw + J\circ dw\circ j(a)\right).
\end{equation}
Here $J$ denotes the complex structure on $M$.
The zero set of this section is the set $\cZ_0$ defined above.
It follows as in the proof of~(i) that the
linearized operator $D_{a,w}:T_{a,v''}\cB_0\to\cE_{a,w}$
is surjective and has a right inverse.  Hence the zero set $\cZ_0$
is a smooth Hilbert submanifold of $\cB_0$.
Again as in the proof of~(i) restriction to
the boundary gives rise to a smooth injective immersion
$$
\cZ_0\to\cV:(a,w)\mapsto(a,\beta),\qquad
\beta:=w\circ\rho_a^{-1}|\Gamma_a.
$$
The image of a sufficiently small neighbourhood
of $(a_0,w_0:=H_A|\Omega)$ under this immersion is $\cV''$;
the neighborhood is $\cZ_0\cap(A\times\cW_0)$ after shrinking 
$A$ and $\cW_0$, if necessary.  Hence $\cV''$ is a smooth 
Hilbert submanifold of $\cV$.  That it is a complex submanifold
follows, as in the proof of Theorem~11.9 in~\cite{RS}, 
by introducing an auxiliary (almost) complex structure on $\cZ_0$.  
Namely, the push forward of the complex structure on $P''$ 
by the Hardy trivialization
$$
\pi_A\times\rho:P''\to A\times\Omega
$$
of~\ref{hardyTrivialization} has the form~(\ref{eq:JP}) 
for a smooth map $j:A\to\cJ(\Omega)$
and a smooth $1$-form $\eta:TA\to\Vect(\Omega)$ 
satisfying~(\ref{eq:jeta1}) and~(\ref{eq:jeta2}).
Since $\rho$ is holomorphic near $\p P'$ with respect to
the complex structure of $\Omega$ it follows that
$\eta$ vanishes near $A\times\p\Omega$. The tangent space
$T_{(a,w)}\cZ_0$ is the kernel of the operator $\cD_{a,w}$
from $T_{(a,w)}\cB_0$ to $\Omega^{0,1}_{j(a)}(\Omega,w^*TM)$
given by
\begin{equation}\label{eq:Daw}
\cD_{(a,w)}(\hat{a},\hat{w})=D_w\hat{w} +\frac12 J(w)dw\cdot dj(a)\hat{a}.
\end{equation}
It follows from~(\ref{eq:jeta1}) and~(\ref{eq:jeta2}) that the automorphisms
$$
\left(\hat{a},\hat{w}\right)\mapsto
\left(\sqrt{-1}\hat{a}, J(w)\hat{w}-dw\cdot\eta(a,\hat{a})\right)
$$
define an almost complex structure on $\cZ_0$. Since $\eta$ vanishes
near the boundary, the embedding
$$
\cZ_0\to A\times H^s(\Gamma,M): (a,w)\mapsto(a,w\circ\iota_{a_0})
$$
is holomorphic. Hence $\cV''$ is a complex submanifold 
of $\cV$ as claimed.

That the projection 
$\cV''\to A$ is a submersion follows from the fact that the 
linearized operator~(\ref{eq:Daw}) of the section~(\ref{eq:section}) 
is already surjective when differentiating in the direction 
of a vector field $\hat v$ along~$v$. 
This completes the proof Theorem~\ref{thm:V}.
\end{proof}

\dfn\label{def:core}
Let $\pi_A:P\to A$ be a nodal family
and denote by 
$$
C_1,\dots,C_\sk\subset P
$$
the components of the singular set near $P_{a_0}$.
The set
$$
A_0:=\pi_A(C_1)\cap\cdots\cap\pi_A(C_\sk)
$$
is called the \jdef{core} of the family.  
Recall from~\cite[Definition~12.1]{RS} that we 
call $\pi_A$ \jdef{regular nodal}  if
the submanifolds $\pi_A(C_\si)$ intersect transversally. 
In this case, the core  $A_0$ is a complex submanifold of $A$
of codimension $\sk$.
We call an unfolding $(\pi_A:P\to A,R_*,a_0)$ \jdef{regular nodal} 
iff the ambient family $\pi_A:P\to A$  is regular nodal.
In~\cite[Theorem~5.6]{RS} we constructed a universal 
unfolding which is regular nodal.
By the uniqueness of universal unfoldings it follows 
(after shrinking $A$ if necessary)
that every universal unfolding is regular nodal.
\nfd

\begin{theorem}\label{thm:transverse}
Continue the notation of~\ref{Hardy}, \ref{hardyTrivialization}, 
\ref{cV}, and Definition~\ref{def:core}, and
fix an integer $s+1/2>1$.  Assume that the unfolding 
$(\pi_A,R_*,a_0)$ (of marked nodal Riemann surfaces) 
is universal.  Let $w_0:\Sigma\to P_{a_0}$ 
be a desingularization with induced structures 
$s_{0,*}$, $\nu_0$, $j_0$, $v_0:=h_0\circ w_0$ on~$\Sigma$. 
Then the configuration $(\Sigma,s_{0,*},\nu_0,j_0,v_0)$ 
is stable; assume that it is regular. Then the following holds.
\begin{description}
\item[(i)]
$\cV'$ and $\cV''$ intersect transversally in $\cV$ at 
$(a_0,\beta_0:=h_0|\Gamma_{a_0})$. 
\item[(ii)]
The projection $\cV'\cap\cV''\to A$ is tranverse 
to $A_0$ at $(a_0,\beta_0)$.
\end{description}
\end{theorem}

\begin{proof}
Recall the auxiliary Hilbert manifold structure on $\cV$ from~\ref{cV}
given by the bijection~(\ref{eq:Vtriv}). The tangent space at 
$(a,\gamma)\in A\times H^s(\Gamma,M)$
is the set of pairs $(\hat{a},\hat{\gamma})$ with $\hat{a}\in T_aA$
and $\hat{\gamma}\in H^s(\Gamma,\gamma^*TM)$.  We abuse notation
and write
$$
T_{(a,\beta)}\cV = T_aA \times  H^s(\Gamma,\gamma^*TM),\qquad
\gamma:=\beta\circ\iota_a. 
$$
Below we prove the following. 

\medskip\noindent
{\bf Claim:} {\it If $\hat\gamma\in\Omega^0(\Gamma,\gamma_0^*TM)$ is a smooth
vector field along $\gamma_0:=\beta_0\circ\iota_{a_0}$ then the pair 
$(0,\hat\gamma)$ belongs to the sum $T_{(a_0,\beta_0)}\cV'+T_{(a_0,\beta_0)}\cV''$.}

\medskip\noindent
We show first that this claim implies~(i).  
By part~(ii) of Theorem~\ref{thm:V} the sum 
$T_{\beta_0}\cV'_{a_0}+T_{\beta_0}\cV''_{a_0}$ is a closed
subspace of $T_{\beta_0}\cV_{a_0}$ and hence 
$T_{(a_0,\beta_0)}\cV'+T_{(a_0,\beta_0)}\cV''$ is a closed 
subspace of $T_{(a_0,\beta_0)}\cV$.  Hence the claim implies
that every vertical tangent vector $(0,\hat\gamma)$ with
$\hat\gamma\in H^s(\Gamma,\gamma^*TM)$ is contained in the sum
$T_{(a_0,\beta_0)}\cV'+T_{(a_0,\beta_0)}\cV''$.  Since the projection
$\cV'\to A$ is a submersion by part~(iv), this implies
$$
T_{(a_0,\beta_0)}\cV'+T_{(a_0,\beta_0)}\cV''=T_{(a_0,\beta_0)}\cV.
$$
Thus we have proved that~(i) follows from the claim.

The desingularization $w_0:\Sigma\to P_{a_0}$ induces a decomposition
$$
\Sigma=\Sigma'\cup\Sigma'',\qquad 
\Sigma':=w_0^{-1}(P'_{a_0}),\qquad
\Sigma'':=w_0^{-1}(P''_{a_0}).
$$
The intersection $\Sigma'\cap\Sigma''=\p\Sigma'=\p\Sigma''$ is diffeomorphic
to the $1$-manifold $\Gamma$ in~(\ref{eq:trivialize}). 
To simplify the notation we assume that
$$
\Gamma=\Sigma'\cap\Sigma''.
$$
The core admits a smooth desingularization
$$
\iota:A_0\times\Sigma\to P_0:=\pi_A^{-1}(A_0)
$$
that agrees with $w_0:\Sigma\to P_{a_0}$ at the base point $a_0$ and
with the trivialization~(\ref{eq:trivialize}) on $A_0\times\Gamma$.
Choose $\iota$ so that it maps each component of $A_0\times\cup\nu$
to the corresponding component $C_\si$ of the singular set and so 
that 
$$
\iota^{-1}(R_\si)  =  A_0\times\{s_{0,\si}\},\qquad
\si=1,\dots,\sn.
$$
For $a\in A_0$ define the desingularization
$\iota_a:\Sigma\to P_a$ by 
$$
\iota_a(z):=\iota(a,z).
$$ 
The trivialization induces a map $j:A_0\to\cJ(\Sigma)$ determined by the
condition that $\iota_a$ is holomorphic with respect to $j(a)$ for every
$a\in A_0$.  Since $(\pi_A,R_*,a_0)$ is a universal unfolding as in~\cite{RS},
the map $j:A\to\cJ(\Sigma)$ contains a local slice of the $\Diff(\Sigma)$-action.

We prove the claim. Let $\hat\gamma\in\Omega^0(\Gamma,\gamma_0^*TM)$ 
be a smooth vector field along~$\gamma_0$.  There exist  
$\xi'\in \Omega^0(\Sigma',v_0^*TM)$,
$\xi''\in \Omega^0(\Sigma'',v_0^*TM)$, and 
$\eta\in\Omega^{0,1}(\Sigma,v_0^*TM)$
such that 
$$
\hat{\gamma}=(\xi'-\xi'')|\Gamma,\qquad
D_{w_0}\xi'=\eta|\Sigma',\qquad
D_{w_0}\xi''=\eta|\Sigma''.
$$
To see this take $\xi'=0$ and construct $\xi''$
so that $D_{w_0}\xi''$ vanishes to infinite order along $\Gamma$. 
(The equation determines the Taylor expansion along $\Gamma$
and then use Emile Borel's extension theorem.) By the hypothesis
that the stable map $(\Sigma,s_{0,*},\nu_0,j_0,v_0)$ is regular,
there exists $\hat{a}\in T_{a_0}A$ and $\hat v\in\Omega^0(\Sigma/\nu_0,w_0^*TM)$
such that
$$
\eta=\cD_{a_0,v_0}(\hat{a},\hat v)
:=D_{v_0}\hat v + \frac12dv_0\cdot j_0dj(a)\hat{a}.
$$
It follows that the pair $((\xi'-\hat v)|\Gamma,-\hat a)$ represents a tangent vector 
to $\cV'$ and the pair $((\xi''-\hat v)|\Gamma,-\hat a)$ represents a tangent vector 
to $\cV''$.  Their difference is equal to $(0,\hat\gamma)$.  This proves the claim
and hence part~(i) of the theorem. 

We prove~(ii).  
By~(i) and Theorem~\ref{thm:V}~(ii), the intersection
$\cV'\cap\cV''$ has complex dimension
\begin{eqnarray*}
\dim_\C(\cV'\cap\cV'')
&=& \INDEX_\C(D_{v_0})+\dim_\C(A)  \\
&=&  (\sm-3)(1-\sg) + \inner{c_1}{\sd} + \sn 
\end{eqnarray*}
where $\sd:=[v_0]\in H_2(M;\Z)$ 
denotes the homology  class represented by $v_0$.  
Now abbreviate
$$
\gamma_0 := v_0|\Gamma = \beta_0\circ\iota_{a_0}:\Gamma\to M.
$$
Assertion~(ii) follows from the fact that the subspace 
$$
\cX_0 := \left\{(\hat a,\hat\gamma)
\in T_{(a_0,\beta_0)}\cV'\cap T_{(a_0,\beta_0)}\cV''\,\big|\,
\hat a\in T_{a_0}A_0\right\}
$$
has dimension
\begin{equation}\label{eq:dimX0}
\dim_\C\cX_0 = (\sm-3)(1-\sg) + \inner{c_1}{\sd} + \sn - \sk.
\end{equation}
To prove this we observe that the pair 
$(\hat a,\hat\gamma)\in T_{a_0}A\times\Omega^0(\Gamma,\gamma_0^*TM)$ 
belongs to the intersection $T_{(a_0,\beta_0)}\cV'\cap T_{(a_0,\beta_0)}\cV''$
if and only if there exists a vector field $\hat v\in\Omega^0(\Sigma/\nu,v_0^*TM)$
satisfying 
$$
\cD_{a_0,v_0}(\hat{a},\hat v)
= D_{v_0}\hat v + \frac12dv_0\cdot j_0dj(a)\hat{a}
= 0,\qquad
\hat v|\Gamma=\hat\gamma.
$$
Since the restriction of the operator 
$$
D_{v_0}:\Omega^0(\Sigma/\nu,v_0^*TM)\to\Omega^{0,1}(\Sigma,v_0^*TM)
$$
is Fredholm with index 
$$
\INDEX_\C(D_{v_0}) = \sm(1-\sg) + \inner{c_1}{\sd} 
$$
and 
$$
\dim_\C A_0 = 3\sg-3 + \sn - \sk
$$
and the augmented operator
$$
\cD_{a_0,v_0}:T_{a_0}A_0\times
\Omega^0(\Sigma/\nu,v_0^*TM)
\to\Omega^{0,1}(\Sigma,v_0^*TM)
$$
is surjective, this implies~(\ref{eq:dimX0}) and hence part~(ii)
of the theorem. 
\end{proof}

\para\label{infuniv}
For every $a\in A$ there is a map 
\begin{equation}\label{eq:HB}
\cU_a\to\cV_a:(\alpha,b)\mapsto \beta:=H_B\circ\alpha
\end{equation}
which sends $\cU'_a$ to $\cV'_a$ and $\cU''_a$ to $\cV''_a$. 
It follows from our definitions and Theorems~\ref{thm:U} 
and~\ref{thm:V} that the unfolding $(\pi_B,S_*,H_B,b)$ is 
infinitesimally universal if and only if the operator 
$$
dH_B(\alpha):T_{(\alpha,b)}\cU_a\to T_\beta\cV_a
$$
induces isomorphisms
$$
dH_B(\alpha):T_{(\alpha,b)}\cU'_a\cap T_{(\alpha,b)}\cU''_a
\to T_\beta\cV'_a\cap T_\beta\cV''_a,
$$ 
$$
dH_B(\alpha):
\frac{T_{(\alpha,b)}\cU_a}{T_{(\alpha,b)}\cU'_a+T_{(\alpha,b)}\cU''_a}
\to \frac{T_\beta\cV_a}{T_\beta\cV'_a+T_\beta\cV''_a}
$$ 
for some (and hence every) unfolding $(\pi_A,R_*,H_A,a)$
and fiber isomorphism $f:P_a\to Q_b$.  Thus~(\ref{eq:HB})
is an exact morphism of Fredholm quadruples as in~\ref{fredh} 
below. 
\arap


\section{Fredholm intersection theory} \label{sec:interfred} 
  
\para\label{fredE}
Let $E$ be a Hilbert space and $E',E''\subset E$ be closed subspaces. 
We call $(E,E',E'')$ a \jdef{Fredholm triple} (of subspaces) if  
the intersection $E'\cap E''$ is finite dimensional, 
the sum $E'+E''$ is a closed subspace of $E$, 
and the quotient $E/(E'+E'')$ is finite dimensional. 
The triple $(E,E',E'')$ is Fredholm if and only if the operator 
\begin{equation}\label{eq:E} 
E'\times E''\to E:(x',x'')\mapsto x'+x'' 
\end{equation} 
is Fredholm.  The \jdef{Fredholm index} of the triple is defined as the 
Fredholm index of the operator~(\ref{eq:E}). The image of~(\ref{eq:E}) 
is the sum $E'+E''$ and its kernel is isomorphic to $E'\cap E''$ 
via the inclusion 
$$
E'\cap E''\to E'\times E'':x\mapsto(x,-x).
$$ 
Hence the index of the triple $(E,E',E'')$ is 
$$ 
\INDEX(E,E',E'') := \dim(E'\cap E'') - \dim(E/(E'+E'')). 
$$ 
Standard Fredholm theory implies that the Fredholm property  
and the index are stable under small deformations 
of the subspaces $E'$ and $E''$. 
\arap 
 
\para\label{fredX} 
Let $X$ be a Hilbert manifold, $X',X''\subset X$ 
be smooth submanifolds, and $x_0\in X'\cap X''$. 
We call the quadruple $(X,X',X'',x_0)$ \jdef{Fredholm}  
if the triple $(T_{x_0}X,T_{x_0}X',T_{x_0}X'')$ is Fredholm. 
Define its \jdef{Fredholm index} to be the index of the triple. 
If $(X,X',X'',x_0)$ is Fredholm then so is $(X,X',X'',x)$ 
for $x\in X'\cap X''$ sufficiently close to $x_0$ and both 
quadruples have the same Fredholm index. 
\arap 
 
\begin{lemma}[Normal coordinates]\label{le:fredX} 
Let $(X,X',X'',x_0)$ be a Fredholm qua\-druple as in~\ref{fredX} 
and abbreviate 
$$
E:=T_{x_0}X,\qquad E':=T_{x_0}X',\qquad E'':=T_{x_0}X''.
$$  
Then there are coordinates $u,x',x'',\xi$ defined in a neighborhood 
of $x_0$ in $X$ satisfying the following conditions. 
\begin{description} 
\item[(i)] 
$u$ takes values in $E'\cap E''$ and $u(x_0)=0$. 
\item[(ii)] 
$x'$ takes values in a complement to $E'\cap E''$ 
in $E'$ and $x'(x_0)=0$. 
\item[(iii)] 
$x''$ takes values in a complement to $E'\cap E''$ 
in $E''$ and ${x''(x_0)=0}$. 
\item[(iv)] 
$\xi$ takes values in a complement to $E'+E''$ 
in $E$ and $\xi(x_0)=0$. 
\item[(v)] 
Near $x_0$ the submanifolds $X'$, $X''$ and the subset $X'\cap X''$ 
are given by 
$$ 
X''=\{x'=0,\xi=0\},\qquad X'=\{x''=0,\xi=f(u,x')\}, 
$$ 
$$ 
X'\cap X'' = \{x'=0,x''=0,\xi=0,f(u,0)=0\} 
$$ 
for a smooth function $f$ with $f(0,0)=0$ and $df(0,0)=0$. 
\end{description} 
\end{lemma} 

\begin{proof} 
Choose any coordinate chart $(X'',x_0)\to(E'',0)$ whose differential 
at $x_0$ is the identity.  This coordinate chart can be written as 
$(u,x'')$ where $u$ takes values in $E'\cap E''$ and $x''$ 
takes values in a complement of $E'\cap E''$ in $E''$.  
Extend $(u,x'')$ to a coordinate chart $(X,x_0)\to(E,0)$. 
This extended coordinate chart can be written as $(u,x',x'',\xi)$ 
where $x'$ takes values in a complement of $E'\cap E''$ in $E'$ 
and $\xi$ takes values in a complement of $E'+E''$ in $E$. 
In these coordinates we have 
$$ 
X''=\{x'=0,\xi=0\},\qquad X'=\{x''=\phi(u,x'),\xi=f(u,x')\}. 
$$
where $\phi(0,0)=0$, $d\phi(0,0)=0$ and $f(0,0)=0$, 
$df(0,0)=0$. Now replace $x''$ by $x''-\phi(u,x')$ 
to obtain the required coordinate system. 
\end{proof} 

\begin{corollary}\label{cor:U'U''} 
Let $(X,X',X'',x_0)$ be as in Lemma~\ref{le:fredX}.  
Then there exists a neighborhood $X_0$ of $x_0$ in $X$ 
and finite dimensional submanifolds $U$, $U'$, $U''$ 
of $X$, $X'$, $X''$, respectively, passing through $x_0$ 
such that 
$$
U'=U\cap X',\qquad U''=U\cap X'',\qquad 
U'\cap U'' =X_0\cap X'\cap X'' 
$$
and, for $x\in U'\cap U''$, we have 
$$
T_xU'\cap T_xU''=T_xX'\cap T_xX'',\qquad 
\frac{T_xU}{T_xU'+T_xU''} 
\cong \frac{T_xX}{T_xX'+T_xX''}. 
$$
We call $(U,U',U'',x_0)$ a \jdef{finite dimensional reduction}. 
\end{corollary}

\begin{proof} 
Let $X_0$ be the domain of the normal form coordinates $u,x',x'',\xi$ 
introduced in Lemma~\ref{le:fredX}.  Then 
$$
X_0\cap X'\cap X'' = \{(u,0,0,0)\,|\,f(u,0)=0\},
$$ 
$$
T_xX'\cap T_xX'' = \left\{(\hat u,0,0,0)\,|\,
df(u,0)(\hat u,0)=0\right\},
$$
$$
T_xX'+T_xX'' = \left\{(\hat u,\hat x',\hat x'',\hat\xi)\,\bigg|\,
\hat\xi-\frac{\p f}{\p x'}\hat x' \in\im\frac{\p f}{\p u}\right\}
$$
for $x=(u,0,0,0)\in X_0\cap X'\cap X''$. Hence the submanifolds
\begin{equation}\label{eq:U}
U:=\{(u,0,0,\xi)\},\qquad
U' := \{(u,0,0,f(u,0))\},\qquad 
U'':=\{(u,0,0,0)\}
\end{equation}
satisfy the requirements of the corollary. 
\end{proof}

\para\label{fredh}
A \jdef{morphism} from $(X,X',X'',x_0)$ to $(Y,Y',Y'',y_0)$ 
is a smooth map $h:X\to Y$ such that
$$
h(X')\subset Y',\qquad h(X'')\subset Y'',\qquad h(x_0)=y_0.
$$
The morphism $h$ is called \jdef{exact (at $x_0$)} if the differential 
$dh(x_0):T_{x_0}X\to T_{y_0}Y$ induces isomorphisms
$$
dh(x_0):T_{x_0}X'\cap T_{x_0}X''
\to T_{y_0}Y'\cap T_{y_0}Y''
$$
and 
$$
dh(x_0):\frac{T_{x_0}X}{T_{x_0}X'+T_{x_0}X''}
\to\frac{T_{y_0}Y}{T_{y_0}Y'+T_{y_0}Y''}.
$$
The inclusion of a finite dimensional reduction is an example of an 
exact morphism. 
\arap

\begin{theorem}\label{thm:fredh}
Let $h:(X,X',X'',x_0)\to(Y,Y',Y'',y_0)$ be a morphism 
of Fredholm quadruples. Then the following are equivalent.
\begin{description}
\item[(i)]
$h$ is exact at $x_0$. 
\item[(ii)]
There exist finite dimensional reductions 
$(U,U',U'',x_0)$ of $(X,X',X'',x_0)$ and $(V,V',V'',y_0)$ of 
$(Y,Y',Y'',y_0)$ such that $h$ maps $U$, $U'$, $U''$
diffeomorphically onto $V$, $V'$, $V''$, respectively.
\end{description}
\end{theorem}

\begin{proof}
We prove that~(ii) implies~(i).  By~(ii), the 
homomorphism $dh(x_0)$ from $T_{x_0}X'\cap T_{x_0}X''$
to $T_{y_0}Y'\cap T_{y_0}Y''$ can be written as the composition
$$
T_{x_0}X'\cap T_{x_0}X''
= T_{x_0}U'\cap T_{x_0}U''
\stackrel{dh(x_0)}{\longrightarrow} T_{y_0}V'\cap T_{y_0}V''
= T_{y_0}Y'\cap T_{y_0}Y''
$$
and hence is an isomorphism. Similarly for the 
map from $T_{x_0}X/(T_{x_0}X'+T_{x_0}X'')$ to
${T_{y_0}Y/(T_{y_0}Y'+T_{y_0}Y'')}$.

We prove that~(i) implies~(ii). 
Let $u,x',x'',\xi$ be the normal coordinates
on $X$ introduced in Lemma~\ref{le:fredX} and choose similar
normal coordinates $v,y',y'',\eta$ on $Y$ at $y_0$. 
Thus 
\begin{equation}\label{eq:Y1}
Y''=\{y'=0,\eta=0\},\qquad Y'=\{y''=0,\eta=g(v,y')\},
\end{equation}
\begin{equation}\label{eq:Y2}
Y'\cap Y'' = \{y'=0,y''=0,\eta=0,g(v,0)=0\}
\end{equation}
for a smooth function $g$ with $g(0,0)=0$ and $dg(0,0)=0$.
In these coordinates the morphism $h=(h_1,h_2,h_3,h_4)$
satsfies
\begin{equation}\label{eq:h1}
h_2(u,0,x'',0)=0,\qquad h_4(u,0,x'',0)=0
\end{equation}
(because $h(X'')\subset Y''$),
\begin{equation}\label{eq:h2}
h_3(u,x',0,f(u,x'))=0,
\end{equation}
\begin{equation}\label{eq:h3}
h_4(u,x',0,f(u,x'))=g(h_1(u,x',0,f(u,x')),h_2(u,x',0,f(u,x')))
\end{equation}
(because $h(X')\subset Y'$), and
\begin{equation}\label{eq:h4}
\det(\p h_1/\p u)(0,0,0,0)\ne 0,\qquad
\det(\p h_4/\p\xi)(0,0,0,0)\ne 0
\end{equation}
(because $h$ is exact).  By~(\ref{eq:h1}) and~(\ref{eq:h4}), the restriction of 
$h$ to a neighborhood of $x_0$ in $U$ is an embedding.
Shrinking the domain $X_0\subset X$ of the normal coordinates,
if necessary, we may assume that $h|U:U\to Y$ is an embedding.
Denote 
$$
V:=h(U),\qquad V':=h(U'),\qquad V'':=h(U'').
$$
We must prove that $(V,V',V'',y_0)$ is a finite dimensional
reduction.   
\begin{description}
\item[(a)]
The set $V$ consists of all quadruples of the form $(v,y',y'',\eta)$ 
where 
$$
y':=h_2(u,0,0,\xi),\qquad y'':=h_3(u,0,0,\xi)
$$
and $u,\xi$ are defined by $h_1(u,0,0,\xi)=v$, $h_4(u,0,0,\xi)=\eta$.
\item[(b)]
The set $V'$ consists of all quadruples of the form $(v,y',0,g(v,y'))$ 
where 
$$
y':=h_2(u,0,0,f(u,0)),\qquad 
h_1(u,0,0,f(u,0)):=v.
$$ 
\item[(c)]
The set $V''$ consists of all quadruples of the form $(v,0,y'',0)$ 
where 
$$
y'':=h_3(u,0,0,0),\qquad 
h_1(u,0,0,0):=v.
$$ 
\end{description}
Thus a point in the intersection $V'\cap V''$ has the form
$(v,0,0,0)$ where $v$ satisfies the conditions
\begin{description}
\item[(i)]  $g(v,0)=0$
\item[(ii)]  If $u$ is defined by $h_1(u,0,0,f(u,0)):=v$
then $h_2(u,0,0,f(u,0))=0$.
\item[(iii)]  If $u$ is defined by $h_1(u,0,0,0):=v$
then $h_3(u,0,0,0)=0$.
\end{description}
We show that~(i) implies~(ii) and~(iii) whenever $v$ is sufficiently small. 
For~(ii) we define $u$ as the unique solution of $h_1(u,0,0,f(u,0))=v$
so that
\begin{equation}\label{eq:gh2h4}
g(v,0)=0,\qquad 
g(v,h_2(u,0,0,f(u,0)))= h_4(u,0,0,f(u,0)).
\end{equation}
We claim that for $v$ sufficiently small this implies $f(u,0)=0$.
To see this we use first that the solution $u$ of the equation
$h_1(u,0,0,f(u,0))=v$ satisfies an inequality
\begin{equation}\label{eq:fu1}
\Norm{u}+\Norm{f(u,0)}\le c\Norm{v}
\end{equation}
for $v$ sufficiently small.  Next we use the fact that
$h_2(u,0,0,0)=0$ and hence 
\begin{equation}\label{eq:fu2}
\Norm{h_2(u,0,0,\xi)}\le c\Norm{\xi}.
\end{equation}
Third, we have that $h_4(u,0,0,0)=0$ and $\p h_4/\p\xi$
is invertible at the point $(0,0,0,0)$, hence also at the point
$(u,0,0,0)$ for $u$ sufficiently small.  Hence we have
an inequality
\begin{equation}\label{eq:fu3}
\Norm{h_4(u,0,0,\xi)}\ge c^{-1}\Norm{\xi}
\end{equation}
for a suitable constant $c>0$ and $u$ and $\xi$ sufficiently small. 
Fourth, since $g(0,0)=0$ and $dg(0,0)=0$, there is an inequality
\begin{equation}\label{eq:fu4}
\Norm{g(v,y')-g(v,0)}\le c\left(\Norm{v}+\Norm{y'}\right)\Norm{y'}
\end{equation}
for a suitable constant $c$.  Putting these four inequalities together
and inserting $\xi=f(u,0)$ and $y'=h_2(u,0,0,f(u,0))$ we deduce
$$
\begin{array}{lclr}
\Norm{f(u,0)}
&\le& 
c \Norm{h_4(u,0,0,f(u,0))} & \mbox{by }(\ref{eq:fu3}) \\
&=& 
c\Norm{g(v,h_2(u,0,0,f(u,0)))-g(v,0)} & \mbox{by }(\ref{eq:gh2h4})  \\
&\le&
c^2\left(\Norm{v}+\Norm{h_2(u,0,0,f(u,0))}\right)
\Norm{h_2(u,0,0,f(u,0))}   & \mbox{by }(\ref{eq:fu4}) \\
&\le&
c^3\left(\Norm{v}+c\Norm{f(u,0)}\right)\Norm{f(u,0)}
& \mbox{by }(\ref{eq:fu2}) \\  
&\le&
(c^3+c^5)\Norm{v}\Norm{f(u,0)} & \mbox{by }(\ref{eq:fu1})
\end{array}
$$
for $v$ sufficiently small. With $(c^3+c^5)\Norm{v}<1$ this implies 
$$
f(u,0)=0
$$
as claimed and hence $h_2(u,0,0,f(u,0))=0$, 
by~(\ref{eq:h1}).  Thus we have proved that~(i) implies~(ii).
Since $f(u,0)=0$ we also deduce that our $u$
is the unique solution of $h_1(u,0,0,0)=v$ needed in~(iii).
Using $f(u,0)=0$ again we obtain $h_3(u,0,0,0)=0$, 
by~(\ref{eq:h2}). Thus we have proved that~(i) implies~(ii)
and~(iii) and hence 
$$
V'\cap V'' = \left\{(v,0,0,0)\,|\,g(v,0)=0\right\}
= Y_0\cap Y'\cap Y''
$$
for a suitable open neighborhood $Y_0$ of $y_0$ in $Y$. 

Next we examine the tangent spaces of $V$, $V'$, and $V''$ 
at a point 
$$
y:=(v,0,0,0)\in V'\cap V'',\qquad g(v,0)=0.
$$
Let $x=(u,0,0,0)\in U'\cap U''$ with $f(u,0)=0$ be the 
element with $h(x)=y$.
\begin{description}
\item[(A)]
The tangent space $T_yV$ consists of all vectors 
$\hat y=(\hat v,\hat y',\hat y'',\hat \eta)$ where 
$$
\hat y':=\frac{\p h_2}{\p\xi}\hat\xi,\qquad 
\hat y'':=\frac{\p h_3}{\p u}\hat u+\frac{\p h_3}{\p\xi}\hat\xi
$$
and $\hat u,\hat\xi$ are defined by
\begin{equation}\label{eq:uhat}
\hat u:= \left(\frac{\p h_1}{\p u}\right)^{-1}
\left(\hat v-\frac{\p h_1}{\p\xi}\hat\xi\right)
\end{equation}
\begin{equation}\label{eq:etahat}
\hat\xi:=\left(\frac{\p h_4}{\p\xi}\right)^{-1}\hat\eta.
\end{equation}
Here and below all partial derivatives of $h$ are evaluated 
at $x=(u,0,0,0)$ and we have used the fact that 
$\p h_2/\p u$ and $\p h_4/\p u$ vanish at $x$, by~(\ref{eq:h1}). 
\item[(B)]
The tangent space $T_yV'$ consists of all vectors 
$\hat y=(\hat v,\hat y',0,\hat \eta)$ where 
\begin{equation}\label{eq:uhat-etahat}
\hat y':=\frac{\p h_2}{\p\xi}\frac{\p f}{\p u}\hat u,\qquad 
\hat\eta:=\frac{\p g}{\p v}\hat v+\frac{\p g}{\p y'}\hat y'
\end{equation}
and $\hat u$ is defined 
\begin{equation}\label{eq:uhatB}
\hat u:= \left(\frac{\p h_1}{\p u}+\frac{\p h_1}{\p\xi}\frac{\p f}{\p u}\right)^{-1}\hat v. 
\end{equation}
Here and below all partial derivatives of $f$ are evaluated
at $(u,0)$ and all partial derivatives of $g$ at $(v,0)$. 
\item[(C)]
The tangent space $T_yV''$ consists of all vectors 
$\hat y=(\hat v,0,\hat y'',0)$ where 
\begin{equation}\label{eq:C}
\hat y'':=-\frac{\p h_3}{\p\xi}\frac{\p f}{\p u}\hat u,\qquad
\hat u:=\left(\frac{\p h_1}{\p u}\right)^{-1}\hat v.
\end{equation}
Note that $-(\p h_3/\p\xi)(\p f/\p u)=\p h_3/\p u$, by~(\ref{eq:h2}). 
\end{description}

We prove that the intersection $T_yV'\cap T_yV''$ consists of 
all vectors $\hat y=(\hat v,0,0,0)$ where $\hat v$ satisfies 
the conditions
\begin{equation}\label{eq:vhat1}
\frac{\p g}{\p v}\hat v=0,
\end{equation}
\begin{equation}\label{eq:vhat2}
 \frac{\p f}{\p u}\hat u=0
\end{equation}
where $\hat u$ is given by~(\ref{eq:uhatB}).
First assume $\hat v$ satisfies~(\ref{eq:vhat1}) and~(\ref{eq:vhat2}).
We show that $\hat y:=(\hat v,0,0,0)\in T_yV'\cap T_yV''$. 
By~(\ref{eq:vhat2}), we have $\hat y'=0$ in~(\ref{eq:uhat-etahat}) and hence,
by~(\ref{eq:vhat1}), $\hat\eta =(\p g/\p v)\hat v=0$  in~(\ref{eq:uhat-etahat}). 
Thus $\hat y\in T_yV'$.
Moreover the vector $\hat u$ in~(\ref{eq:uhatB})
satisfies $(\p h_1/\p u)\hat u=\hat v$ by~(\ref{eq:vhat2}) 
and, also by~(\ref{eq:vhat2}), we have $\hat y''=0$ in~(\ref{eq:C}). 
Thus $\hat y\in T_yV''$.

Conversely assume $\hat y\in T_yV'\cap T_yV''$. 
We show that $\hat y=(\hat v,0,0,0)$ where $\hat v$
satisfies~(\ref{eq:vhat1}) and~(\ref{eq:vhat2}).
That $\hat y$ has the form $(\hat v,0,0,0)$ follows immediately
from~(B) and~(C). Equation~(\ref{eq:vhat1}) follows immediately
from~(B) and the fact that $\hat y'=0$.  To prove that $\hat v$
satisfies~(\ref{eq:vhat2}) we differentiate equation~(\ref{eq:h3})
at the point $x=(u,0,0,0)$ with respect to $u$ to obtain
\begin{equation}\label{eq:h5}
\frac{\p h_4}{\p\xi}\frac{\p f}{\p u}
= \frac{\p g}{\p v}\left(\frac{\p h_1}{\p u} + \frac{\p h_1}{\p\xi}\frac{\p f}{\p u}\right)
+  \frac{\p g}{\p y'} \frac{\p h_2}{\p\xi}\frac{\p f}{\p u}.
\end{equation}
Here we have used the fact that $\p h_2/\p u$ and $\p h_4/\p u$ 
vanish at $x$, by~(\ref{eq:h1}).  Evaluating~(\ref{eq:h5})
in the direction of the vector $\hat u$ in~(\ref{eq:vhat2})
gives
$$
\frac{\p h_4}{\p\xi}\frac{\p f}{\p u}\hat u
= \frac{\p g}{\p v}\hat v + \frac{\p g}{\p y'}\hat y'
= 0.
$$
Since $\p h_4/\p\xi$ is invertible this proves~(\ref{eq:vhat2}). 

We prove that
\begin{equation}\label{eq:TV'TV''}
 T_yV'\cap T_yV''=\left\{(\hat v,0,0,0)\,\Big|\,\frac{\p g}{\p v}\hat v=0\right\},
\end{equation}
i.e.  that~(\ref{eq:vhat1}) implies~(\ref{eq:vhat2}). 
Let $\hat u$ be given by~(\ref{eq:uhatB}) and abbreviate
$$
\hat\xi := \frac{\p f}{\p u}\hat u.
$$
Evaluating~(\ref{eq:h5}) again in the direction of the vector 
$\hat u$ in~(\ref{eq:vhat2}) and using~(\ref{eq:vhat1}) we obtain 
$$
\frac{\p h_4}{\p\xi}\hat\xi
= \frac{\p g}{\p y'} \frac{\p h_2}{\p\xi}\hat\xi.
$$
Since $\p g/\p y'$ vanishes at the origin it is small when
$v$ is small and hence, in this case, $\hat\xi=0$ as claimed. 
This proves~(\ref{eq:TV'TV''}). 
By~(\ref{eq:Y1}), the right hand side of~(\ref{eq:TV'TV''})
is $T_yY'\cap T_yY''$.
This proves that 
$$
T_yV'\cap T_yV''=T_yY'\cap T_yY''.
$$

It remains to prove that 
\begin{equation}\label{eq:VY1}
\frac{T_yV}{T_yV'+T_yV''} 
\cong \frac{T_yY}{T_yY'+T_yY''}. 
\end{equation}
Since $T_yV'\cap T_yV''=T_yY'\cap T_yY''$ and the Fredholm
quadruples $(V,V',V'',y)$ and $(Y,Y',Y'',y)$ have the same Fredholm index
for $y\in V'\cap V''$ sufficiently small, both quotient spaces have the same 
dimension.  Hence condition~(\ref{eq:VY1}) is equivalent to
\begin{equation}\label{eq:VY2}
T_yV\cap(T_yY'+T_yY'') \subset T_yV'+T_yV''. 
\end{equation}
The sum $T_yY'+T_yY''$ is the set of all vectors 
$\hat y=(\hat v,\hat y',\hat y'',\hat \eta)$ that satisfy
\begin{equation}\label{eq:TY}
\hat\eta-\frac{\p g}{\p y'}\hat y' \in\im\left(\frac{\p g}{\p v}\right). 
\end{equation}
To prove~(\ref{eq:VY2}) fix a vector 
$\hat y=(\hat v,\hat y',\hat y'',\hat \eta)\in T_yV\cap(T_yY'+T_yY'')$. 
By~(\ref{eq:TY}) there is a vector $\hat v'$ such that
\begin{equation}\label{eq:v'}
\hat\eta-\frac{\p g}{\p y'}\hat y' = \frac{\p g}{\p v}\hat v'.
\end{equation}
We prove that
\begin{equation}\label{eq:yhat}
(\hat v',\hat y',0,\hat\eta)\in T_yV',\qquad
(\hat v'',0,y'',0)\in T_yV'',\qquad \hat v'' := \hat v - \hat v'.
\end{equation}
To see this define the vectors $\hat u$ and $\hat\xi$ by
\begin{equation}\label{eq:uhatxihat}
\frac{\p h_1}{\p u}\hat u+\frac{\p h_1}{\p\xi}\hat\xi=\hat v,\qquad 
\frac{\p h_4}{\p\xi}\hat\xi=\hat\eta
\end{equation}
as in~(A) so that 
\begin{equation}\label{eq:yhat'}
\hat y' =\frac{\p h_2}{\p\xi}\hat\xi,\qquad 
\hat y'' =\frac{\p h_3}{\p u}\hat u+\frac{\p h_3}{\p\xi}\hat\xi.
\end{equation}
Next define $\hat u'$ and $\hat u''$ by 
\begin{equation}\label{eq:uhat'}
\frac{\p h_1}{\p u}\hat u'
+\frac{\p h_1}{\p\xi}\frac{\p f}{\p u}\hat u':=\hat v',\qquad
\frac{\p h_1}{\p u}\hat u'' := \hat v''.
\end{equation}
Then, by~(\ref{eq:h5}), (\ref{eq:v'}), and~(\ref{eq:uhatxihat}-\ref{eq:uhat'}), we have
\begin{eqnarray*}
\frac{\p h_4}{\p\xi}\left(\frac{\p f}{\p u}\hat u'-\hat\xi\right)
&=& 
\frac{\p g}{\p v}
\left(\frac{\p h_1}{\p u} + \frac{\p h_1}{\p\xi}\frac{\p f}{\p u}\right)\hat u' 
+ \frac{\p g}{\p y'}\frac{\p h_2}{\p\xi}\frac{\p f}{\p u}\hat u'
- \hat\eta  \\
&=& 
\frac{\p g}{\p v}\hat v' + \frac{\p g}{\p y'}\hat y' 
- \hat\eta + \frac{\p g}{\p y'}\frac{\p h_2}{\p\xi}
\left(\frac{\p f}{\p u}\hat u'-\hat\xi\right) \\
&=& 
\frac{\p g}{\p y'}\frac{\p h_2}{\p\xi}
\left(\frac{\p f}{\p u}\hat u'-\hat\xi\right).
\end{eqnarray*}
Since $\p g/\p y'$ is small when $v$ is small this implies
$$
\hat\xi = \frac{\p f}{\p u}\hat u',\qquad \hat u'+\hat u''=\hat u.
$$
Here the last equation follows from the first 
and~(\ref{eq:uhatxihat}) and~(\ref{eq:uhat'}).
Now it follows from~(\ref{eq:yhat'}) that
$$
\hat y'' =\frac{\p h_3}{\p u}\hat u+\frac{\p h_3}{\p\xi}\hat\xi
= \left(\frac{\p h_3}{\p u}+\frac{\p h_3}{\p\xi}\frac{\p f}{\p u}\right)\hat u'
+ \frac{\p h_3}{\p u}\hat u''
= \frac{\p h_3}{\p u}\hat u''.
$$
Combining this with~(C) and~(\ref{eq:uhat'}) we find that 
$(\hat v'',0,\hat y'',0)\in T_yV''$.  Likewise it follows from~(B)
and~(\ref{eq:v'}), (\ref{eq:yhat'}) and~(\ref{eq:uhat'})
that $(\hat v',\hat y',0,\eta')\in T_yV'$. Thus we have 
proved~(\ref{eq:yhat}). This completes the proof 
of~(\ref{eq:VY2}) and the theorem.
\end{proof}

Let $A\subset X$ and $B\subset Y$ be arbitrary subsets. 
Recall that $\phi:A\to B$ is by definition a diffeomorpism 
if it is bijective and $\phi$ and $\phi^{-1}$ are smooth,
i.e.\ for every point $x\in A$ there is a smooth
extension of $\phi$ from a neighbourhood of $x$ in $X$ to $Y$,
and for every point $y\in B$ there is a smooth
extension of $\phi^{-1}$ from a neighbourhood 
of $y$ in $Y$ to $X$ (see~\cite{MILNOR}).

\begin{corollary}\label{cor:fredh}
Let $h:(X,X',X'',x_0)\to(Y,Y',Y'',y_0)$ be an exact morphism 
of Fredholm quadruples. Then the following holds.

\smallskip\noindent{\bf (I)}
$h$ maps a neighborhood of $x_0$ in $X'\cap X''$ 
diffeomorphically onto a neighborhood of $y_0$ in $Y'\cap Y''$. 

\smallskip\noindent{\bf (II)}
$h$ is exact at every point $x\in X'\cap X''$ sufficiently close to $x_0$.
\end{corollary}

\begin{proof}
Of course $X'\cap X''$ need not be a manifold. 
Let $(U,U',U'')$ and $(V,V',V'')$ be the finite dimensional reductions
of Theorem~\ref{thm:fredh}.  Then assertion~(I) follows from
the fact that $h^{-1}:V\to U$ extends to a smooth map from
a neighborhood of $V$ to~$X$.  Assertion~(II) follows from 
the equivalence of~(i) and~(ii) in Theorem~\ref{thm:fredh};
namely, if~(ii) holds for $x_0$ then it also holds for every point
$x\in X'\cap X''$ sufficiently close to $x_0$ (with the same 
finite dimensional reductions). This proves the corollary. 
\end{proof}

\begin{theorem}\label{thm:fredstable}
Let $h_\lambda:(X,X'_\lambda,X''_\lambda)\to(Y,Y'_\lambda,Y''_\lambda)$ 
be a smooth  family of morphisms of Fredholm triples parametrized by 
$\lambda\in\Lambda$, where $\Lambda$ is a finite dimensional
manifold, i.e. the map 
$$
h:\Lambda\times X\to\Lambda\times Y,\qquad
h(\lambda,x):=(\lambda,h_\lambda(x)),
$$
is smooth, the sets 
$$
X':=\bigsqcup_\lambda X_\lambda',\qquad
X'':=\bigsqcup_\lambda X_\lambda''
$$
are smooth submanifolds of $\Lambda\times X$, the sets
$$
Y':=\bigsqcup_\lambda Y_\lambda',\qquad
Y'':=\bigsqcup_\lambda Y_\lambda''
$$
are smooth submanifolds of $\Lambda\times Y$, and the 
projections from $X',X'',Y',Y''$ to~$\Lambda$ are submersions.
Let $\lambda_0\in\Lambda$,
$x_0\in X'_{\lambda_0}\cap X''_{\lambda_0}$,
and $y_0:=h_{\lambda_0}(x_0)$. Then the following holds.
\begin{description}
\item[(i)]
The Fredholm indices are related by 
\begin{equation*}
\begin {split}
\INDEX(\Lambda\times X,X',X'',(\lambda_0,x_0))
&=\INDEX(X_{\lambda_0},X'_{\lambda_0},X''_{\lambda_0},x_0)
+\dim\Lambda, \\
\INDEX(\Lambda\times Y,Y',Y'',(\lambda_0,y_0))
&=\INDEX(Y_{\lambda_0},Y'_{\lambda_0},Y''_{\lambda_0},y_0)
+\dim\Lambda.
\end{split}
\end{equation*}
\item[(ii)]
$h_{\lambda_0}$ is exact at $x_0$ if and only if
$h$ is exact at $(\lambda_0,x_0)$.
\end{description}
\end{theorem}

\bigbreak

\begin{proof}
There is a commutative diagram 
$$
\Rectangle{T_{x_0}X'_{\lambda_0}\times T_{x_0}X''_{\lambda_0}}
{}{T_{x_0}X_{\lambda_0}}
{}{}
{T_{(\lambda_0,x_0)}X'\times T_{(\lambda_0,x_0)}X''}{}
{T_{(\lambda_0,x_0)}X}
$$
of Fredholm operators where the horizontal arrows
are as in~\ref{fredE} and the vertical arrows are inclusions. 
The Fredholm index of the top horizontal arrow is 
$\INDEX(X_{\lambda_0},X'_{\lambda_0},X''_{\lambda_0},x_0)$,
the index of the bottom horizontal arrow is
$\INDEX(\Lambda\times X,X',X'',(\lambda_0,x_0))$,
that of the left vertical arrow is $-2\dim\Lambda$,
and that of the right vertical arrow is $-\dim\Lambda$.
(Here we have used the fact that the projections 
$X'\to\Lambda$ and $X''\to\Lambda$ are submersions.)
Hence assertion~(i) follows from the fact that the Fredholm
index of a composition is the sum of the Fredholm indices. 

We prove~(ii). 
Assume first that $h_{\lambda_0}$ is exact at $x_0$
and denote ${y_0:=h_{\lambda_0}(x_0)}$.
We prove that the induced homomorphism 
\begin{equation}\label{eq:hl1}
dh(\lambda_0,x_0):
T_{(\lambda_0,x_0)}X'\cap T_{(\lambda_0,x_0)}X''
\to T_{(\lambda_0,y_0)}Y'\cap T_{(\lambda_0,y_0)}Y''
\end{equation}
is injective.  If 
$(\hat\lambda,\hat x)\in T_{(\lambda_0,x_0)}X'\cap T_{(\lambda_0,x_0)}X''$
and $dh(\lambda_0,x_0)(\hat\lambda,\hat x)=0$ 
then 
$$
\hat\lambda=0,\qquad dh_{\lambda_0}(x_0)\hat x=0.
$$ 
Since the projections $X'\to\Lambda$ and
$X''\to\Lambda$ are submersions we have
$\hat x\in T_{x_0}X'_{\lambda_0}\cap T_{x_0}X''_{\lambda_0}$.
By assumption, this implies $\hat x=0$. 
This shows that~(\ref{eq:hl1}) is injective, as claimed. 
We prove that the induced homomorphism
\begin{equation}\label{eq:hl2}
dh(\lambda_0,x_0):\frac{T_{(\lambda_0,x_0)}(\Lambda\times X)}
{T_{(\lambda_0,x_0)}X'+T_{(\lambda_0,x_0)}X''}
\to \frac{T_{(\lambda_0,y_0)}(\Lambda\times Y)}
{T_{(\lambda_0,y_0)}Y'+T_{(\lambda_0,y_0)}Y''}
\end{equation}
is surjective.  Let
$(\hat\lambda,\hat y)\in T_{(\lambda_0,y_0)}(\Lambda\times Y)$.
Since the projection $X'\to\Lambda$ is a submersion,
there is a vector $\hat x\in T_{x_0}X$ such that 
$(\hat\lambda,\hat x)\in T_{(\lambda_0,x_0)}X'$.
Define $\hat y_0\in T_{y_0}Y$ by
$$
(0,\hat y_0):=(\hat\lambda,\hat y)-dh(\lambda_0,x_0)(\hat\lambda,\hat x).
$$ 
By assumption, there exists a vector 
$\hat x_0\in T_{x_0}X$ such that 
$$
\hat y_0-dh_{\lambda_0}(x_0)\hat x_0\in T_{y_0}Y'+T_{y_0}Y''.
$$
Hence 
$$
(0,\hat y_0)-dh(\lambda_0,x_0)(0,\hat x_0)
\in T_{(\lambda_0,y_0)}Y'+T_{(\lambda_0,y_0)}Y''.
$$
and hence 
$$
(\hat\lambda,\hat y)-dh(\lambda_0,x_0)(0,\hat x_0)
\in T_{(\lambda_0,y_0)}Y'+T_{(\lambda_0,y_0)}Y''.
$$
This shows that~(\ref{eq:hl2}) is surjective, as claimed. 
Moreover, by~(i) the quadruples 
$(\Lambda\times X,X',X'',(\lambda_0,x_0))$ and
$(\Lambda\times Y,Y',Y'',(\lambda_0,y_0))$
have the same Fredholm index.  Hence~(\ref{eq:hl1})
and~(\ref{eq:hl2}) are bijective and so $h$ is exact at
$(\lambda_0,x_0)$.   

Conversely, assume that $h$ is exact at $(\lambda_0,x_0)$ so 
that~(\ref{eq:hl1}) and~(\ref{eq:hl2}) are bijective.
We prove that the induced homomorphism
\begin{equation}\label{eq:h01}
dh_{\lambda_0}(x_0):T_{x_0}X'_{\lambda_0}\cap T_{x_0}X''_{\lambda_0}
\to T_{y_0}Y'_{\lambda_0}\cap T_{y_0}Y''_{\lambda_0}
\end{equation}
is injective. Let
$\hat x\in T_{x_0}X'_{\lambda_0}\cap T_{x_0}X''_{\lambda_0}$
and suppose that $dh_{\lambda_0}(x_0)\hat x=0$.  Then 
$$
(0,\hat x)\in T_{(\lambda_0,x_0)}X'\cap T_{(\lambda_0,x_0)}X'',\qquad
dh(\lambda_0,x_0)(0,\hat x)=(0,0).
$$
Since~(\ref{eq:hl1}) is injective, this implies $\hat x=0$.
This shows that~(\ref{eq:h01}) is injective.
We prove that the induced homomorphism
\begin{equation}\label{eq:h02}
dh_{\lambda_0}(x_0):\frac{T_{x_0}X}
{T_{x_0}X'_{\lambda_0}+T_{x_0}X''_{\lambda_0}}
\to \frac{T_{y_0}Y}
{T_{y_0}Y'_{\lambda_0}+T_{y_0}Y''_{\lambda_0}}
\end{equation}
is surjective. 
Let $\hat y\in T_{y_0}Y$. Since~(\ref{eq:hl2}) 
is surjective, there exists a pair 
$(\hat\lambda,\hat x)\in T_{(\lambda_0,x_0)}(\Lambda\times X)$
such that 
$$
(0,\hat y)-dh(\lambda_0,x_0)(\hat\lambda,\hat x)
\in T_{(\lambda_0,y_0)}Y'+T_{(\lambda_0,y_0)}Y''.
$$
Write
\begin{equation}\label{eq:yhat1}
(0,\hat y)-
dh(\lambda_0,x_0)(\hat\lambda,\hat x)
= (\hat\lambda',\hat y') + (\hat\lambda'',\hat y'')
\end{equation}
where 
$$
(\hat\lambda',\hat y')\in T_{(\lambda_0,y_0)}Y',\qquad
(\hat\lambda'',\hat y'')\in T_{(\lambda_0,y_0)}Y''.
$$
Since the projections ${X'\to\Lambda}$ and $X''\to\Lambda$ 
are submersions, there exist tangent vectors $\hat x',\hat x''\in T_{x_0}X$
such that 
$$
(\hat\lambda',\hat x')\in T_{(\lambda_0,x_0)}X',\qquad
(\hat\lambda'',\hat x'')\in T_{(\lambda_0,x_0)}X''.
$$
Define the tangent vectors $\hat y_0',\hat y_0''\in T_{y_0}Y$
by 
\begin{equation}\label{eq:yhat2}
\begin{split}
(0,\hat y_0') 
&:= (\hat\lambda',\hat y')-dh(\lambda_0,x_0)(\hat\lambda',\hat x')
\in T_{(\lambda_0,y_0)}Y',\\
(0,\hat y_0'')
&:= (\hat\lambda'',\hat y'')-dh(\lambda_0,x_0)(\hat\lambda'',\hat x'')
\in T_{(\lambda_0,y_0)}Y''.
\end{split}
\end{equation}
Since the projections $Y'\to\Lambda$ and $Y''\to\Lambda$ 
are submersions we have 
$$
\hat y_0'\in T_{y_0}Y'_{\lambda_0},\qquad
\hat y_0''\in T_{y_0}Y''_{\lambda_0}.
$$
Moreover, by~(\ref{eq:yhat1}), we have
$$
\hat\lambda+\hat\lambda'+\hat\lambda''=0
$$
and hence, by~(\ref{eq:yhat1}) and~(\ref{eq:yhat2}), 
$$
\hat y - dh_{\lambda_0}(x_0)(\hat x+\hat x'+\hat x'')
= \hat y_0'+\hat y_0''
\in T_{y_0}Y'_{\lambda_0}+T_{y_0}Y''_{\lambda_0}.
$$
Hence~(\ref{eq:h02}) is surjective, as claimed.
Now it follows again from the index identities in~(i) 
that~(\ref{eq:h01}) and~(\ref{eq:h02}) are bijective 
and hence $h_{\lambda_0}$ is exact at $x_0$.  
This proves the theorem. 
\end{proof}

\begin{corollary}\label{cor:fredstable}
Let $h_\lambda:(X,X'_\lambda,X''_\lambda)\to(Y,Y'_\lambda,Y''_\lambda)$
be as in Theorem~\ref{thm:fredstable} and suppose that $h_{\lambda_0}$
is exact at $x_0\in X'_{\lambda_0}\cap X''_{\lambda_0}$.
Then the following holds. 
\begin{description}
\item[(i)]
If $\lambda$ is sufficiently close to $\lambda_0$ and 
$x\in X'_\lambda\cap X''_\lambda$ is sufficiently close to $x_0$
then $h_\lambda$ is exact at $x$. 
\item[(ii)]
If $\Lambda\to Y:\lambda\mapsto y_\lambda$ is a smooth
map such that $y_\lambda\in Y'_\lambda\cap Y''_\lambda$
for every~$\lambda$ then, after shrinking $\Lambda$
if necessary, there exists a unique smooth map
$\Lambda\to X:\lambda\mapsto x_\lambda$ 
such that $x_\lambda\in X'_\lambda\cap X''_\lambda$
and $h_\lambda(x_\lambda)=y_\lambda$ for every~$\lambda$. 
\end{description}
\end{corollary}

\begin{proof}
Theorem~\ref{thm:fredstable} and Corollary~\ref{cor:fredh}
\end{proof}

\rmk\label{rmk:sh} 
All the results of this section continue to hold
in the complex category, i.e. all Hilbert spaces are complex,
all Hilbert manifolds are complex, all maps are complex,
the family  $\{h_\lambda\}_{\lambda\in\Lambda}$ 
in Theorem~\ref{thm:fredstable} is a holomorphic
family of holomorphic morphisms  of complex Fredholm triples, etc.
As a result the map $\Lambda\to X$ in Corollary~\ref{cor:fredstable}
is holomorphic.
\kmr


\section{Proofs of the main theorems}\label{sec:proof}

\begin{proof}[Proof of Theorem~\ref{thm:stable}.]
Assume the unfolding $(\pi_B:Q\to B,S_*,H_B,b_0)$ 
is iinfinitesimally universal. 
Let $\cU,\cU',\cU''$ be the manifolds in~\ref{cU}
and let $\cV,\cV',\cV''$ be the manifolds in~\ref{cV}
for 
$$
P=Q,\qquad A=B,\qquad \pi_A=\pi_B,\qquad
R_*=S_*,\qquad H_A=H_B,
$$
and an appropriate Hardy decomposition
$Q=Q'\cup Q''$. For $a\in A=B$ denote $b_a:=a$,
let $\alpha_a:\Gamma_a\to Q_{b_a}$ be the inclusion
of $\Gamma_a:=Q'_a\cap Q''_a$ into $Q_{b_a}$,
and abbreviate $\beta_a:=H_B\circ\alpha_a:\Gamma_a\to M$. 
Then the morphism
\begin{equation}\label{eq:QQ}
\cU_a\to\cV_a:(\alpha,b)\mapsto\beta:=H_B\circ\alpha
\end{equation}
from the Fredholm quadruple $(\cU_a,\cU'_a,\cU''_a,(\alpha_a,b_a))$
to $(\cV_a,\cV'_a,\cV''_a,\beta_a)$ is exact for $a=a_0=b_0$,
by Theorems~\ref{thm:U} and~\ref{thm:V} (see~\ref{infuniv}).
The same theorems assert that the family~(\ref{eq:QQ})
of morphisms of Fredholm quadruples satisfies the 
requirements of Theorem~\ref{thm:fredstable}. 
Hence it follows from Corollary~\ref{cor:fredstable}
that~(\ref{eq:QQ}) is exact for $a=b$ sufficiently close to
$a_0=b_0$.  Hence, again by Theorems~\ref{thm:U} and~\ref{thm:V},
the unfolding $(\pi_B:Q\to B,S_*,H_B,b)$ 
is infinitesimally universal for $b$ sufficiently close 
to $b_0$. This proves the theorem.
\end{proof}

\begin{proof}[Proof of Theorem~\ref{thm:universal}.]
We proved `only if' in Section~\ref{sec:unfolding}.  To prove `if'
assume that $(\pi_B:Q\to B,S_*,H_B,b_0)$ is an infinitesimally 
universal unfolding.  We prove that it is universal.
Let $(\pi_A:P\to A,R_*,H_A,a_0)$ be another unfolding
of maps and $f_0:P_{a_0}\to Q_{b_0}$ be a fiber isomorphism.
Choose a Hardy decomposition $P=P'\cup P''$
and open subsets $U'$, $U''$, and $U:=U'\cap U''$
of $Q$ as in~\ref{UQ}, \ref{Hardy}, and~\ref{f=id}.
Let $\cU$, $\cU'$, $\cU''$ be as in~\ref{cU}
and $\cV$, $\cV'$, $\cV''$ be as in~\ref{cV}.
Then
$$
(\alpha_0:=f_0|\Gamma_{a_0},b_0)\in\cU'_{a_0}\cap\cU''_{a_0},\qquad
\beta_0:=H_A|\Gamma_{a_0}\in\cV'_{a_0}\cap\cV''_{a_0}.
$$
Since the unfolding $(\pi_B,S_*,H_B,b_0)$ is
infintesimally universal the map 
$$
\cU_{a_0}\to\cV_{a_0}:(\alpha,b)\mapsto \beta:=H_B\circ\alpha
$$
is an exact morphism of Fredholm triples as in~\ref{fredh} 
(see~\ref{infuniv}).  By Theorems~\ref{thm:U} and~\ref{thm:V}
the family of maps 
$$
\cU_a\to\cV_a:(\alpha,b)\mapsto \beta:=H_B\circ\alpha,
$$
parametrized by $a\in A$ satisfies the hypotheses of 
Theorem~\ref{thm:fredstable} (in the complex category).  Moreover, there is 
a holomorphic map 
$$
A\to \cV:a\mapsto(a,\beta_a),\qquad 
\beta_a:=H_A|\Gamma_a\in\cV'_a\cap\cV''_a.
$$
Hence it follows from Corollary~\ref{cor:fredstable} and Remark~\ref{rmk:sh}
that, after shrinking $A$ if necessary, there exists a
unique holomorphic map
\begin{equation}\label{eq:2ndmap}
A\to\cU:a\mapsto(a,\alpha_a,b_a),\qquad 
(\alpha_a,b_a)\in\cU'_a\cap\cU''_a,
\end{equation}
such that
$
\beta_a=H_B\circ\alpha_a
$
for every $a\in A$.  Define $\phi:A\to B$ by $\phi(a):=b_a$, 
for every $a\in A$ let $f_a:P_a\to Q_{b_a}$ be the unique
fiber isomorphism with $f_a|\Gamma_a=\alpha_a$, and
define $\Phi:P\to Q$ by $\Phi|P_a:=f_a$.  Then $\phi$ is holomorphic.
That the restriction of $\Phi$ to $\INT(P')$ is holomorphic
follows from~\cite[Lemma~10.18]{RS}. To prove that the restriction 
of $\Phi$ to $\INT(P'')$ is holomorphic we write it as the composition
$$
\INT(P'')\to A\times\Omega\to \cU''\times\Omega\to Q
$$
where the first map is $\pi_A\times\rho$,
the second map is the product of~(\ref{eq:2ndmap}) 
with the identity, and the third map is the evaluation map 
$(a,f'',z)\mapsto f''(\rho_a^{-1}(z))$. 
All four spaces are complex manifolds and
all three maps are holomorphic. 
The argument is as in Step~3 in the
proof of~\cite[Theorem~5.3]{RS}.
It is important to remember that the complex structure
on the factor $\Omega$ depends on $a\in A$ and is twisted
by $\eta(a,\hat{a})$ as in~(\ref{eq:JP}).
This proves that $\Phi$ is holomorphic on $P\setminus\p P'$.
Since $\Phi$ is continuous, it is holomorphic everywhere.
This proves the theorem.
\end{proof}

\begin{proof}[Proof of Theorem~\ref{thm:exists}.]
Given the work done in Section~\ref{sec:unfolding} it remains to prove `if'
under the assumptions that $(\Sigma,s_{0,*},\nu_0,j_0,v_0)$ 
is a regular stable map and the underlying marked nodal 
Riemann surface $(\Sigma,s_{0,*},\nu_0,j_0)$ is still stable. 
Let $(\pi_A:P\to A,R_*,a_0)$ be a universal unfolding of this 
marked nodal Riemann surface (in the sense of~\cite[Definition 5.1]{RS})
and $w_0:\Sigma\to P_{a_0}$ be a desingularization of the 
central fiber.  Define the holomorphic map $h_0:P_{a_0}\to M$
by $h_0\circ w_0:=v_0$.  Choose a Hardy decomposition
$$
P = P'\cup P'',\qquad
\Gamma_a:=P_a\cap P'\cap P'',
$$
as in~\ref{Hardy}, fix an integer $s+1/2>1$, 
and define  $\cV$, $\cV'$, $\cV''$ as in~\ref{cV}. 
The desingularization $w_0:\Sigma\to P_{a_0}$
induces a decomposition 
$$
\Sigma=\Sigma'\cup\Sigma'',\qquad \Sigma'\cap\Sigma''
=\p\Sigma'=\p\Sigma'',
$$
with $\Sigma':=w_0^{-1}(P')$ and $\Sigma'':=w_0^{-1}(P'')$.
As in the proof of Theorem~\ref{thm:V} the map 
$w_0^{-1}\circ\iota_{a_0}$ is a diffeomorphism from $\Gamma$
in~(\ref{eq:trivialize}) to $\Sigma'\cap\Sigma''$ and, to simplify the notation, 
we assume that $\Gamma=\Sigma'\cap\Sigma''$ so that
$\iota_{a_0}=w_0|\Gamma:\Gamma\to P_{a_0}$.   
The infinitesimally universal unfolding of the stable 
map $(\Sigma,s_{0,*},\nu_0,j_0,v_0)$ is the tuple
$$
(\pi_B:Q\to B,S_*,H_B,b_0)
$$
defined by
\begin{equation}\label{eq:BQS}
\renewcommand{\arraystretch}{1.9}
\begin{array}{l}
B:=\cV'\cap\cV'', \quad
Q:=\bigl\{(p,\beta)\in P\times B\,\big| \,
\beta\in\cV'_{\pi_A(p)}\cap\cV''_{\pi_A(p)}\,\bigr\}, \\
\pi_B(p,\beta):=(\pi_A(p),\beta),\qquad
b_0:=(a_0,\beta_0), \\
S_\si:= \bigl\{(p,\beta)\in Q\,\big|\,p\in R_\si\bigr\},\qquad
H_B(p,\beta):= h_\beta(p),
\end{array}
\end{equation}
where $h_\beta:P_a\to M$ is the unique holomorphic map 
with 
$$
h_\beta|\Gamma_a=\beta.
$$  
As in~\ref{cV}, $\cV$ is a complex Hilbert manifold and by part~(iii)
of Theorem~\ref{thm:V} the sets  $\cV'$ and $\cV''$ are complex 
submanifolds of $\cV$.   By part~(i) of Theorem~\ref{thm:transverse}, 
the submanifolds $\cV'$ and $\cV''$ intersect transversally at 
$(a_0,\beta_0)$ and hence $B=\cV'\cap\cV''$ is a complex
submanifold of $\cV$ (after shrinking $\cV'$ 
and $\cV''$ if necessary).  By Theorem~\ref{thm:V}, 
$B$ has dimension
\begin{equation}\label{eq:dimB}
\dim_\C B = (\sm-3)(1-\sg) + \inner{c_1}{\sd} + \sn.
\end{equation}

We prove that $Q$ is a complex submanifold of $P\times\cV$.
Define
$$
 f:B\to A \qquad\mbox{by}\qquad  f(a,\beta):=a
$$
for $(a,\beta)\in B=\cV'\cap\cV''$.
Then the projection $\pi_B:Q\to B$ is the \jdef{pullback}
of the projection $\pi_A:P\to A$ by the map $f$, i.e. 
$Q$ is the preimage of the diagonal in $A\times A$ under the 
holomorphic map
$$
\pi_A\times f:P\times B\to A\times A
$$
and $\pi_B$ is the restriction of projection on the first factor to $Q$.
The map $\pi_A\times f$ is transverse to the diagonal if and only if 
\begin{equation}\label{eq:TPV}
T_{\pi_A(p)}A = \mathrm{im}\,d\pi_A(p)
+ d\pi_\cV(a,\beta)\left(T_{(a,\beta)}\cV'\cap T_{(a,\beta)}\cV''\right)
\end{equation}
for every $p\in P$ and every $\beta\in\cV'_a\cap\cV''_a$ 
with $a=\pi_A(p)$, where $\pi_\cV:\cV\to A$ denotes the
obvious projection.  Equation~(\ref{eq:TPV}) follows 
immediately from part~(ii) of Theorem~\ref{thm:transverse}. 
Hence $Q$ is a complex submanifold of $P\times\cV$ 
and the projection $\pi_B:Q\to B$ is holomorphic. 
We prove that
the map $\pi_B$ is a nodal family of Riemann surfaces
in Lemma~\ref{le:pullback} below.
The subset $S_\si\subset Q$ is the transverse
intersection of the complex submanifolds 
$R_\si\times\cV$ and $Q$, and hence is a complex 
submanifold of $Q$ (of codimension one).  

We prove that $H_B:Q\to M$ is holomorphic.  
For this we use the Hardy decomposition
$$
Q=Q'\cup Q'',\qquad
Q':=Q\cap(P'\times\cV),\qquad
Q'':=Q\cap(P''\times\cV).
$$
That $H_B$ is holomorphic in the interior of $Q'$ follows from 
Lemma~\ref{le:localmodel}~(iii).  To prove that $H_B$ 
is holomorphic in the interior of $Q''$ write it 
as the composition
$$
\INT(Q'') \to B\times\Omega \to \cV''\times\Omega \to M
$$
where the first map is given by a Hardy trivialization 
$\pi_B\times\rho$, the second by the inclusion $B\to\cV''$,
and the third is the evaluation map
$
((a,\beta),z)\mapsto(h''_\beta(\rho_a^{-1}(z)))
$
where $h''_\beta:P_a''\to M$ is the unique holomorphic map
with $h''_\beta|\Gamma_a=\beta$. 
As in the proof of Theorem~\ref{thm:universal} all four
spaces are complex manifolds and all three maps are holomorphic.
This proves that $H_B$ is holomorphic in $Q\setminus\p Q'$. 
Since $H_B$ is continuous it is holomorphic everywhere. 

We prove that the unfolding $(\pi_B:Q\to B,S_*,H_B,b_0)$
is infinitesimally universal.  Note that 
$
Q_{b_0} = P_{a_0}\times\{\beta_0\}
$
and define $u_0:\Sigma\to Q_{b_0}$ by
$$
u_0(z) := (w_0(z),\beta_0).
$$
Since $h_{\beta_0}\circ w_0=v_0$ we have 
$$
H_B\circ u_0(z) = H_B(w_0(z),\beta_0)
=h_{\beta_0}(w_0(z)) = v_0(z)
$$
for every $z\in\Sigma$. As before we denote by 
$f:B=\cV'\cap\cV''\to A$ the obvious projection
and by $b_0=(a_0,\beta_0)\in B$ the base point.
Then the kernel of the derivative $df(b_0):T_{b_0}B\to T_{a_0}A$ 
is the intersection $T_\beta\cV'_{a_0}\cap T_\beta\cV''_{a_0}$. 
Hence, for $z\in\Sigma$ we have $p:=w_0(z)\in P_{a_0}$, 
$q:=u_0(z)=(w_0(z),\beta_0)\in Q_{b_0}$, and 
$$
\ker d(f\circ\pi_B)(q) = \ker d\pi_A(p)\times 
\left(T_\beta\cV'_{a_0}\cap T_\beta\cV''_{a_0}\right).
$$
The restriction of
$dH_B(q):T_qQ\to T_{v_0(z)}M$
to this space is 
$$
dH_B(u_0(z))(\hat p,\hat\beta) = \hat v(z) + dv_0(z)\hat z
$$
where $\hat z\in T_z\Sigma$ is the unique element with
$dw_0(z)\hat z=\hat p$ and $\hat v\in\Omega^0(\Sigma/\nu,v_0^*TM)$
is the unique vector field along $v_0$ that satisfies the nodal 
condition, belongs to the kernel of $D_{v_0}$, and 
satisfies $\hat v|\Gamma=\hat\beta\circ\iota_{a_0}$.  

We prove that the induced map
\begin{equation}\label{eq:keruv}
dH_B(u_0):\ker D_{u_0}\to\ker\,D_{v_0}
\end{equation}
is bijective. The domain of $\cD_{u_0}$ is the space
$$
\cX_{u_0}:=\left\{(\hat w,\hat b)
\in\Omega^0(\Sigma/\nu,w_0^*TP)\times T_bB\,\bigg|\,
\begin{aligned}
&\hat w(s_{0,\si})\in T_{w_0(s_{0,\si})}R_\si \\
&d\pi_A(w_0)\hat w\equiv df(b_0)\hat b
\end{aligned}
\right\},
$$
the target space can be identified with
$$
\cY_{u_0} = \cY_{w_0} = \left\{\eta\in\Omega^{0,1}(\Sigma,w_0^*TP)\,|\,
d\pi_A(w_0)\eta\equiv 0\right\},
$$
and the operator is given by
$$
D_{u_0}(\hat w,\hat b) := D_{w_0}\hat w.
$$ 
Since the unfolding $(\pi_A,R_*,a_0)$ 
(of marked nodal Riemann surfaces) is universal, the operator  
$$
D_{w_0}:
\cX_{w_0}:=\left\{\hat w\in\Omega^0(\Sigma/\nu,w_0^*TP)\,\bigg|\,
\begin{aligned}
&\hat w(s_{0,\si})\in T_{w_0(s_{0,\si})}R_\si \\
&d\pi_A(w_0)\hat w\equiv\mbox{ constant}
\end{aligned}
\right\}
\to\cY_{u_0}
$$
is bijective.  
It follows that the projection $(\hat w,\hat b)\mapsto\hat b$
is an isomorphism from the kernel of $D_{u_0}$ to the kernel 
of the linear map $df(b_0):T_{b_0}B\to T_{a_0}A$. 
Now recall that $f:B=\cV'\cap\cV''\to A$ denotes the obvious projection.  
Then the kernel of $df(a_0,\beta_0):T_{(a_0,\beta_0)}(\cV'\cap\cV'')\to T_{a_0}A$
is the intersection $T_{\beta_0}\cV'_{a_0}\cap T_{\beta_0}\cV'_{a_0}$
which, by Theorem~\ref{thm:V}~(ii), is isomorphic to the kernel of $D_{v_0}$.
The composite isomorphism 
$$
\ker\,D_{u_0}\to\ker\,df(a_0,\beta_0)\to \ker\,D_{v_0}
$$
is given by $(0,\hat b)\mapsto\hat\beta\mapsto\hat v$
where $\hat b=(0,\hat\beta)$ and $\hat v$ is the unique
element in the kernel of $D_{v_0}$ with 
$\hat v|\Gamma=\hat\beta\circ\iota_{a_0}$. 
This map is precisely~(\ref{eq:keruv}) which is therefore 
an isomorphism. 

Now it follows from Theorem~\ref{thm:transverse}~(ii)
that the nodal family $(\pi_B,S_*,b_0)$ is regular nodal, 
i.e.~the projections of the critical manifolds intersect transversally 
at $b_0$.
Hence, by~\cite[Lemma~12.2]{RS},
the operator $D_{u_0}$ has Fredholm index 
\begin{eqnarray*}
\INDEX_\C(D_{u_0})
&=&
3-3\sg - \sn + \dim_\C B \\
&=&
\sm(1-\sg) + \inner{c_1}{\sd} \\
&=&
\INDEX_\C(D_{v_0}).
\end{eqnarray*}
Here the second equality follows from~(\ref{eq:dimB}). 
Since the kernels are isomorphic it follows that cokernels
of $D_{u_0}$ and $D_{v_0}$ have the same dimensions. 
Moreover, the induced homomorphism
$
dH_B(u_0):\coker D_{u_0}\to\coker\,D_{v_0}
$
is surjective, by Remark~\ref{rmk:regular}, 
and hence is bijective. This completes the proof 
of Theorem~\ref{thm:exists}. 
\end{proof}

\begin{lemma} \label{le:pullback} 
Let
$\pi_A:P\to A$ be a nodal family and $f:B\to A$
be a holomorphic map such that 
$f\times \pi_A:B\times P\to A\times A$
is transverse to the diagonal. Then the 
pullback  $\pi_B:Q\to B$ of $\pi_A$ by $f$
is a nodal family. 
\end{lemma}

\begin{proof} 
The pullback is defined by
$$
Q:=\left\{(b,p)\in B\times P\,|\,\pi_A(p)=f(b)\right\},\qquad
\pi_B(b,p):=b.
$$
The condition that $f\times \pi_A:B\times P\to A\times A$
is transverse to the diagonal implies that $Q$ is a submanifold
of $B\times P$.
We prove that  
\begin{description}
\item[(i)] 
$(b,p)\in Q$ is a regular point of $\pi_B$ 
if $p$ is a regular point of $\pi_A$, and
\item[(ii)] 
$(b,p)\in Q$ is a nodal point of $\pi_B$ 
if $p$ is a nodal point of $\pi_A$.
\end{description}
To prove~(i) assume w.l.o.g.~that $P=\C\times A$ 
so $Q=\C\times\mathrm{graph}(f)$.
Then $\pi_B(b,z,f(b))=b$ so $\pi_B$ is a submersion.

To prove~(ii) assume that w.l.o.g.~that  
$P=\C\times \C\times U$,
$A=\C\times U$,
$\pi_A(x,y,u)=(xy,u)$, and 
$f(b)=(\zeta(b),g(b))\in\C\times U$. 
Then 
$$
  Q=\{(b,x,y,u)\,|\,xy=z=\zeta(b),\;\; u=g(b)\}.
$$
The condition that $f\times\pi_A$ is transverse to the diagonal
at  $(b,x,y,u)\in Q$ is that for all
$
(\hat{z}_1,\hat{u}_1,\hat{z}_2,\hat{u}_2)\in
T_{(z,u)}A\times T_{(z,u)}A=\C\times T_uU\times\C\times T_uU
$
the equations
\begin{eqnarray*}
\hat{z}_1&=&d\zeta(b)\hat{b}+\hat{z}\\
\hat{u}_1&=& dg(b)\hat{b}+\hat{u}\\
\hat{z}_2&=&\hat{x}y+x\hat{y}+\hat{z}\\
\hat{u}_2&=&\hat{v}+\hat{u}
\end{eqnarray*}
have a solution 
$$
\hat{b}\in T_bB,\quad  
(\hat{x},\hat{y},\hat{v})\in T_{(x,y,u)}P=\C^2\times T_uU,\quad
(\hat{z},\hat{u})\in T_aA=\C\times T_uU.
$$
At a nodal point we have $x=y=0$ so transversality implies
that $d\zeta(b)\ne0$. This implies that there is a coordinate system
on $B$ with $\zeta$ as its first element. 
The pullback to $Q$ of the coordinates other than
$\zeta$ together with the functions $x$ and $y$ give the desired nodal
coordinates on $Q$.
This proves~(ii) and the lemma.
\end{proof}

\begin{corollary}\label{cor:transverse-core}
Let $\pi_A:P\to A$ be regular nodal family
and $f:B\to A$ be a holomorphic map which is transverse
to the core $A_0$ of $\pi_A$. Then the hypothesis of
Lemma~\ref{le:pullback} holds,
the pullback $\pi_B:Q\to B$ is regular nodal,
and its core is $B_0:=f^{-1}(A_0)$.
\end{corollary}

\begin{proof}  Denote by 
$
C_1,\dots,C_\sk\subset P
$ 
the components of the singular set of $\pi_A$.
The proof of Lemma~\ref{le:pullback} shows that the
hypothesis that $f\times\pi_A$ is transverse to the diagonal
is equivalent to the hypothesis that $f$ is transverse
to each $\pi_A(C_\si)$. The hypothesis that $\pi_A$ is regular
nodal is that
these projections  $\pi_A(C_\si)$ of the critical manifolds intersect transversally.
Hence $T_aA_0=\bigcap_\si T_a\pi_A(C_\si)$ so $f$ is certainly
transverse to each $\pi_A(C_\si)$ and the hypothesis of
Lemma~\ref{le:pullback} holds.

The hypothesis that $\pi_A$ is regular nodal
implies that in a neighborhood of each point of the core $A_0$ of $\pi_A$ there
are coordinates $z_1,\ldots,z_\sk,u_1,\ldots$
on $A$ such that for each $\si$, $z_i$ together with the remaining
coordinates for the base coordinates of a nodal coordinate system.
In particular, $\pi_A(C_\si)=\{z_\si=0\}$.
The transversality hypothesis implies
that the functions $f^*z_\si$ are independent, i.e.
the sequence $f^*z_1,\ldots,f^*z_\sk$ extends to a coordinate system on $B$.
Now the proof of  Lemma~\ref{le:pullback} shows that 
for each $\si$ a reordering of these coordinates
which puts $f^*z_\si$ first is the base coordinate system of a nodal
coordinate system. The core $B_0$ is then defined by
$f^*z_1=\cdots f^*z_\sk=0$ which shows that $B_0=f^{-1}(A_0)$.
\end{proof}

\dfn\label{def:proper}
Let $(\pi_A:P\to A,R_*,H_A,a_0)$ and $(\pi_B:Q\to B,S_*,H_B,b_0)$
be two unfoldings of type 
$(\sg,\sn,\sd)$.
A sequence of fiber isomorphisms ${f_k:P_{a_k}\to Q_{b_k}}$ 
is said to \jdef{DMG converge} to a fiber isomorphism
$f_0:P_{a_0}\to Q_{b_0}$ if  ${a_k\to a_0}$, $b_k\to b_0$, and 
for every Hardy decomposition $P=P'\cup P''$ as in~\ref{Hardy} 
the sequence 
${f_k\circ\iota_{a_k}:\Gamma\to Q}$ converges to 
${f_0\circ\iota_{a_0}:\Gamma\to Q}$ in  the $\Cinf$ topology.
(DMG convergence of fiber isomorphisms is essentially 
the same as DM convergence in~\cite[Definition~13.7]{RS}.
The only difference is that in the former case we deal with unfoldings
of stable maps whereas in the latter case we deal with unfoldings
of marked nodal Riemann surfaces, i.e.~the two notions of fiber
isomorphism differ.)
\nfd

\begin{lemma}\label{le:proper}
Let $(\pi_A:P\to A,R_*,H_A,a_0)$ and $(\pi_B:Q\to B,S_*,H_B,b_0)$
be two universal unfoldings of type $(\sg,\sn,\sd)$,
$(\Phi,\phi):(P,A)\to(Q,B)$ be the germ of a morphism satisfying
$H_B\circ\Phi=H_A$, $\phi(a_0)=b_0$, and $\Phi_{a_0}=f_0$,
$a_k\in A$ and $b_k\in B$ be two sequences with
 $a_k\to a_0$ and $b_k\to b_0$,
and $f_k:P_{a_k}\to Q_{b_k}$ be a sequence of fiber isomorphisms.
Then the following are equivalent.
\begin{description}
\item[(i)]
The sequence $(a_k,f_k,b_k)$ DMG converges to $(a_0,f_0,b_0)$.
\item[(ii)]
For $k$ sufficiently large we have $\phi(a_k)=b_k$ and $\Phi_{a_k}=f_k$.
\end{description}
\end{lemma}

\begin{proof}
That~(ii) implies~(i) is obvious. We prove that~(i) implies~(ii).
Recall the Hardy decomposition in the definition of
the spaces $\cU$, $\cU'$, $\cU''$ in~\ref{cU} and
$\cV$, $\cV'$, $\cV''$ in~\ref{cV}.  Then
$$
\bigl(a,\Phi_a|\Gamma_a,\phi(a)\bigr) \in \cU'\cap \cU'',\qquad
(a_k,f_k|\Gamma_{a_k},b_k)\in\cU'\cap \cU''
$$
for every $a\in A$ and every sufficiently large $k$,
by DMG convergence.  The sequences 
$(a_k,\Phi_{a_k}|\Gamma_{a_k},\phi(a_k))$
and $(a_k,f_k|\Gamma_{a_k},b_k)$ converge
to the same point $(a_0,f_0|\Gamma_{a_0},b_0)\in\cU'\cap\cU''$.
Moreover, their images under the Fredholm map
$$
\cU'\cap\cU''\to\cV'\cap\cV'':(a,\alpha,b)\mapsto(a,H_B\circ\alpha)
$$
agree because
$$
H_B\circ f_k = H_A|P_{a_k} = H_B\circ\Phi_{a_k}.
$$
Moreover it follows from infinitesimal universality and
Theorems~\ref{thm:U}, \ref{thm:V}, and~\ref{thm:fredstable}
that the map $(a,\alpha,b)\mapsto(a,H_B\circ\alpha)$ 
from $(\cU,\cU',\cU'',(a_0,f_0|\Gamma_0,b_0))$
to $(\cV,\cV',\cV'',(a_0,H_B\circ f_0|\Gamma_0))$
is an exact morphism of Fredholm quadruples 
(see~\ref{fredh}).  Hence 
$(f_k|\Gamma_{a_k},b_k)=(\Phi_{a_k}|\Gamma_{a_k},\phi(a_k))$
for $k$ sufficiently large, by Corollary~\ref{cor:fredh},
and hence also $f_k=\Phi_{a_k}$.
This proves the lemma.
\end{proof}

\begin{proof}[Proof of Theorem~\ref{thm:proper}.]
Let $(\pi:Q\to B,S_*,H)$ be a universal family and denote by 
$(B,\Gamma)$ the associated etale groupoid of~\ref{B-Gamma}.   
We prove that this groupoid is proper.
Thus let $(a_k,f_k,b_k)$ be a sequence in $\Gamma$ such that
$a_k$ converges to $a_0$ and $b_k$ converges to $b_0$.
We must show that there is a fiber isomorphism $f_0:Q_{a_0}\to Q_{b_0}$
such that a suitable subsequence of $f_k$ DMG converges to $f_0$.
To see this we assume first that 
the underlying marked nodal Riemann surface associated to 
a desingularization of $Q_{a_0}$ is stable.  Then the same holds
for $Q_{b_0}$ and we may assume w.l.o.g.~that our universal 
unfolding has the form~(\ref{eq:BQS}) as constructed in the proof
of Theorem~\ref{thm:exists} near $a_0$ and $b _0$.  
It then follows that $(a_k,f_k,b_k)$ induces a sequence 
$(a_k',f_k',b_k')$ of fiber isomorphisms for the underlying 
universal family $(\pi':Q'\to B',S'_*)$  of stable marked nodal 
Riemann surfaces such that $a_k'$ and $b_k'$ converge to 
$a_0'$ and $b_0'$, respectively. By~\cite[Theorem~6.6]{RS}, 
the sequence $f_k'$ DM-converges to a fiber isomorphism 
$f_0':Q'_{a'_0}\to Q'_{b'_0}$.  Since $H_B\circ f_k=H_B|Q_{a_k}$,
we find that $f_0'$ induces a fiber isomorphism
$f_0:Q_{a_0}\to Q_{b_0}$ and it follows from the definitions 
that $f_k$ DMG converges to $f_0$.  
This proves the assertion 
under the stability assumption for the underlying marked nodal
Riemann surface.  If that does not hold, we choose 
an embedding of our universal family into another 
family $(\pi':Q'\to B',S'_*,T'_*,H')$ that is a universal
unfolding of each of its fibers and remains stable after
discarding $H'$.  Then the existence of a
DMG-convergent subsequence follows immediately from 
what we have already proved. 
\end{proof}


\section{The Gromov topology}\label{sec:topology}

In this section we prove that the topology 
on the moduli space of (regular) stable maps that is induced 
by the orbifold structure agrees with the topology used elsewhere
in the literature. To define convergence of a sequence in this topology 
we need to recall the notion of {\em deformation} 
from~\cite[Definition~13.2]{RS}. 

\para\label{suture}
Let $\Sigma$ be a compact oriented surface
and $\gamma\subset\Sigma$ be a disjoint union of
embedded circles. We denote by $\Sigma_\gamma$
the compact surface with boundary which results
by \jdef{cutting open} $\Sigma$ along $\gamma$.
This implies that there  is  a local embedding 
$$
\sigma:\Sigma_\gamma\to\Sigma
$$
which
maps $\INT(\Sigma_\gamma)$ one to one onto $\Sigma\setminus\gamma$
and maps $\p \Sigma_\gamma$ two to one onto $\gamma$.
One might call $\sigma$ the {\em suture map} and $\gamma$ the {\em incision}.
\arap

\dfn \label{deformation}
Let   $(\Sigma',\nu')$ and $(\Sigma,\nu)$ be nodal surfaces.
A smooth map $\phi:\Sigma'\setminus\gamma'\to\Sigma$
is called a $(\nu',\nu)$-\jdef{deformation}
iff  $\gamma'\subset\Sigma'\setminus\bigcup\nu'$ is a disjoint union
of embedded circles such that (where $\sigma:\Sigma'_{\gamma'}\to\Sigma'$
is the suture map just defined) we have
\begin{itemize}
\item
$
\phi_*\nu':=\bigl\{  \{\phi(y'_1),\phi(y'_2)\}\,|\,
\{y'_1,y'_2\}\in\nu'\bigr\} \subset\nu.
$
\item
$\phi$ is a diffeomorphism from $\Sigma'\setminus \gamma'$
onto $\Sigma\setminus \gamma$, where
$\gamma:=\bigcup(\nu\setminus\phi_*\nu')$.
\item
$\phi\circ\sigma|\INT(\Sigma'_{\gamma'})$
extends to a continuous surjective map
$\Sigma'_{\gamma'}\to\Sigma$ such that
the preimage of each nodal point in $\gamma$
is a component of $\p\Sigma'_{\gamma'}$ and two boundary components
which map under $\sigma$ to the same component of $\gamma'$
map to a nodal pair $\{x,y\}\in\gamma$.
\end{itemize}
Each component of $\gamma'$ is called a {\em vanishing cycle} 
of the deformation $\phi$. A sequence 
$\phi_k:(\Sigma_k\setminus\gamma_k,\nu_k)\to(\Sigma,\nu)$
of $(\nu_k,\nu)$-deformations
is called \jdef{monotypic} if $(\phi_k)_*\nu_k$
is independent of $k$.
\nfd

\dfn\label{def:convergence} 
Let $M$ be a complex manifold. 
A sequence $(\Sigma_k,s_{k,*},\nu_k,j_k,v_k)$ of configurations
in $M$ of type 
$(\sg,\sn,\sd)$
is said to \jdef{converge monotypically} to 
a  configuration $(\Sigma,s_*,\nu,j,v)$ of type 
$(\sg,\sn,\sd)$
iff there is a monotypic sequence 
$\phi_k:\Sigma_k\setminus\gamma_k \to\Sigma\setminus\gamma$
of $(\nu_k,\nu)$-deformations satisfying the following conditions.
\begin{description}
\item[(Marked points)]
For $\si=1,\ldots,\sn$ the sequence $\phi_k(s_{k,\si})$
converges to $s_\si$ in $\Sigma$.
\item[(Complex structure)]
The sequence $(\phi_k)_*j_k$ of complex structures
on $\Sigma\setminus \gamma$ converges to 
$j|(\Sigma\setminus \gamma)$ in the $\Cinf$ topology.
\item[(Map)]
The sequence $(\phi_k)_*v_k:=v_k\circ\phi_k^{-1}$
converges to $v|(\Sigma\setminus \gamma)$ 
in the $\Cinf$ topology on $\Cinf(\Sigma\setminus\gamma,M)$.
\item[(Energy)]
For some (and hence every) pair of Riemannian metrics on 
$\Sigma$ and~$M$ we have 
$$
\lim_{\eps\to 0}\lim_{k\to\infty}\int_{B_\eps(\gamma)}\Abs{d(v_k\circ\phi_k^{-1})}^2
= 0,
$$
where $B_\eps(\gamma)\subset\Sigma$ denotes the $\eps$-neighborhood
of $\gamma\subset\cup\nu$. 
\end{description}
The sequence $(\Sigma_k,s_{k,*},\nu_k,j_k,v_k)$ is said to
\jdef{Gromov converge} to $(\Sigma,j,s,\nu,v)$ if,
after discarding finitely many terms, it is the disjoint union
of finitely many sequences which converge 
monotypically to  $(\Sigma,s,\nu,j,v)$.
\nfd

\begin{figure}[htp] 

\centering  

\includegraphics[scale=0.6]{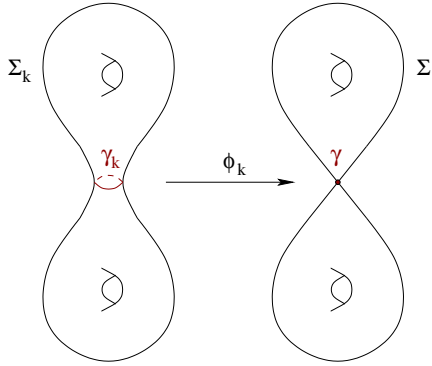} 

\caption{{Gromov convergence.}}\label{fig:gromov} 

\end{figure} 

\begin{theorem}\label{thm:gromovConvergence}
Let $(\Sigma,s_*,\nu,j,v)$ be a stable map, 
$(\pi:Q\to B,S_*,H,b_0)$ be a universal  unfolding, 
${u_0:\Sigma\to Q_{b_0}}$ be a desingularization with induced 
structures $s_*$, $\nu$, $j$, and $v$ on $\Sigma$, and
$(\Sigma_k,s_{k,*},\nu_k,j_k,v_k)$  be a sequence of stable maps.
Then the following are equivalent.
\begin{description}
\item[(i)] 
The sequence $(\Sigma_k,s_{k,*},\nu_k,j_k,v_k)$ 
Gromov converges to $(\Sigma,s_*,\nu,j,v)$.
\item[(ii)] 
After discarding finitely many terms, there exist $b_k\in B$ 
and desingularizations ${u_k:\Sigma_k\to Q_{b_k}}$ inducing
$s_{k,*}$, $\nu_k$, $j_k$,  $v_k$ such that $b_k$ converges to~$b_0$.
\end{description}
If~(i) holds with a sequence of deformations 
$\phi_k:\Sigma\setminus\gamma_k\to\Sigma$
then the sequence $u_k$ in~(ii) can be chosen such that 
$u_k(\gamma_k)$ converges to the nodal set in $Q_{b_0}$
and $u_k\circ\phi_k^{-1}:\Sigma\setminus\cup\nu$ converges
to $u_0|(\Sigma\setminus\cup\nu)$ in the $\Cinf$ topology. 
\end{theorem}

\begin{proof} 
We prove~(ii) implies~(i). 
Let $u:\Sigma\to Q_{b_0}$ be a desingularization.
Assume that $b_k$ converges to $b$ and that
$u_k:\Sigma_k\to Q_{b_k}$ is a sequence of desingularizations
inducing $(s_{k,*}$, $\nu_k$, $j_k$,  $v_k$).
As in the proof of~\cite[Theorem~13.6]{RS}
there are maps $\psi_b:Q_b\to Q_{b_0}$ and deformations
$\phi_k:\Sigma_k\setminus \gamma_k\to\Sigma$
such that $\psi_b$ agrees with a smooth trivialization away from
the nodal set, $\psi_{b_0}$ is the identity,  and
$$
u\circ\phi_k = \psi_{b_k}\circ u_k:
\Sigma_k\setminus \gamma_k\to Q_{b_0}.
$$
Assume w.l.o.g. that the sequence $\phi_k$ is monotypic
so that there is a subset $\gamma\subset\cup\nu$ such
that $\phi_k:\Sigma_k\setminus \gamma_k\to\Sigma\setminus\gamma$
is a diffeomorphism.
As in~\cite{RS}  the sequence $\phi_k(s_{k,\si})$
converges to $s_\si$ in $\Sigma$ and
the sequence $(\phi_k)_*j_k$ of complex structures
on $\Sigma\setminus \gamma$
converges to $j|(\Sigma\setminus \gamma)$ in the $\Cinf$ topology.
Now
$
\psi_{b_k}^{-1}\circ u_0=  u_k\circ\phi_k^{-1}
$
so
$$
H\circ\psi_{b_k}^{-1}\circ u_0
=  H\circ u_k\circ\phi_k^{-1}=v_k\circ\phi_k^{-1}.
$$
Since $\psi_{b_0}$ is the identity the left hand side
(and hence also $(\phi_k)_*v_k=v_k\circ\phi_k^{-1}$) 
converges to $v_0|(\Sigma_0\setminus \gamma)$ 
in the $\Cinf$ topology on $\Cinf(\Sigma\setminus\gamma,M)$.

We prove~(i) implies~(ii) under the additional hypothesis that
the marked nodal Riemann surface $(\Sigma,s_*,\nu,j)$
is stable. By the uniqueness of universal unfoldings
we may asssume that $(\pi,S_*,H,b_0)$ is given by~(\ref{eq:BQS}).
By assumption, the sequence $(\Sigma_k,s_{k,*},\nu_k,j_k)$ 
obtained by discarding the maps $v_k$ consists 
of stable marked nodal Riemann surfaces and it 
DM-converges to $(\Sigma,s_*,\nu,j)$ as in~\cite[Definition~13.3]{RS}.
Hence Theorem~13.6 in~\cite{RS} asserts that there exists
a sequence $a_k\in A$ converging to $a_0$ and, 
for sufficiently large $k$, desingularizations
$w_k:\Sigma_k\to P_{a_k}$ inducing the structures
$s_{k,*}$, $\nu_k$, $j_k$ on $\Sigma_k$. 
By~\cite[Remark~13.9]{RS}, the desingularizations 
$w_k$ can be chosen such that the sequence
$$
w_k\circ\phi_k^{-1}:\Sigma\setminus\cup\nu\to P
$$
converges to $w_0$ in the $\Cinf$ topology. 
Define ${h_k:P_{a_k}\to M}$ 
and ${h_0:P_{a_0}\to M}$~by
$$
h_k\circ w_k:=v_k,\qquad h_0\circ w_0:=v_0.
$$
Since $w_k\circ\phi_k^{-1}$ converges to $w_0$,
the sequence $\phi_k\circ w_k^{-1}\circ\rho_{a_k}^{-1}$ 
(with $\rho$ as in~\ref{hardyTrivialization})
converges to $w_0^{-1}$ in the $\Cinf$ topology 
on $\Omega=P_{a_0}''$. This implies that the sequence
$$
h_k\circ\rho_{a_k}^{-1} = (v_k\circ\phi_k^{-1})
\circ(\phi_k\circ w_k^{-1}\circ\rho_{a_k}^{-1})
$$
converges to $v_0\circ w_0^{-1}=h_0$ in the 
$\Cinf$ topology on $\Omega$.  
By definition of $\cV''$, this implies
$$
b_k := (a_k,\beta_k)\in\cV'\cap\cV''=B,\qquad
\beta_k:=h_k|\Gamma_{a_k}\in\cV_{a_k}''
$$
for $k$ sufficiently large.  Here we have also used the fact 
that $h_k|P'_{a_k}$ takes values in $\cV'$ for large $k$, by 
the {\it (Energy)} axiom and the standard compactness arguments 
for pseudoholomorphic curves (see~\cite[Chapter~4]{MS}).  
Since $a_k$ converges to $a_0$ and 
$\beta_k\circ\rho_{a_k}^{-1}|\p\Omega$ converges
to $\beta_0:=h_0|\Gamma_{a_0}=h_0|\p\Omega$,
we deduce that $b_k$ converges to $b_0:=(a_0,\beta_0)$. 
Thus we have proved that~(i) implies~(ii) under the assumption 
that the marked nodal Riemann surface $(\Sigma,s_*,\nu,j)$
is stable.

We prove~(i) implies~(ii) in general.
Suppose the sequence $(\Sigma_k,s_{k,*},\nu_k,j_k,v_k)$ 
Gromov converges to $(\Sigma,s_*,\nu,j,v)$ and the underlying
marked nodal Riemann surface $(\Sigma,s_*,\nu,j)$ is not 
stable.  Then we can add marked points to $\Sigma_k$
and $\Sigma$ such that the resulting sequence still 
Gromov converges and the augmented marked nodal Riemann
surface $(\Sigma,s_*,t_*,\nu,j)$ is stable. By what we have 
already proved, the augmented sequence 
$(\Sigma_k,s_{k,*},t_{k,*},\nu_k,j_k,v_k)$ satisfies~(ii). 
Let $(\pi_A:P\to A,R_*,T_*,H_A,a_0)$ be a universal
unfolding of the augmented stable map. Removing the 
additional sections $T_*$ results in an unfolding that is no longer 
universal but, by definition of universal, admits a morphism to 
$(\pi:Q\to B,S_*,H,b_0)$. Hence the original sequence 
$(\Sigma_k,s_{k,*},\nu_k,j_k,v_k)$ also satisfies~(ii).
This proves the theorem. 
\end{proof}


\section{Concluding remarks}

It would be interesting to know to what extent the techniques developed
in this paper extend to the nonintegrable case. Since the linearized
Cauchy--Riemann operators $D_v$ are not complex linear in this case
the resulting moduli space will at best be a smooth (not a complex)
orbifold.  In the definition of a universal unfolding we can 
at most expect the existence of a smooth morphism 
${(\Phi,\phi):\pi_A\to\pi_B}$.  An analogue of the universal unfolding
theorem (Theorem~\ref{thm:universal}) for the nonintegrable case
will depend on an answer to the following question. 

Given an almost complex structure $J$ on $\R^{2\sm}$
and a complex number $z\in\INT(\D)$ define the set
$$
\cN_z:= \left\{(\xi,\eta,z)\in 
H^s(S^1,\R^{2\sm})^2\,\Bigg|\,
\begin{array}{l}
\exists\mbox{ a $J$-holomorphic map }  \\
v:N_z\to\R^{2\sm}\mbox{ in } H^{s+1/2} \\
\mbox{s.t. }\xi=v\circ\iota_{1,z},\; 
\eta=v\circ\iota_{2,z}
\end{array}
\right\}
$$
where $N_z$ is as in~\ref{standardNode}. 
It is easy to prove that this set is a smooth submanifold 
of $H^s(S^1,\R^{2\sm})\times H^s(S^1,\R^{2\sm})$ 
for every $z$.  A natural question to ask is if the disjoint union
$$
\cN := \bigcup_{z\in\INT(\D)}\{z\}\times\cN_z
$$
is a smooth submanifold of 
$\INT(\D)\times H^s(S^1,\R^{2\sm})\times H^s(S^1,\R^{2\sm})$.
In Lemma~\ref{le:localmodel} this was proved in the integrable 
case. However, we have examples of finite dimensional analogues 
where this fails.  On the other hand, we expect that the Hadamard
proof of the unstable manifold theorem carries over
to the infinite dimensional setting and shows that the $\cN_z$
form a continuous family of smooth submanifolds. This would
give an alternative approach to the gluing theorem for 
pseudoholomorphic curves.  Moreover, one could then carry 
over the techniques of this paper to prove that, in the 
almost complex case, the regular stable maps form a $C^0$ orbifold.


\end{document}